\documentclass{amsart}
\allowdisplaybreaks
\usepackage{dsfont}
\usepackage[usenames,dvipsnames,svgnames,table]{xcolor}
\usepackage[utf8]{inputenc}
\usepackage[T1]{fontenc}
\usepackage{ tipa }
\usepackage{helvet}
\usepackage{color}
\usepackage{mathrsfs}  
\usepackage{stmaryrd}  
\usepackage{amsmath,amsxtra,amsthm,amssymb,xr,enumerate,fullpage,comment, graphicx, mathtools}
\usepackage[all]{xy}
\usepackage[bbgreekl]{mathbbol}
\usepackage{bbm}
\usepackage{bm}
\usepackage{tikz-cd}

\makeatletter
\newlength{\doublefracgap}
\DeclareRobustCommand{\doublefrac}[2]{%
  \mathinner{\mathpalette\doublefrac@{{#1}{#2}}}%
}
\newcommand{\doublefrac@}[2]{\doublefrac@@#1#2}
\newcommand{\doublefrac@@}[3]{%
  \ooalign{%
    \raisebox{\doublefracgap}{$\m@th#1\frac{#2}{\phantom{#3}}$}\cr
    \raisebox{-\doublefracgap}{$\m@th#1\frac{\phantom{#2}}{#3}$}\cr
  }%
}

\makeatother

\DeclareMathSymbol{\invques}{\mathord}{operators}{`>}
\DeclareUnicodeCharacter{00BF}{\tmquestiondown}
\DeclareRobustCommand{\tmquestiondown}{%
  \ifmmode\invques\else\textquestiondown\fi
}

\usepackage{verbatim}
\numberwithin{equation}{section}

\makeatletter
\newcommand{\mylabel}[2]{#2\def\@currentlabel{#2}\label{#1}}
\makeatother

\newtheorem{theorem}{Theorem}[section]
\newtheorem{lemma}[theorem]{Lemma}

\newtheorem{proposition}[theorem]{Proposition}
\newtheorem{conjecture}[theorem]{Conjecture}
\newtheorem{corollary}[theorem]{Corollary}
\newtheorem{defn}[theorem]{Definition}

\newtheorem{remark}[theorem]{Remark}

\newtheorem{assumption}[theorem]{Assumption}

\newtheorem*{conjecture*}{Conjecture}

\newtheorem{lthm}{Theorem}

\newcommand{\fb}{\mathfrak{b}}
\newcommand{\fc}{\mathfrak{c}}
\newcommand{\fd}{\mathfrak{d}}
\newcommand{\bb}{\mathbb}
\newcommand{\frk}{\mathfrak}
\newcommand{\pr}{\mathrm{pr}}

\newcommand{\Tr}{\operatorname{Tr}}
\newcommand{\Gal}{\operatorname{Gal}}
\newcommand{\Fil}{\operatorname{Fil}}

\newcommand{\DD}{\mathbb{D}}

\newcommand{\CC}{\mathbb{C}}

\newcommand{\NN}{\mathbb{N}}
\newcommand{\QQ}{\mathbb{Q}}
\newcommand{\Qp}{{\mathbb{Q}_p}}
\newcommand{\Zp}{\mathbb{Z}_p}
\newcommand{\ZZ}{\mathbb{Z}}
\renewcommand{\AA}{\mathbb{A}}
\newcommand{\FF}{\mathbb{F}}

\newcommand{\ord}{\mathrm{ord}}

\newcommand{\fp}{\mathfrak{p}}
\newcommand{\fq}{\mathfrak{q}}

\newcommand{\vp}{\varphi}
\newcommand{\cL}{\mathcal{L}}
\newcommand{\cH}{\mathcal{H}}
\newcommand{\cO}{\mathcal{O}}
\newcommand{\Iw}{\mathrm{Iw}}

\newcommand{\GL}{\mathrm{GL}}

\newcommand{\col}{\mathrm{Col}}
\newcommand{\image}{\mathrm{Im}}
\newcommand{\cyc}{\textup{cyc}}

\newcommand{\ff}{\mathfrak{f}}

\newcommand{\Hom}{\mathrm{Hom}}
\newcommand{\Sel}{\mathrm{Sel}}
\newcommand{\Char}{\mathrm{char}}
\newcommand{\Ind}{\mathrm{Ind}}
\newcommand{\ac}{\textup{ac}}
\newcommand{\LL}{\Lambda}
\newcommand{\TT}{\mathbb{T}}

\newcommand{\h}{\textup{\bf h}}

\newcommand{\lra}{\longrightarrow}

\newcommand{\res}{\textup{res}}

\newcommand{\Bf}{\mathbf{f}}
\newcommand{\BF}{\mathbf{F}}
\newcommand{\Bg}{\mathbf{g}}
\newcommand{\ur}{\textup{ur}}

\newcommand{\cP}{\mathcal{P}}

\newcommand{\Dcris}{\mathbb{D}_{\rm cris}}

\newcommand{\adj}{\mathrm{adj}}

\newcommand{\cor}{\mathrm{cor}}

\newcommand{\id}{\mathrm{id}}

\newcommand{\etale}{\textup{\'et}}
\newcommand{\bk}{(\bm{\kappa}_{\ac}^{-1})}
\newcommand{\bS}{\mathbf{S}}
\newcommand{\Bh}{\mathbf{h}}
\newcommand{\BI}{\mathbf{I}}

\newcommand{\cC}{\mathcal{C}}

\newcommand{\DdR}{\DD_{\mathrm{dR}}}

\newcommand{\green}[1]{\textcolor{ForestGreen}{#1}}

\newcommand{\m}{\mathfrak{m}}

\newcommand{\cD}{\mathcal{D}}

\newcommand{\cG}{\mathcal{G}}

\newcommand{\sF}{\mathscr{F}}
\newcommand{\sL}{\mathscr{L}}

\newcommand{\cR}{\mathcal{R}}

\newcommand{\fe}{\mathfrak{e}}
\newcommand{\sW}{\mathscr{W}}

\newcommand{\fa}{\mathfrak{a}}

\newcommand{\wt}{{\rm wt}}
\newcommand{\Fr}{{\rm Fr}}

\newcommand{\CH}{{\rm CH}}
\newcommand{\End}{{\rm End}}

\newcommand{\hh}{\h}
\newcommand{\Q}{\QQ}
\newcommand{\Z}{\ZZ}

\usepackage{hyperref}

\definecolor{pinegreen}{rgb}{0.0, 0.47, 0.44}

 \definecolor{pAlgae}{RGB}{87,115,135}
\definecolor{airforceblue}{rgb}{0.36, 0.54, 0.66}
	\definecolor{bondiblue}{rgb}{0.0, 0.58, 0.71}
\definecolor{britishracinggreen}{rgb}{0.0, 0.26, 0.15}
\definecolor{camouflagegreen}{rgb}{0.47, 0.53, 0.42}
\definecolor{darkcyan}{rgb}{0.0, 0.55, 0.55}

\hypersetup{
    colorlinks = true,
    linkcolor=blue,
     filecolor=blue,
     citecolor = darkcyan,      
     urlcolor=cyan,
    linkbordercolor = {white},
}


\begin{document}

\author{Ra\'ul Alonso}
\address{Ra\'ul Alonso\newline UCD School of Mathematics and Statistics\\ University College Dublin\\ Belfield\\Dublin 4\\Ireland}
\email{raul.alonsorodriguez@ucd.ie}

\author{K\^az\i m B\"uy\"ukboduk}
\address{K\^az\i m B\"uy\"ukboduk\newline UCD School of Mathematics and Statistics\\ University College Dublin\\ Belfield\\Dublin 4\\Ireland}
\email{kazim.buyukboduk@ucd.ie}

\author{Antonio Cauchi}
\address{Antonio Cauchi\newline UCD School of Mathematics and Statistics\\ University College Dublin\\ Belfield\\Dublin 4\\Ireland}
\email{antonio.cauchi@ucd.ie}

\author{Antonio Lei}
\address{Antonio Lei\newline Department of Mathematics and Statistics\\University of Ottawa\\
150 Louis-Pasteur Pvt\\
Ottawa, ON\\
Canada K1N 6N5}
\email{antonio.lei@uottawa.ca}

\title{{T}\lowercase{rito}-\lowercase{non-ordinary} I\lowercase{wasawa theory of diagonal cycles}}

\begin{abstract}
Our goal in this paper is to introduce and study the Euler system of signed diagonal cycles associated with a \emph{trito-non-ordinary} triple product of the form $f^B \times g^B \times \h^B$, where $f^B$ (resp. $g^B$) is a $p$-ordinary (resp. non-ordinary) eigenform on an indefinite quaternion algebra $B_{/\QQ}$ of weight $2$, and $\h^B$ is a primitive Hida ($p$-ordinary) family. When $B=\mathrm{M}_2(\bb{Q})$ is split and $\h=\h^B$ has CM by an imaginary quadratic field, this allows us to develop the signed anticyclotomic Iwasawa theory for the base change ${\rm BC}_{K/\QQ}(\pi_f)\times {\rm BC}_{K/\QQ}(\pi_g)\times \psi$, where $\psi$ is a Hecke character of $K$. We formulate a signed Perrin-Riou-style Iwasawa main conjecture in this setting, and obtain a result on one inclusion in this conjecture. Our methods also allow us to extend Hsieh's construction of the balanced triple-product $p$-adic $L$-function to the trito-non-ordinary scenario, and to define its signed counterparts.
\end{abstract}

\maketitle

\tableofcontents

\section{Introduction}
Throughout this article, $p$ is a fixed odd prime number. Once and for all, we fix embeddings $\iota:\overline{\bb{Q}}\hookrightarrow \bb{C}$ and $\iota_p:\overline{\bb{Q}}\hookrightarrow \overline{\bb{Q}}_p$.

The construction of Gross--Kudla--Schoen diagonal cycles in the product of three Kuga--Sato varieties was initially studied in \cite{GrossSchoen-cycle,grosskudla}. Darmon and Rotger later incorporated these cycles into the Euler system framework and demonstrated that they satisfy a $p$-adic Gross--Zagier formula via the $p$-adic Abel--Jacobi map and established a connection of these cycles with the Garrett--Rankin $p$-adic $L$-function \cite{darmonrotger14,darmonrotger17}. This development has ignited extensive research into diagonal cycles and their applications in Iwasawa theory. Important progress was made in \cite{asterisque}, where the authors established the variational properties of these cycles in families of modular forms, and proved explicit reciprocity laws linking diagonal cycles to three-variable $p$-adic $L$-functions.

In a different direction, the construction of families of diagonal cycles has been adapted to define anticyclotomic Euler systems for the product of two elliptic modular forms in \cite{ACR25,ACR}, building on the framework developed in \cite{JNS}. These works focus exclusively on the case where all three factors involved in the (families of) triple products are $p$-ordinary.

In the present article, we construct unbounded (resp. bounded signed) anticyclotomic Euler systems of diagonal cycles for the twisted base change ${\rm BC}_{K/\QQ}(\pi_f)\times {\rm BC}_{K/\QQ}(\pi_g)\times \psi$, where $f$ (resp. $g$) is a $p$-ordinary (resp. $p$-non-ordinary) eigenform, $K$ is an imaginary quadratic field, and $\psi$ is a Hecke character of $K$ (see Theorem~\ref{thmB} and Theorem~\ref{thmC} below).  Along the way, we also construct unbounded (resp. bounded signed) balanced triple-product $p$-adic $L$-functions (in a single variable; see Theorem~\ref{thmE} below), extending the work of Hsieh~\cite{hsiehpadicbalanced} to the trito-non-ordinary\footnote{From the Greek prefix \textit{trito-}, meaning ``one-third''.} setting, which we believe is of independent interest.

The main arithmetic applications of these constructions are towards (signed) Iwasawa main conjectures for ${\rm BC}_{K/\QQ}(\pi_f)\times {\rm BC}_{K/\QQ}(\pi_g)\times \psi$, which we formulate in \S\ref{subsubsection_125_2025_11_02}. Employing the Euler system machinery of Jetchev--Nekov\'a\v{r}--Skinner~\cite{JNS}, we prove a containment in this conjecture (cf. Theorem~\ref{thmD}). We remark that, assuming that the Abel--Jacobi map ${\rm cl}_1$ is injective, the non-triviality of this containment follows from \cite{YZZ} (proof of the Gross--Kudla conjecture), provided that we have $L'({\rm BC}_{K/\QQ}(\pi_f)\times {\rm BC}_{K/\QQ}(\pi_g)\times \psi,2)\neq 0$ for some Hecke character $\psi$ of $K$ whose theta-series is a weight-2 specialisation of the CM Hida family $\h$ relevant to our discussion (see \S\ref{subsec_CM_Hida_fam}). 

To explain our main results, we first introduce some notation. 

\subsection{Diagonal cycles for towers of threefolds}

Let $N=N^+N^-$ be a positive integer coprime to $p$, where $N^+$ and $N^-$ are coprime and $N^-$ is square-free. Let $B$ be the indefinite quaternion algebra over $\QQ$ of discriminant $N^-$. We also put $E = \QQ^{\oplus 3}$ and $B_E = B^{\oplus 3}$.  Once and for all fix an isomorphism 
 \[
     \Psi=\prod_{q\nmid N^-}\Psi_q\,:\quad B(\widehat\QQ^{(N^-)})\simeq M_2(\widehat\QQ^{(N^-)}).
 \] 
We fix a maximal order $\mathcal{O}_{B}$ such that $\Psi_q(\mathcal{O}_{B} \otimes \ZZ_q) = M_2(\ZZ_q)$ for all $q \nmid \infty N^{-}$. We then let $U_0(N)$ be the open compact subgroup of $\widehat{B}^{\times}$ defined by the Eichler order $R_N$ of level $N^+$ in $\mathcal{O}_{B}$  and let 
$$U_1( N) = \{ g \in U_0(N)\,:\, \Psi_q(g_q) \equiv \left( \begin{smallmatrix}
    \star & \star \\ 0 & 1
\end{smallmatrix} \right)\, (\text{mod } N \ZZ_q ) \text{ for } q| N^+ \}. $$

For a natural number $n$, we put
 \begin{align*}
     U_n &:= \{ g \in U_1(N) \,:\, \Psi_p(g_p) \equiv \left( \begin{smallmatrix} a & \star  \\  & a \end{smallmatrix}\right)\, (\text{mod } p^n)\, ,\,a \in\Zp^\times\}\,, \\
     U_{1,n}& : = \{ g \in  U_1(N)  \,:\, \Psi_p(g_p) \equiv \left( \begin{smallmatrix} \star &  \star \\  & 1 \end{smallmatrix}\right)\, (\text{mod } p^n)\}\,.     
  \end{align*}
  We write $S_n$ for the Shimura curve given by $U_n$.
 Let us denote by $\iota: B \hookrightarrow B_E$ the diagonal embedding, and let $\mathcal{U}_1(N) : = U_1(N)^{\times 3} \subseteq \widehat{B}_E^{\times}$. For each $\mathbf{n}=(n_1,n_2,n_3) \in \mathbb{Z}_{\geq 0}^3$, we define
\begin{align*}
   \mathcal{U}_{\mathbf{n}} &:=\left \{ g \in  \mathcal{U}_{1}(N)\,:\, g \equiv \left(\left( \begin{smallmatrix}  a & \star \\ 0 & a \end{smallmatrix}\right) \mod p^{n_1},\left( \begin{smallmatrix}  a & \star \\ 0 & a \end{smallmatrix}\right) \mod p^{n_2},\left( \begin{smallmatrix} a & \star \\ 0 & a \end{smallmatrix}\right) \mod p^{n_3}\right)\, ,\,a \in\Zp^\times\right\}.
\end{align*}
The corresponding Shimura threefold is denoted by $\mathbf{X}_{\mathcal{U}_{\mathbf{n}}}$. Via a twisted diagonal embedding of the Shimura curve $S_n$ where $n=\max\{n_1,n_2,n_3\}$, we define the diagonal cycle $\Delta_{\mathcal{U}_{\mathbf{n}} } \in {\rm CH}^2(\mathbf{X}_{\mathcal{U}_{\mathbf{n}}})$ (see Definition~\ref{def_2025_07_03_1206}). The general machinery of Loeffler given in \cite{LoefflerSpherical} implies that these classes satisfy a norm relation given by the Hecke operator:
 \[ {\rm pr}^{\mathbf{n}+\mathbf{1}}_{\mathbf{n}}(\Delta_{\mathcal{U}_{\mathbf{n+1}}}) = \mathbb{U}_p' \, \Delta_{\mathcal{U}_{\mathbf{n}}}.\]

\subsection{Main results}
Let $f$ and $g$ be weight-two forms on $B$ of level $U_1(N)$. We assume that $f$ is $p$-ordinary and $g$ is $p$-non-ordinary. 
\subsubsection{} 
\label{subsubsec_2025_11_02_0946}
Our primary goal in the present paper is to build an anticyclotomic Euler system for the Rankin--Selberg product of $f$ and $g$ using the classes $\Delta_{\mathcal{U}_{\mathbf{n}} }$. In particular, we show that there exist signed integral classes, in the spirit of the constructions in \cite{BFSuper,BBL1, BL-MSMF,BBL2,sprung16,CCSS}. To do so, we first prove that the classes naturally defined by $\Delta_{\mathcal{U}_{\mathbf{n}} }$ satisfy a three-term norm relation. More precisely, starting off with $\Delta_{\mathcal{U}_{\mathbf{n}} }$,  we define a class in Galois cohomology 
$ \Delta_{\mathbf{n}}^{\etale} \in H^1(\Q,  H^3_{\etale} (\mathbf{X}_{\mathbf{n}}\times_{\QQ} \overline{\QQ}, \ZZ_p(2)))$, where $\mathbf{X}_{\mathbf{n}}$ denotes the Shimura threefold of level $U_{n_1} \times U_{n_2} \times U_{n_3}$, which in turn gives rise to
\[
\Delta_{n}^{\etale}(f,g),\ \widetilde{\Delta}_{n}^{\etale}(f,g)\in  H^1(\QQ, T_f^*\otimes T_g^* \otimes e_{\mathrm{U}_{p}'} H^1_{\etale} (X(U_{n})\times_{\QQ}
\overline{\QQ}, \ZZ_p))\,,
\]
where $T_f^*$ and $T_g^*$ are $G_\QQ$-representations attached to $f$ and $g$, realized as quotients of $H^1_{\etale} (X(U_{1})\times_{\QQ}
\overline{\QQ}, \ZZ_p)$ and $H^1_{\etale} (X(U_{0})\times_{\QQ}
\overline{\QQ}, \ZZ_p)$, respectively.
Our first main result is the following norm relations of these classes.

\begin{lthm}[Corollary~\ref{cor_2025_03_22_0955}]\label{thmA}
        For any positive integer $n$, we have
    \begin{align*}
        ({\rm id}, {\rm id},\pi_1)_*\,\Delta_{n+1}^{\etale}(f, g)&=a_p(g) \Delta_{n}^{\etale}(f,g)-\chi_g(p)\widetilde{\Delta}_{n}^{\etale}(f,g)\,,\\
        ({\rm id}, {\rm id},\pi_1)_*\,\widetilde\Delta_{n+1}^{\etale}(f, g)&=p\cdot \Delta_{n} ^{\etale}(f,g)\,,
    \end{align*}
       where $\pi_1$ denotes the natural projection map $X(U_{n+1})\to X(U_n)$.
\end{lthm}

\subsubsection{} We assume until \S\ref{subsubsec_2025_10_31_1451} that $B=\mathrm{M}_2(\bb{Q})$. We fix $K$ to be an imaginary quadratic field where $p$ splits. We write $(p)=\frk{p}\overline{\frk{p}}$, where $\frk{p}$ is the prime of $K$ determined by our fixed embedding $\iota_p:\overline{\bb{Q}}\hookrightarrow \overline{\bb{Q}}_p$. We assume that $p$ does not divide the class number of $K$. Let $\Lambda=\mathcal{O}\llbracket1+p\bb{Z}_p\rrbracket$, where $\cO$ is the ring of integers of a finite extension of $\Qp$ that contains all the coefficients of $f$ and $g$. Let $\Gamma_{\ac}$ be the Galois group of the anticyclotomic $\bb{Z}_p$-extension of $K$. We can identify $\Gamma_{\ac}$  with the anti-diagonal in $(1+p\bb{Z}_p)\times (1+p\bb{Z}_p)\cong \mathcal{O}_{K,\frk{p}}^{(1)}\times \mathcal{O}_{K,\overline{\frk{p}}}^{(1)}$ via the Artin map. Let 
$\Gamma_{\ac}\xrightarrow{\,\kappa_{\ac}\,} \bb{Z}_p^\times$ be the character defined by mapping the element $((1+p)^{-1},(1+p))$ to $1+p$, and $\Gamma_{\ac}\xrightarrow{\,\bm{\kappa}_{\ac}\,} \Lambda^\times$ the character defined by mapping the element $((1+p)^{-1}, (1+p))$ to the group-like element $[1+p]$. For an integer $n\ge1$, let $\Gamma_n=(1+p\bb{Z}_p)/(1+p\bb{Z}_p)^{p^{n-1}}$ and let $\Lambda_n=\mathcal{O}[\Gamma_n]$.

Let $\varphi_{0,\frk{P}}$ be the $p$-adic avatar of a Hecke character $\varphi_0$ of $K$ of conductor dividing $\frk{cp}$ and infinity type $(-1,0)$, where $\frk{c}$ is coprime to $p$ and $\frk{P}$ is the maximal ideal of a sufficiently large extension of $\QQ_p$ (cf. \S\ref{subsubsec_2025_11_12_1049}). Let $\chi_{\varphi_0}$ denote the central character of $\varphi_0$. We impose the following self-duality condition.
\begin{itemize}
    \item[(\mylabel{item_self_duality}{\textbf{SD}})] $\chi_f\chi_g\chi_{\varphi_0}\varepsilon_K=\omega^{2a}$ for some $a\in \bb{Z}$.
\end{itemize}

Let $m$ be a square-free product of primes coprime to $Np$.  Let $K[m]$ denote the maximal $p$-subextension of the ring class field of $K$ of conductor $m$.
Through suitable modifications of the classes ${\Delta}_{n}^{\etale}(f,g)$ and $\widetilde{\Delta}_{n}^{\etale}(f,g)$, we build on the ideas developed in \cite{LLZ2} to define classes
\[
z_{n,m},\,\tilde{z}_{n,m}\in H^1(\bb{Q},T_f^*\otimes T_g^*\otimes \Ind_{K[m]}^{\bb{Q}}\Lambda_n(\varphi_{0,\frk{P}}^c\bm{\kappa}_{\ac}^{-1})(2)),
\]
where $c$ denotes the complex conjugation. By Shapiro's isomorphism, these classes can also be regarded as classes in
\[
H^1(K[m],T_f^\ast\otimes T_g^\ast\otimes \Lambda_n(\varphi_{0,\frk{P}}^c\bm{\kappa}_{\ac}^{-1})(2))
\simeq H^1(K[mp^n],T_f^\ast\otimes T_g^\ast(\varphi_{0,\frk{P}}^c)(2))\,.
\]

\subsubsection{} Let us put $M=T_f^\ast\otimes T_g^\ast(\varphi_{0,\frk{P}}^c)(2)$ to ease our notation, and consider the following condition. 
\begin{itemize}
    \item[(\mylabel{item_H0}{\textbf{H}$^0$})] The residual $G_{K}$-representation $ \overline{M}$ satisfies $H^0(K,\overline{M})=0$\,.
\end{itemize}
Building on Theorem~\ref{thmA}, we prove the following norm relations for the classes $z_{n,m}$.
\begin{lthm}[Propositions~\ref{prop:horizontal-norm} and~\ref{prop_thm_2025_07_03_1627}]\label{thmB}
Suppose that \eqref{item_H0}  holds. Let $n$ be a positive integer, $m$ a square-free product of primes coprime to $Np$ that are split in $K$, and $q$ a prime coprime to $Nmp$ that is split in $K$. Then
    \begin{align*}
        \mathrm{cor}_{K[mqp^n]/K[mp^n]}(z_{n,mq}) &= \cP_q\cdot z_{n,m},
    \end{align*}
    where $\cP_q$ is an explicit Euler factor.
Furthermore, if $n\geq 2$, then
    \[
    \mathrm{cor}_{K[mp^{n+1}]/K[mp^n]}(z_{n+1,m})=a_p(g)z_{n,m}-\chi_g(p)\mathrm{res}_{K[mp^{n-1}]/K[mp^n]}(z_{n-1,m}),
    \]
    where $\chi_g$ is the nebentype of $g$.
\end{lthm}

\subsubsection{} The second norm relation in Theorem~\ref{thmB} is similar to the one satisfied by Heegner points on an elliptic curve. It is precisely this relation that allows us to apply the machinery of Sprung to define signed integral classes.  For $n\ge1$, let $C_{g,n}$ denote the $2\times2$ matrix over $\Lambda\simeq\cO\llbracket X\rrbracket$ given by
$\begin{bmatrix}
    a_p(g)&1\\ -\chi_g(p)\Phi_n&0
    \end{bmatrix}$, where $\Phi_n=\sum_{j=0}^{p-1}(1+X)^{jp^{n-1}}$ denotes the $p^n$-th cyclotomic polynomial in $1+X$. 
\begin{lthm}[Theorems~\ref{thm:generic-decomposition} and \ref{thm:decompo}]\label{thmC}
     Suppose that \eqref{item_H0}  holds. Let $m$ be a square-free product of primes coprime to $Np$ that are split in $K$. There exist unique elements $z_m^\sharp,z_m^\flat\in H^1(K[m],T_f^\ast\otimes T_g^\ast\otimes \Lambda(\varphi_{0,\frk{P}}^c\bm{\kappa}_{\ac}^{-1})(2))$ such that
     \[
    \begin{bmatrix}
        z_{n,m}\\-\chi_g(p)\res_{K[mp^{n-1}]/K[mp^n]}(z_{n-1,m})
    \end{bmatrix}
    =C_{g,n-1}\cdots C_{g,1}\begin{bmatrix}
        \pr^\infty_n z_m^\sharp\\\pr^\infty_n z_m^\flat
    \end{bmatrix},
    \]
    where  
    $$H^1(K[m],T_f^\ast\otimes T_g^\ast\otimes \Lambda(\varphi_{0,\frk{P}}^c\bm{\kappa}_{\ac}^{-1})(2))\xrightarrow{\,\pr^\infty_n\,} H^1(K[m],T_f^\ast\otimes T_g^\ast\otimes \Lambda_n(\varphi_{0,\frk{P}}^c\bm{\kappa}_{\ac}^{-1})(2))$$ 
    denotes the natural projection map.
    Furthermore,  let $q$ be a prime coprime to $Nmp$ and split in $K$. Then
    \begin{align*}
        \mathrm{cor}_{K[mq]/K[m]}\left(z_{m q}^\bullet\right) &= \cP_q\cdot z_{m}^\bullet\,,\qquad \bullet\in\{\sharp,\flat\}\,.
    \end{align*}
\end{lthm}

\subsubsection{} 
\label{subsubsection_125_2025_11_02}
 We assume until \S\ref{subsubsec_2025_10_31_1451} that the residual representation $\overline{M}$ is absolutely irreducible (see also our ``big image'' hypothesis \eqref{item_BI}, which is required in Theorem~\ref{thmD} below).

Similar to previous works on non-ordinary Iwasawa theory for $\GL_2$, the matrices $C_{g,n}$ can be used to define local signed Coleman maps. Their kernels then give rise to signed Selmer groups that we denote by $H^1_{\bullet,\bullet}(K[p^\infty],M)$ for $\bullet\in \lbrace \sharp,\flat\rbrace$. Also, the dual local conditions for the $p$-divisible module $A:=(M\otimes_{\bb{Z}_p}\bb{Q}_p)/M$ give rise to Selmer groups $\Sel_{\bullet,\bullet}(K[p^\infty],A)$, and we denote by $X_{\bullet,\bullet}(K[p^\infty],A)$ their Pontryagin duals. We refer the reader to \S\ref{subsec:ES} for precise definitions of these objects.

Under the assumption that $a_p(g)=0$ and that hypothesis \eqref{item_H0minus} (formulated in \S\ref{subsec_5_1_2025_11_02_1128}) holds, we prove in \S\ref{subsec_5_2_2025_11_02_1128} that $z_1^\bullet\in H^1_{\bullet,\bullet}(K[p^\infty],M)$. Based on these constructions, we may formulate the following signed Iwasawa main conjectures, similar to the Heegner-point Iwasawa main conjectures of Perrin-Riou studied in \cite{PR-Heegner-IMC} (see also the supersingular analogue discussed in \cite{castellawan}).

\begin{conjecture*}
    Let $\bullet\in \lbrace \sharp,\flat\rbrace$. The class $z_1^\bullet\in H^1(K[p^\infty],M)$ lies in $H^1_{\bullet,\bullet}(K[p^\infty],M)$. Moreover, $z_1^\bullet$ is not $\Lambda$-torsion, both $H^1_{\bullet,\bullet}(K[p^\infty],M)$ and $X_{\bullet,\bullet}(K[p^\infty],A)$ are rank-one $\Lambda$-modules and
    \[
    \Char_{\Lambda}\left(X_{\bullet,\bullet}(K[p^\infty],A)_{\mathrm{tors}}\right)=\Char_\Lambda\left(\frac{H^1_{\bullet,\bullet}(K[p^\infty],M)}{\Lambda \cdot z_1^\bullet}\right)^2.
    \]
\end{conjecture*}

The Euler system machinery developed in \cite{JNS} implies, under appropriate hypotheses, one inclusion in these main conjectures:

\begin{lthm}[Theorem~\ref{thm:mainconjectures}]\label{thmD}
Let $\bullet\in\lbrace \sharp,\flat\rbrace$. Assume that $a_p(g)=0$, that hypothesis \eqref{item_H0minus} holds and that the ``big image'' hypothesis \eqref{item_BI} holds. If $z_1^\bullet\in H^1_{\bullet,\bullet}(K[p^\infty],M)$ is not $\Lambda$-torsion, then both $H^1_{\bullet,\bullet}(K[p^\infty],M)$ and $X_{\bullet,\bullet}(K[p^\infty],A)$ are rank-one $\Lambda$-modules and
\[
\Char_{\Lambda}\left(X_{\bullet,\bullet}(K[p^\infty],A)_{\mathrm{tors}}\right)\supseteq\Char_{\Lambda} \left(\frac{H^1_{\bullet,\bullet}(K[p^\infty],M)}{\Lambda \cdot z_1^\bullet} \right)^2.
\]
\end{lthm}

We refer the reader to Remark~\ref{rem_2025_10_13_0908} for a discussion of our assumption $a_p(g)=0$. 

\subsubsection{} 
\label{subsubsec_2025_10_31_1451}
In Appendix~\ref{sec:Regularized_definite_diagonal_cycles}, we discuss the definite counterpart of our diagonal cycles. Namely, taking $B$ to be definite and letting $\mathbf{S}_{\mathcal{U}_{\mathbf{n}}}$ be the (zero-dimensional) Shimura set associated with $B_E$ of level $\mathcal{U}_{\mathbf{n}}$, we define an element $\Theta_{\mathcal{U}_{\mathbf{n}}}\in  \Zp[\mathbf{S}_{\mathcal{U}_{\mathbf{n}}}]$ by imitating the construction of $\Delta_{\mathcal{U}_{\mathbf{n}}}$. Thanks to~\cite{LoefflerSpherical}, these theta elements satisfy a norm relation identical to that of the diagonal cycles. Moreover, they give rise to an alternative construction of Hsieh's diagonal theta elements, which were used by Hsieh~\cite{hsiehpadicbalanced} to construct balanced triple-product $p$-adic $L$-functions, with the advantage that our approach a priori allows more flexibility in the $p$-adic variation. We explain in \S\ref{subsec:Loefflerfamilyofthetaelements} how Hsieh’s construction of diagonal theta elements can be reinterpreted within Loeffler’s framework (see Proposition \ref{prop_rewriting_Hsieh} and Corollary \ref{cor_rewriting_Hsieh}).

In \S\ref{subsec_signedpadic}, using the theta elements $\Theta_{\mathcal{U}_{\mathbf{n}}}\in  \Zp[\mathbf{S}_{\mathcal{U}_{\mathbf{n}}}]$, we
construct unbounded (resp. bounded signed) balanced $p$-adic $L$-functions in a single variable. Let us give an overview of our construction. Let $f$ and $g$ be weight-two cusp forms on $\GL_2$ as above and let $\h$ be a primitive Hida family on $\GL_2$  with coefficients in a normal domain $\BI$, finite flat over $\Lambda = \cO\llbracket1+p\Zp\rrbracket$. Assume that we have a (non-zero) primitive Jacquet--Langlands lift $(f^B,g^B,\h^B)$ to $B$ (see \eqref{item_JL} for a list of conditions that guarantee this). Then, mimicking the definition of ${\Delta}_{n}^{\etale}(f,g)$ (cf. \S\ref{subsubsec_2025_11_02_0946}), we construct elements $\Theta_{n}(f,g) \in e_{\mathrm{U}_{p}'}\mathcal{O}[S_{U_{n}}]$ that satisfy the relation
    \begin{align}\label{eq_31_10_2025_2253}
        {\rm pr}_{n}^{n+1}\,\Theta_{n+1}(f,g)=a_p(g)\cdot\Theta_{n}(f,g)-\chi_g(p)\cdot \res_{n-1}^{n}\Theta_{n-1}(f,g)\,
    \end{align}
 for any integer $n \geq 2$ (cf. Proposition~\ref{prop_2025_10_15_1640}).   
Let $\alpha$ and $\beta$ denote the roots of the Hecke polynomial of $g$ at $p$. Then,  for $\xi \in \{ \alpha, \beta\}$ and $n \geq 2$, \eqref{eq_31_10_2025_2253} let us show that the $p$-stabilized theta element $$\Theta_{n,\xi}(f,g):=\frac{1}{\xi^{n}}\left(\Theta_{n}(f,g)-\frac{\chi_g(p)}{\xi}\res_{n-1}^{n} \Theta_{n-1}(f,g)\right) \,\in\, e_{\mathrm{U}_{p}'}{\rm Frac}(\mathcal{O})[S_{U_n}]$$
is norm-compatible, in the sense that ${\rm pr}_{n}^{n+1} \Theta_{n+1,\xi}(f,g) = \Theta_{n,\xi}(f,g)$. 

\begin{lthm}[Theorems~\ref{thm_2025_10_15_2117} and \ref{thm_2025_10_16_0900}]\label{thmE}
Let $\sL$ and $\mathcal{D}_{?}$ be as in Definition~\ref{defn_PR_rings}.
\item[i)] Let $\xi\in\{\alpha,\beta\}$ be such that $\ord_p(\xi)<1$.  There exists a unique element 
$$\Theta_{\infty,\xi}(f,g)\in \mathscr{L} \otimes_\LL\varprojlim_n e_{\mathrm{U}_{p}'}\cO[S_{U_n}]$$ such that its image under the natural morphism 
$ \mathscr{L} \otimes_{\LL}\varprojlim_n e_{\mathrm{U}_{p}'}\cO[S_{U_n}] \lra e_{\mathrm{U}_{p}'}{\rm Frac}(\mathcal{O})[S_{U_n}]$ coincides with $\Theta_{n,\xi}(f,g)$. Let us put $\Theta_{\infty,\xi}(f,g,\h):=\h^{B, \dagger}\left(\Theta_{\infty,\xi}(f,g)\right)\in \mathscr{L} \otimes_\LL \mathbf{I}$, where $\h^{B, \dagger}$ is the self-dual twist of $\h$ $($cf. Definition~\ref{def_2025_11_02_0949}\,$(\textup{iii}))$. Then, $\Theta_{\infty,\xi}(f,g,\h) \in \mathcal{D}_{\ord_p(\xi)}$\,.
\item[ii)] There exist unique $\Theta_{\infty}^\sharp(f,g,\h)$, $\Theta_{\infty}^\flat(f,g,\h) \in \mathbf{I}$ such that
    \[
    \begin{bmatrix}
      \Theta_{\infty,\alpha}(f,g,\h)\\  
      \Theta_{\infty,\beta}(f,g,\h)\end{bmatrix}=Q_g^{-1}M_{\log,g}
\begin{bmatrix}
      \Theta_{\infty}^\sharp(f,g,\h)\\  
      \Theta_{\infty}^\flat(f,g,\h)
\end{bmatrix}    ,
    \]
    where $Q_g$ and $M_{\log,g}$ are the matrices introduced in Definition \ref{defn_01_11_25}.
\end{lthm}

\noindent We refer to the elements $\Theta_{\infty,\xi}(f,g,\h)$, resp.  $\Theta_{\infty}^\bullet(f,g,\h)$ for $\bullet \in \{\sharp, \flat \}$, as the unbounded (resp. bounded signed) balanced triple-product $p$-adic $L$-functions. These are trito-non-ordinary versions of Hsieh's balanced $p$-adic $L$-function\footnote{Greenberg and Seveso in \cite{GreenbergSeveso} have also constructed balanced triple-product $p$-adic $L$-functions in the non-ordinary scenario. Since they work over wide open discs properly contained in the full weight space, the (essential part of the) denominators can be absorbed in the coefficient ring. Our work emphasises Iwasawa theory and integrality, which requires that we work over the full weight space. } and, thanks to the calculation on trilinear periods in \cite[\S4]{hsiehpadicbalanced}, they enjoy identical interpolative properties (which characterise them thanks to their indicated growth properties).

\subsubsection{} We conclude by noting that the striking parallelism between Hsieh's construction of theta elements in terms of the special locus on Shimura sets in the definite scenario, and the construction of diagonal cycles in the indefinite setting suggests the existence of a bipartite Euler system (in the sense of Bertolini--Darmon and Howard~\cite{BertoliniDarmon2005, howard06}), but we do not explore this point further in the present article.

\subsection*{Acknowledgments}
We thank Haining Wang for helpful discussions during the preparation of the article. We are grateful to Henri Darmon, Ming-Lun Hsieh, and Yifeng Liu for their encouragement and valuable comments. RA's, KB’s, and AC's research in this publication were conducted with the financial support of Taighde \'{E}ireann -- Research Ireland under Grant number IRCLA/2023/849 (HighCritical). AL's research is supported by the NSERC Discovery Grants Program RGPIN-2020-04259.

\section{\texorpdfstring{$p$-adic\,}\ families of diagonal cycles}
\label{sec:families_of_indefinite_cycles}

The goal of this section is to provide a construction of $p$-adic families of diagonal cycles in the triple product of Shimura curves inspired by the construction of the balanced $p$-adic $L$-function of Hsieh in \cite{hsiehpadicbalanced} (see Appendix  \ref{sec:Regularized_definite_diagonal_cycles} where, in order to highlight the parallelism, we review the construction of theta elements and $p$-adic $L$-functions therein) and by the general technique proposed by Loeffler in \cite{LoefflerSpherical}.

\vspace{-0.1cm}
\subsection{Shimura curves and threefolds}\label{subsec:ShimuraCurves}$\,$

Let $N=N^+N^-$ be a positive integer coprime to $p$, where $N^+$ and $N^-$ are coprime and $N^-$ is square-free. Let $B$ be the indefinite quaternion algebra over $\QQ$ of discriminant $N^-$. We also put $E = \QQ^{\oplus 3}$ and $B_E = B^{\oplus 3}$.  Once and for all fix an isomorphism 
 \begin{equation}
     \Psi=\prod_{q\nmid N^-}\Psi_q\,:\quad B(\widehat\QQ^{(N^-)})\simeq M_2(\widehat\QQ^{(N^-)}).
 \end{equation} 
We fix a maximal order $\mathcal{O}_{B}$ such that $\Psi_q(\mathcal{O}_{B} \otimes \ZZ_q) = M_2(\ZZ_q)$ for all $q \nmid \infty N^{-}$. We then let $U_0(N)$ be the open compact subgroup of $\widehat{B}^{\times}$ defined by the Eichler order $R_N$ of level $N^+$ in $\mathcal{O}_{B}$  and let 

$$U_1( N) = \{ g \in U_0(N)\,:\, \Psi_q(g_q) \equiv \left( \begin{smallmatrix}
    \star & \star \\ 0 & 1
\end{smallmatrix} \right)\, (\text{mod } N \ZZ_q ) \text{ for } q| N^+ \}. $$

\vspace{-0.1cm}

\subsubsection{} For a natural number $n$, we put
 \begin{align}
   \label{eqn_new_level}
     U_n &:= \{ g \in U_1(N) \,:\, \Psi_p(g_p) \equiv \left( \begin{smallmatrix} a & \star  \\  & a \end{smallmatrix}\right)\, (\text{mod } p^n)\, ,\,a \in\Zp^\times\}\,, \\
 \label{eqn_2024_07_31_1640}
     U_{Z,n}  &:= \{ g \in U_1(N) \,:\, \Psi_p(g_p) \equiv \left( \begin{smallmatrix} a &  \\  & a \end{smallmatrix}\right)\, (\text{mod } p^n)\, ,\,a \in\Zp^\times\}\,, 
\\
 \label{eqn_2024_07_31_1649}
     U_{0,n} &: = \{ g \in U_1(N)  \,:\, \Psi_p(g_p) \equiv \left( \begin{smallmatrix} \star & \star \\  & \star \end{smallmatrix}\right)\, (\text{mod } p^n)\}\,,
\\
 \label{eqn_2024_07_31_1650}
     U_{1,n} &: = \{ g \in  U_1(N)  \,:\, \Psi_p(g_p) \equiv \left( \begin{smallmatrix} \star &  \star \\  & 1 \end{smallmatrix}\right)\, (\text{mod } p^n)\}\,.
 \end{align}
\subsubsection{}\label{subsec_2025_10_14_1351} Let us denote by $\iota: B \hookrightarrow B_E$ the diagonal embedding, and let $\mathcal{U}_i(N) : = U_i(N)^{\times 3} \subseteq \widehat{B}_E^{\times}$ for $i = 0,1$. For each $\mathbf{n}=(n_1,n_2,n_3) \in \mathbb{Z}_{\geq 0}^3$, we define 
\begin{align*}
    \mathcal{U}_{1,\mathbf{n}} &:=    U_{1,n_1} \times U_{1,n_2} \times U_{1,n_3}, 
    \\ 
    \mathcal{U}_{Z,\mathbf{n}} &:=\left \{ g \in  \mathcal{U}_{1}(N)\,:\, g \equiv \left(\left( \begin{smallmatrix}  a_1 & \star \\ 0 & a_0 \end{smallmatrix}\right) \mod p^{n_1},\left( \begin{smallmatrix}  a_2 & \star \\ 0 & a_0 \end{smallmatrix}\right) \mod p^{n_2},\left( \begin{smallmatrix} a_3 & \star \\ 0 & a_0 \end{smallmatrix}\right) \mod p^{n_3}\right)\, ,\,a_i \in\Zp^\times\right\},  \\ 
    \mathcal{U}_{\mathbf{n}} &:=\left \{ g \in  \mathcal{U}_{1}(N)\,:\, g \equiv \left(\left( \begin{smallmatrix}  a & \star \\ 0 & a \end{smallmatrix}\right) \mod p^{n_1},\left( \begin{smallmatrix}  a & \star \\ 0 & a \end{smallmatrix}\right) \mod p^{n_2},\left( \begin{smallmatrix} a & \star \\ 0 & a \end{smallmatrix}\right) \mod p^{n_3}\right)\, ,\,a \in\Zp^\times\right\},\\
    \mathcal{U}^{\mathbf{n}} &:=\left \{ g \in  \mathcal{U}_{1}(N)\,:\, g \equiv \left(\left( \begin{smallmatrix}  a & 0\\ \star  & a \end{smallmatrix}\right) \mod p^{n_1},\left( \begin{smallmatrix}  a &0 \\ \star  & a \end{smallmatrix}\right) \mod p^{n_2},\left( \begin{smallmatrix} a & 0 \\ \star  & a \end{smallmatrix}\right) \mod p^{n_3}\right)\, ,\,a \in\Zp^\times\right\}, \\ 
     \mathcal{U}_{{\rm spl},\mathbf{n}} &:=U_{n_1} \times U_{n_2} \times U_{n_3}\,.
\end{align*}
When $\mathbf{n}=(n,n,n)$, we write $\mathcal{U}_{\bullet,(n)}$ in place of $ \mathcal{U}_{\bullet,\mathbf{n}}$ for any choice of $\bullet \in \{ 1, Z , \emptyset, {\rm spl} \}$. Note that $\mathcal{U}_{1,(n)}$ and $\mathcal{U}_{(n)}$ are normal subgroups of $\mathcal{U}_{Z,(n)}$, and the quotient $\mathcal{U}_{Z,(n)}/\mathcal{U}_{1,(n)}$ is isomorphic to $(\ZZ / p^n \ZZ )^\times$.   

\subsubsection{Notation} If $\mathbf{m}=(m_1,m_2,m_3), \mathbf{n}=(n_1,n_2,n_3)\in \ZZ^3$, we put $\mathbf{m}+\mathbf{n}:=(m_1+n_1,m_2+n_2,m_3+n_3)$.

\subsubsection{} We fix an isomorphism $\epsilon_\infty : B \otimes \mathbb{R} \simeq M_2(\mathbb{R})$, and we let $X(U)$ be the Shimura curve over $\QQ$ of level $U$ (where $U$ denotes any one of the level subgroups above), whose complex points are given by the double cosets in \[ B^{\times} \backslash (\mathcal{H}^{\pm} \times \cO_{B,f}^\times / U), \] 
where $B^{\times}$ acts on $\mathcal{H}^{\pm}:=\mathbb{C}\setminus \mathbb{R}$ via the isomorphism $\epsilon_\infty: (B \otimes \mathbb{R})^\times \simeq \GL_2(\mathbb{R})$\footnote{We let $\GL_2(\mathbb{R})$ act on $\mathcal{H}^{\pm}$ via the natural left action $\left ( \begin{smallmatrix}
    a & b \\ c & d 
\end{smallmatrix} \right ) \cdot z = \frac{a z +b }{c z + d}$, for $\left ( \begin{smallmatrix}
    a & b \\ c & d 
\end{smallmatrix} \right ) \in \GL_2(\mathbb{R})$ and $z \in \mathcal{H}^{\pm}$.}. The Shimura curve $X(U)$ is proper unless $N^-=1$ and, if $N^+ $ is big enough, is a fine moduli space of elliptic curves, resp. false elliptic curves\footnote{These are abelian surfaces with an action by $\mathcal{O}_{B}$ and level-$U$ structures.}, if $N^-=1$, resp. $N^-\not =1$. This admits an integral refinement, in the sense that $X(U)$ has an integral model $\mathscr{X}(U)$ over $\ZZ[1/N p]$, which represents the following moduli problem: to any $\ZZ[1/Np]$-scheme $S$, we assign the set of isomorphisms classes of $(A/S, \alpha, \eta_{U})$, where $A/S$ is an abelian scheme over $S$ of relative dimension $2$, $\alpha: \mathcal{O}_{B} \hookrightarrow {\rm End}_S(A)$, and $\eta_{U}$ is a level-$U$ structure on $A/S$.

\subsubsection{} 

Given a sufficiently small open compact subgroup $\mathcal{U} \subseteq  \widehat{B}_E^{\times}$, we denote by $\mathbf{X}_{\mathcal{U}}$ the Shimura threefold over $\QQ$ of level $\mathcal{U}$. When $\mathcal{U} = U_1 \times U_2 \times U_3$,  the threefold $\mathbf{X}_{\mathcal{U}}$ is the product of Shimura curves $X(U_1) \times X(U_2) \times X(U_3)$. The diagonal embedding $\iota: B \hookrightarrow B_E$ induces a morphism from the Shimura curve to the Shimura threefold as follows. If $U \subseteq \cO_{B,f}^\times$ is any sufficiently small open compact subgroup which maps into $\iota(\cO_{B,f}^\times) \cap  \mathcal{U}$, we have a morphism of varieties over $\QQ$ 
\[ \iota_{U,\mathcal{U}}:X(U) \hookrightarrow \mathbf{X}_{\mathcal{U}}\,,\]
which can be described as $[z,g] \mapsto [((z,z,z),(g,g,g))]$ for complex points. When $U=\iota^{-1}(\mathcal{U})$, the morphism $\iota_{U,\mathcal{U}}$ is a closed immersion of codimension 2.  Finally, recall that right-multiplication by $g  \in \cO_{B,f}^\times$, resp. $g  \in \cO_{B_E,f}^\times$, induces a morphism of Shimura curves $\mathcal{T}_g: X(U) \to X(g^{-1} U g)$, resp. of Shimura threefolds $\mathcal{T}_g: \mathbf{X}_{\mathcal{U}} \to \mathbf{X}_{g^{-1}\mathcal{U}g}$, defined over the reflex field $\QQ$.

\subsection{Hecke correspondences}

\subsubsection{} 
\label{subsubsec_A21_2025_06_30_1533}
For a level subgroup $U \subseteq \widehat{B}^\times$ and any ring $R$, we denote by $\mathcal{H}_U( \widehat{B}^\times,R)$ the space of bi-$U$ invariant compactly supported $R$-valued functions on $\widehat{B}^\times$, endowed with the convolution product $\ast$. When $R=\ZZ$, we simply denote it by $\mathcal{H}_U(\widehat{B}^\times)$. We define $\mathcal{H}_{\mathcal{U}}( \widehat{B}_E^\times,R)$ and $\mathcal{H}_{\mathcal{U}}( \widehat{B}_E^\times)$ analogously. 

For $U \subseteq \mathcal{O}_{B,f}^\times$ as above and $g \in \mathcal{O}_{B,f}^\times$, the characteristic function $[UgU] \in \mathcal{H}_U( \widehat{B}^\times)$ defines a Hecke correspondence 
\[\xymatrix{ 
&X(g^{-1}U g \cap U) \ar[r]^-{\mathcal{T}_{g^{-1}}} \ar[ld]_{\rm pr} &X(U \cap g U g^{-1}) \ar[rd]^{\rm pr'} & \\ X(U)  & & &  X(U)  } \]
which acts on the cohomology of $X(U)$ via $({\rm pr}')_* \circ (\mathcal{T}_{g^{-1}})_* \circ ({\rm pr})^*$. 

Let $\mathcal{H}_{B,\bullet}^{\rm sph}(Np^n)$ denote the subalgebra of $\mathcal{H}_{U_{\bullet,n}}( \widehat{B}^\times)$, where $\bullet \in \{\emptyset,  1\}$ in \S\ref{subsubsec_A21_2025_06_30_1533} and \S\ref{subsubsec_A22_2025_06_30_1533}, generated by the characteristic functions $[U_{\bullet,n}\, g\, U_{\bullet,n}]$, with $g = ( g_q)_q \in \widehat{B}^\times$ such that $g_q = 1$ for every $q \mid Np^n$. To be more precise, for a prime $\ell$ let us consider 
 $\eta_\ell = \left (\begin{smallmatrix}
     \ell & \\ & 1
 \end{smallmatrix} \right) \in \GL_2(\QQ_\ell)$ and $s_\ell =  \left (\begin{smallmatrix}
     \ell & \\ & \ell
 \end{smallmatrix} \right) \in \GL_2(\QQ_\ell)\,,$ 
 and define \[\mathrm{T}_\ell':= [U_{\bullet,n} \Psi^{-1}(\eta_\ell^{-1}) U_{\bullet,n} ],\,\quad \mathrm{S}_\ell:= [U_{\bullet,n}\Psi^{-1}(s_\ell) U_{\bullet,n} ]\,,\]
and put $\mathcal{H}_{B,\bullet}^{\rm sph}(Np^n) := \ZZ \left[ \mathrm{T}_\ell', \mathrm{S}_\ell, \mathrm{S}_\ell^{-1},\, \ell \nmid Np^n \right]$.

When $n >0$, we further consider the operators ${\rm U}_p'$ and $\langle x \rangle$, for $x \in \ZZ_p^\times$, defined by 
\[{\rm U}_p':= [U_{\bullet,n} \Psi^{-1}(\eta_p^{-1}) U_{\bullet,n} ],\,\quad \langle x \rangle:= [U_{\bullet,n} \Psi^{-1}\left( \begin{smallmatrix}
    1 & \\ & x
\end{smallmatrix}\right) U_{\bullet,n} ], \]
and set $\mathcal{H}_{B,\bullet}(Np^n) : = \mathcal{H}_{B,\bullet}^{\rm sph}(Np^n)[{\rm U}_p',  \langle x \rangle, x \in \ZZ_p^\times  ] = \ZZ \left[\{ \mathrm{T}_\ell', \mathrm{S}_\ell, \mathrm{S}_\ell^{-1},\, \ell \nmid Np^n \}, \{ {\rm U}_p',  \langle x \rangle, x \in \ZZ_p^\times\} \right]$.

\noindent The algebra of Hecke correspondences generated by $\mathcal{H}_{B,\bullet}(Np^n)$ will be denoted by $\mathrm{T}_{B,\bullet}(Np^n)$.

\subsubsection{}
\label{subsubsec_A22_2025_06_30_1533}
By a slight abuse of notation, we let $\eta_p$ also denote its image in $(B \otimes \ZZ_p)^\times$ via $\Psi^{-1}$. We let $X(U_{\bullet,n+1}) \xrightarrow[]{\pi_i} X(U_{\bullet,n})$, $i=1,2$, be the two degeneracy maps given as follows: the map $\pi_1$ is the natural degeneracy map, whereas $\pi_2$ is given by 
\[X(U_{\bullet,n+1}) \xrightarrow[]{\mathcal{T}_{\eta_p^{-1}}} X(\eta_p U_{\bullet,n+1} \eta_p^{-1}  ) \xrightarrow[]{{\rm pr}'} X( U_{\bullet,n} ), \]
where ${\rm pr}'$ denotes the natural degeneracy map associated with the inclusion $\eta_p U_{\bullet,n+1} \eta_p^{-1}  \subseteq U_{\bullet,n}$. These projections factor through the Shimura curve of level $U_{\bullet,n} \cap U_{0,n+1}$ as follows. We have a commutative diagram (for $i=1,2$)
\[ \xymatrix{ &X(U_{\bullet,n+1}) \ar[ld]_-{\mu} \ar[dr]^-{\pi_i} & \\ X(U_{\bullet,n} \cap U_{0,n+1}) \ar[rr]_-{\varpi_i} && X(U_{\bullet,n})} \]
where $X(U_{\bullet,n+1}) \xrightarrow[]{\mu} X(U_{\bullet,n} \cap U_{0,n+1})$ denotes the natural projection, and $\varpi_2$ is defined by composing the natural degeneracy map with  $\mathcal{T}_{\eta_p^{-1}}$. When we want to emphasize the levels involved, we denote $\pi_1$ by ${\rm pr}^{n+1}_{n}$.

We record the following statement from \cite{KLZ2}, which will be used in the proof of the $Q$-system relations for the diagonal cycles.

\begin{proposition}[{\cite[Proposition 2.4.5]{KLZ2}}] \label{Prop_of_KLZ17}As morphisms $$H^1_{\etale}(X(U_{\bullet,1}),\ZZ_p(j)) \to H^1_{\etale}(X(U_1(N)),\ZZ_p(j))$$ for any $j \in \ZZ$ we have 
\begin{align*}
    \pi_{2,*} \circ {\rm U}_p' &= p \cdot   \pi_{1,*} \\ 
    \pi_{1,*} \circ {\rm U}_p' &= \mathrm{T}_p' \circ \pi_{1,*} -  \mathrm{S}_p^{-1} \circ  \pi_{2,*}. 
\end{align*}
    
\end{proposition}

\subsubsection{} 
\label{subsubsec_2027_07_02_1557}

For $\bullet \in \{\emptyset,  1, {\rm spl}\}$ and a triple $\mathbf{n}=(n_1,n_2,n_3)$ of integers as above, we let $\mathcal{H}_{B_E,\bullet}^{\rm sph}(Np^{\mathbf{n}})$ denote the subalgebra of $\mathcal{H}_{\mathcal{U}_{\bullet,\mathbf{n}}}( \widehat{B}_E^\times)$ generated by the characteristic functions $[\mathcal{U}_{\bullet,\mathbf{n}}\, \underline{g}\, \mathcal{U}_{\bullet,\mathbf{n}}]$, with $\underline{g}\in \widehat{B}_E^\times$ such that $\underline{g_q} = 1$ for every $q \mid Np$. We define
\begin{align*}
 \mathbb{U}_p' := [\mathcal{U}_{\bullet,\mathbf{n}} \iota( \eta_p^{-1}) \mathcal{U}_{\bullet,\mathbf{n}}]\,.
\end{align*}

\noindent For $i,j \in \{1,2,3\}$ and $g \in (B \otimes \ZZ_p)^\times$, we let $^{i}g, {}^{ij}g \in (B_E \otimes \ZZ_p)^\times$ denote the elements that equal to $g$ in position $i$, resp. in positions $i$ and $j$, and to the identity element otherwise. Then we let $^{i}\mathbf{U}_p'$ and $^{ij}\mathbf{U}_p'$ be the operators associated to $^{i}\eta_p^{-1}$ and $^{ij}\eta_p^{-1}$, respectively. Finally, if $x,y,z \in \ZZ_p^\times$, we let $\langle x, y, z \rangle$ be the operator attached to the element $\left(\left( \begin{smallmatrix}
    1 & \\ & x
\end{smallmatrix}\right),\left( \begin{smallmatrix}
    1 & \\ & y
\end{smallmatrix}\right),\left( \begin{smallmatrix}
    1 & \\ & z
\end{smallmatrix}\right)\right)$, and set $\langle \mathbb{x} \rangle := \langle x, x,x \rangle$.

Let us define
$$\mathcal{H}_{B_E,\bullet}(Np^{\mathbf{n}}) : = \mathcal{H}_{B_E,\bullet}^{\rm sph}(Np^{\mathbf{n}}) [\{  \mathbb{U}_p',{}^{i}\mathbf{U}_p',{}^{ij}\mathbf{U}_p', \quad  i,j \in \{ 1,2,3 \}\}, \{ \langle x,y,z \rangle, \quad \, x,y,z \in \ZZ_p^\times \}\rangle]$$
 and let $\mathrm{T}_{B_E,\bullet}(Np^{\mathbf{n}})$ be the algebra of Hecke correspondences generated by $\mathcal{H}_{B_E,\bullet}(Np^{\mathbf{n}})$.

 \subsubsection{Compatibility of Hecke actions} 
 \label{subsubsec_2025_10_15_1335}
 For $\mathbf{n}=(n_1,n_2,n_3)$, let us denote by 
 $$\mathbf{X}_{\mathcal{U}_{\mathbf{n}}}  \xrightarrow{\,{\rm b}_{\mathbf{n}}\,} \mathbf{X}_{n_1,n_2,n_3} = \mathbf{X}_{\mathbf{n}} :=X({U_{n_1}}) \times X({U_{n_2}}) \times X({U_{n_3}})$$ 
 the natural projection map.

\begin{proposition}
\label{prop_2025_05_15_0828}
    Let $T$, resp.  $\widetilde{T}$, be the Hecke correspondences in $\mathrm{T}_{B_E}(Np^{\mathbf{n}})$, resp. $\mathrm{T}_{B_E,{\rm spl}}(Np^{\mathbf{n}})$, associated with $g \in \mathcal{O}_{B_E,f}^\times$. Then
     \[ {\rm b}_{\mathbf{n}}^* \circ T = \widetilde{T}  \circ {\rm b}_{\mathbf{n}}^*, \text{ and } {\rm b}_{\mathbf{n},\ast} \circ  \widetilde{T} = T \circ {\rm b}_{\mathbf{n},\ast}.\]
\end{proposition}

 \begin{proof}
 The Hecke correspondences $T$ and $\widetilde{T}$ are given by \[ \mathbf{X}_{\mathcal{U}_{\mathbf{n}}} \xleftarrow{{\rm pr}}  \mathbf{X}_{g^{-1}\,\mathcal{U}_{\mathbf{n}}\, g \,\,\cap \,\mathcal{U}_{\mathbf{n}}}    \xrightarrow{{\rm pr}' \circ \mathcal{T}_{g^{-1}}}   \mathbf{X}_{\mathcal{U}_{\mathbf{n}}},\]
 \[\mathbf{X}_{\mathbf{n}} \xleftarrow{\widetilde{\rm pr}}  \mathbf{X}_{g^{-1}\mathcal{U}_{{\rm spl},\mathbf{n}} \,g\,\, \cap\, \mathcal{U}_{{\rm spl},\mathbf{n}}}    \xrightarrow{\widetilde{\rm pr}' \circ \mathcal{T}_{g^{-1}}}   \mathbf{X}_{\mathbf{n}}.\]
    We claim that the result follows from the existence of a finite \'etale morphism $\psi:  \mathbf{X}_{g^{-1}\mathcal{U}_{\mathbf{n}} g \cap \mathcal{U}_{\mathbf{n}}}  \to  \mathbf{X}_{g^{-1}\mathcal{U}_{{\rm spl},\mathbf{n}} g \cap \mathcal{U}_{{\rm spl},\mathbf{n}}}  $ such that 
    $$ \xymatrix{  \mathbf{X}_{\mathcal{U}_{\mathbf{n}}} \ar[d]_-{{\rm b}_{\mathbf{n}}} & & \mathbf{X}_{g^{-1}\mathcal{U}_{\mathbf{n}} g \cap \mathcal{U}_{\mathbf{n}}} \ar[d]^-{\psi}  \ar[ll]_-{{\rm pr}}  \ar[rr]^-{{\rm pr}' \circ \mathcal{T}_{g^{-1}}}   &  &\mathbf{X}_{\mathcal{U}_{\mathbf{n}}}\ar[d]^{{\rm b}_{\mathbf{n}}} \\ \mathbf{X}_{\mathbf{n}} & &  \mathbf{X}_{g^{-1}\mathcal{U}_{{\rm spl},\mathbf{n}} g \cap \mathcal{U}_{{\rm spl},\mathbf{n}}}  \ar[ll]_-{\widetilde{\rm pr}}  \ar[rr]^-{\widetilde{\rm pr}' \circ \mathcal{T}_{g^{-1}}}   & &\mathbf{X}_{\mathbf{n}}   }$$
    is commutative and has Cartesian squares. 
 Indeed, if that is the case, we have 
    \begin{align*}
        {\rm b}_{\mathbf{n}}^* \circ \widetilde{T} & =   {\rm b}_{\mathbf{n}}^* \circ (\widetilde{\rm pr}' \circ \mathcal{T}_{g^{-1}})_* \circ \widetilde{\rm pr}^* \\ &=   ({\rm pr}' \circ \mathcal{T}_{g^{-1}})_* \circ \psi^* \circ \widetilde{\rm pr}^* \\
        &=   ({\rm pr}' \circ \mathcal{T}_{g^{-1}})_* \circ {\rm pr}^* \circ {\rm b}_{\mathbf{n}}^* \\ 
        & = T \circ {\rm b}_{\mathbf{n}}^* ,
    \end{align*}
    where the second equality follows from the Cartesianness of the right square, while the third follows from the commutativity of the left square. Similarly, 
  \begin{align*}
        {\rm b}_{\mathbf{n},\ast} \circ T & =    {\rm b}_{\mathbf{n},\ast} \circ ({\rm pr}' \circ \mathcal{T}_{g^{-1}})_* \circ {\rm pr}^* \\ &=   (\widetilde{\rm pr}' \circ \mathcal{T}_{g^{-1}})_* \circ \psi_* \circ {\rm pr}^* \\
        &=  (\widetilde{\rm pr}' \circ \mathcal{T}_{g^{-1}})_* \circ \widetilde{\rm pr}^* \circ {\rm b}_{\mathbf{n},\ast} \\ 
        & = \widetilde{T} \circ {\rm b}_{\mathbf{n},\ast} ,
    \end{align*}
where the second equality follows from the commutativity of the right square, while the third follows from Cartesianness of the left square.
We are left to show that such a $\psi$ exists. As the level subgroups $\mathcal{U}_{\mathbf{n}}$ and $\mathcal{U}_{{\rm spl},\mathbf{n}}$ differ only at $p$, the existence of $\psi$ is easily seen for any $g$ with trivial $p$-component. If $T$ and $\tilde{T}$ are given by the diamond operator $ \langle x, y, z \rangle$, with $x,y,z \in \ZZ_p^\times$, $\psi$ can be taken to be ${\rm b}_{\mathbf{n}}$ (in this case all the horizontal arrows of the diagram above are isomorphisms). We now verify the existence of $\psi$ for $g = \iota(\eta_p^{-1})$; the other cases follow similarly. When $g = \iota(\eta_p^{-1})$, the $p$-component of $V: = g^{-1}\mathcal{U}_{\mathbf{n}} g \cap \mathcal{U}_{\mathbf{n}}$, resp. $U: = g^{-1}\mathcal{U}_{{\rm spl},\mathbf{n}} g \cap \mathcal{U}_{{\rm spl},\mathbf{n}}$, is \[ V_p = \left \{ (g_1,g_2,g_3) \in \mathcal{U}_{\mathbf{n},p} \,:\, \Psi_p(g_i) \equiv \left( \begin{smallmatrix} \star &  0 \\ \star & \star \end{smallmatrix}\right)\, (\text{mod } p) \, \,\, \forall i \right \}, \]
resp. \[ U_p = \left \{ (g_1,g_2,g_3) \in \mathcal{U}_{{\rm spl},\mathbf{n},p} \,:\, \Psi_p(g_i) \equiv \left( \begin{smallmatrix} \star &  0 \\ \star & \star \end{smallmatrix}\right)\, (\text{mod } p) \, \,\, \forall i\right \}. \]
Let $\psi : \mathbf{X}_V \to \mathbf{X}_U$ denote the natural projection induced by the inclusion $V \subseteq U$. It is a Galois cover with Galois group isomorphic to $\mathcal{U}_{{\rm spl},\mathbf{n}} / \mathcal{U}_{\mathbf{n}}$. The commutativity of the diagram above is immediate. The Cartesianness of the two squares follows from the fact $ \mathcal{U}_{\mathbf{n}} \backslash \mathcal{U}_{{\rm spl},\mathbf{n}} / U  $ is a singleton and $U \cap \mathcal{U}_{\mathbf{n}} = V$. 
 \end{proof}

\subsection{Modified diagonal cycles}
\label{subsec_4_3_2024_07_31_1634}
To define the regularized diagonal cycles, we define the element
\[ u := \left(  \left( \begin{smallmatrix} 1 &  \\  & 1 \end{smallmatrix} \right), \left( \begin{smallmatrix} 1 &  1 \\  & 1 \end{smallmatrix} \right),\left( \begin{smallmatrix}  &  1 \\ -1 &  \end{smallmatrix} \right)\right) \in \GL_2^3(\Zp)\,.\]
\subsubsection{}
    This choice of $u$ is motivated by the following observation. Let $Q$, resp. $\overline{Q}$, 
denote the upper-triangular, resp. lower-triangular, Borel parabolic of $\GL_2^3/\Zp$ and let $\GL_2/\Zp$ act on the left on the flag variety $\GL_2^3/\overline{Q}$. It is well-known that it acts with an open orbit, whose stabilizer is the center $Z$ of $\GL_2$. A representative of this orbit can then be chosen to be $u$\footnote{Note that this flag variety can be identified with $3$-copies of $\mathbb P^1(\ZZ_p)$, allowing each copy of ${\GL}_2(\ZZ_p)$ to act on the \emph{right} as M\"obius transformations, so that the lower triangular Borel subgroup is the stabilizer of $[1:0]$. The point in  the flag variety $\mathbb P^1(\ZZ_p)\times \mathbb P^1(\ZZ_p)\times \mathbb P^1(\ZZ_p)$ that corresponds to $u$ is then 
$$[1:0]u:=\left( [1: 0] \left(\begin{smallmatrix} 1 &  1 \\  & 1 \end{smallmatrix} \right), [1: 0]\left( \begin{smallmatrix} 1 &  \\  & 1 \end{smallmatrix} \right),[1: 0]\left( \begin{smallmatrix}  &  1 \\ -1 &  \end{smallmatrix} \right)\right)=\left([1:1], [1:0], [0:1]\right)\,.$$
The left action of $\GL_2(\ZZ_p)$ on the flag variety (once it is identified by $\mathbb P^1(\ZZ_p)\times \mathbb P^1(\ZZ_p)\times \mathbb P^1(\ZZ_p)$ as above) is given by coordinate-wise M\"obius transformations. In particular, for $g=\left(\begin{smallmatrix} a &  b \\ c & d \end{smallmatrix} \right)\in \GL_2(\ZZ_p)$, we have
$$g\cdot u=\left([a+b:c+d],[a:c], [b:d] \right)\in \mathbb P^1(\ZZ_p)\times \mathbb P^1(\ZZ_p)\times \mathbb P^1(\ZZ_p)\,.$$
It is then clear that the stabilizer of $u$ under this action is the center of $\GL_2(\ZZ_p)$, and that the orbit of $u$ is indeed open.}. This observation will allow us to employ the machinery developed in \cite{LoefflerSpherical} to show that the images of the Shimura curves inside their triple product twisted by right multiplication by $u$ are norm compatible.

\begin{defn}
\label{defn_2024_03_12_1105}
Let us consider the following Shimura curves and threefolds:
\begin{itemize}
    \item $X_{Z,n} := X(U_{Z,n})$, where we recall that 
$U_{Z,n} = \{ g \in U_{0,n} \,:\, g \equiv \left ( \begin{smallmatrix} a &  \\  & a \end{smallmatrix} \right)\, \mod p^{n}  \}$,
\item $\mathbf{X}_{\mathcal{U}_\mathbf{n}}$ and $\mathbf{X}_{\mathcal{U}^{\mathbf{n}}}$ are the Shimura threefolds associated with the level subgroups $\mathcal{U}_{\mathbf{n}}$ and  \begin{align*}
  \mathcal{U}^{\mathbf{n}} &=\left \{ g \in  \mathcal{U}_{1}(N)\,:\, g \equiv \left(\left( \begin{smallmatrix}  a & 0\\ \star  & a \end{smallmatrix}\right) \mod p^{n_1},\left( \begin{smallmatrix}  a &0 \\ \star  & a \end{smallmatrix}\right) \mod p^{n_2},\left( \begin{smallmatrix} a & 0 \\ \star  & a \end{smallmatrix}\right) \mod p^{n_3}\right)\, ,\,a \in\Zp^\times\right\}.
     \end{align*}
\end{itemize}
\end{defn}
As before, we write $(n)$ in place of the triple $(n,n,n)$, and $\mathbf{X}_{\mathcal{U}_{(n)}}$ in place of $\mathbf{X}_{\mathcal{U}_{(n,n,n)}}$ (similarly for $\mathbf{X}_{\mathcal{U}^{(n)}}$).  Note that right multiplication by $\iota(\eta_p^{\mathbf{n}})$ (where $\eta_p^{\mathbf{n}}:=(\eta_p^{n_1},\eta_p^{n_2},\eta_p^{n_3})$) induces an isomorphism $\mathcal{T}_{\iota(\eta_p^{\mathbf{n}})}: \mathbf{X}_{\mathcal{U}^{\mathbf{n}}}  \xrightarrow{\sim} \mathbf{X}_{\mathcal{U}_\mathbf{n}}$. We start with the following auxiliary lemma.

\begin{lemma}\label{lemma_pullback_opencompacts}
For every $n \in \NN$, we have $u\, \mathcal{U}^{(n)} u^{-1} \cap \iota(\cO_{B,f}^\times) = U_{Z,{n}}$. 
\end{lemma}

\begin{proof}
The $p$-component of an element $g = \left(\left( \begin{smallmatrix}
    a_1 & b_1 \\ 
    c_1 & d_1
\end{smallmatrix} \right),\left( \begin{smallmatrix}
    a_2 & b_2 \\ 
    c_2 & d_2
\end{smallmatrix} \right),\left( \begin{smallmatrix}
    a_3 & b_3 \\ 
    c_3 & d_3
\end{smallmatrix} \right) \right)\in u\, \mathcal{U}^{(n)} u^{-1} $ reduces modulo $p^{n}$ to 
\[ \left(\left(\begin{smallmatrix} a &  \\ c_1 & a
\end{smallmatrix}\right), \left(\begin{smallmatrix} a +c_2 & -c_2 \\ c_2 & a-c_2
\end{smallmatrix}\right), \left(\begin{smallmatrix} a  & -c_3 \\  & a
\end{smallmatrix}\right)\right). \] 
Then $g$ belongs to $\iota(\cO_{B,f}^\times)$ if and only if $c_i \equiv 0$ modulo $p^{n}$. This shows the desired equality. 
\end{proof}

\subsubsection{} By Lemma \ref{lemma_pullback_opencompacts}, we obtain closed immersions 
\begin{align*}
    \iota_{n}^u \,:\quad  X_{n} &\lra \mathbf{X}_{\mathcal{U}^{(n)}}\,,
\end{align*} given by the composition of the diagonal embedding $\iota_{U_{Z,n}\,,\, u \mathcal{U}^{(n)} u^{-1}}$ with $\mathcal{T}_u$.

\begin{lemma}\label{lemma_cartesian_diagram}
For any positive integer $n$, we have the commutative diagram
\[ \xymatrix{  
  & & \mathbf{X}_{\mathcal{U}^{(n+1)}} \ar[d]^-{\mu_n^{E} } \\ X_{Z,n+1} \ar[d]_-{\varpi_{Z,n}} \ar[urr]^-{\iota_{n+1}^u} \ar[rr] & & \mathbf{X}_{\mathcal{U}^{(n)}_\circ } \ar[d]^-{\varpi_n^{E}} \\ 
X_{Z,n}\ar[rr]_-{\iota_{n}^u} & & \mathbf{X}_{\mathcal{U}^{(n)}} }\, \]
where $\mathcal{U}^{(n)}_{\circ} = \mathcal{U}^{(n)} \,\cap\,  \iota(\eta_p)\, \mathcal{U}^{(n)} \iota(\eta_p)^{-1} $, and all the vertical maps are the natural projections, with Cartesian bottom square.
\end{lemma}

\begin{proof}
The commutativity of the triangle in the diagram is obvious. The horizontal arrows are closed immersions because of Lemma \ref{lemma_pullback_opencompacts} and the equality \[ u\, \mathcal{U}^{(n+1)} u^{-1} \cap \iota(\cO_{B,f}^\times) = u \, \mathcal{U}^{(n)}_\circ u^{-1} \cap \iota(\cO_{B,f}^\times).\] The claim that the bottom square is Cartesian thus follows from the fact that the degrees of the vertical maps $\varpi_{Z,n} $ and $\varpi_n^{E}$ are both $p^3$ (in particular, they agree). Note that both equalities, which can be checked directly, follow from \cite[Lemma 4.4.1]{LoefflerSpherical} as $u \in \GL_2^3(\Zp)$ satisfies the hypotheses (A), (B) of \cite[\S 4.3]{LoefflerSpherical}, with $\mathcal{F}=\GL_2^3/\overline{Q}$,  $Q^0_H = \GL_2$, and $\overline{Q}^0_G=Z \cdot \overline{N}$, where $Z$ and $\overline{N}$ denote the center of $\GL_2$ and the unipotent radical of $\overline{Q}$ respectively. 
\end{proof}

\subsubsection{} Let 
$${\rm pr}^{\mathbf{n}+\mathbf{r}}_{\mathbf{n}}:{\rm CH}^2(\mathbf{X}_{\mathcal{U}_{\mathbf{n+r}}}) \to {\rm CH}^2(\mathbf{X}_{\mathcal{U}_{\mathbf{n}}})\,,\qquad \mathbf{n}=(n_1,n_2,n_3), \, \mathbf{r}=(r_1,r_2,r_3)\in \ZZ_{\geq 0}^3$$
denote the pushforward maps induced from the natural degeneracy map $\mathbf{X}_{\mathcal{U}_{\mathbf{n+r}}}\xrightarrow{} \mathbf{X}_{\mathcal{U}_{\mathbf{n}}}$. Let us also denote by 
$$\mathbf{1}_{n}= \sum_{\mathcal{C} \in \pi_0(X_{Z,n})} [\mathcal{C}] \in {\rm CH}^0(X_{Z,n})$$  
the fundamental cycle of $X_{Z,n}$, where $\pi_0(X_{Z,n})$ denotes the set of connected components of $X_{Z,n}$. Moreover, we let $\iota(\eta_p^n)_\star$ be the pushforward of the isomorphism $\mathcal{T}_{\iota(\eta_p^n)}: \mathbf{X}_{\mathcal{U}^{(n)}} \to \mathbf{X}_{\mathcal{U}_{(n)}}$.

\begin{defn}\label{def_2025_07_03_1206}
    \item[i)] For $n \in \mathbb N$, we define $\Delta_{\mathcal{U}_{(n)}} := \iota(\eta_p^n)_\ast \circ \iota_{n,\ast}^u(\mathbf{1}_{n}) \in {\rm CH}^2(\mathbf{X}_{\mathcal{U}_{(n)}})$.        
    \item[ii)]    For every ${\mathbf{n} } = (n_1,n_2,n_3) \in \mathbb N^3$, we define $\Delta_{\mathcal{U}_{\mathbf{n}}} : = {\rm pr}^{(n)}_{\mathbf{n}} \Delta_{\mathcal{U}_{(n)}} $, where $n = \max \{ n_1,n_2,n_3 \}$.
\end{defn}

\begin{remark}
\label{remark_2025_03_2025_1427}
By definition, the class $\Delta_{\mathcal{U}_{(n)}} $  is represented by the image of the codimension-2 cycle $X_{Z,n}$ inside $\mathbf{X}_{\mathcal{U}_{(n)}}$, via the embedding  given by $\mathcal{T}_{\iota(\eta_p^n)} \,\circ\, \mathcal{T}_{u} \,\circ \,\iota_{U_{Z,n}, u \mathcal{U}^{(n)} u^{-1}}$. Note that this embedding sends $ [(\tau,g)] \in X_{Z,n}$, with $\tau \in \CC \smallsetminus \mathbb{R}$ and $g \in \widehat{B}^{ \times}$, to the element
$$
[(\tau, \tau,\tau), \left( g \left( \begin{smallmatrix} p^n & 1 \\ & 1 \end{smallmatrix} \right),  g \left( \begin{smallmatrix}
    p^n &  \\  & 1 
\end{smallmatrix} \right), g  \left( \begin{smallmatrix}
  & 1 \\ -p^n  &   
\end{smallmatrix} \right) \right)  ] \in \mathbf{X}_{\mathcal{U}_{(n)}}.
$$
Our choice of the twisting of the diagonal embedding is therefore concurrent with Hsieh's construction of diagonal theta elements in the definite case (\textit{cf}. Proposition \ref{prop_rewriting_Hsieh}).
\end{remark}

\begin{proposition}\label{prop:norm_comp_indefinite}
 For every triple ${\mathbf{n} } = (n_1, n_2, n_3)$ of positive integers, let us put 
 $$\mathbf{n}+\mathbf{1} =  (n_1 +1 , n_2 +1, n_3+1)\,.$$
 We then have 
 \[ {\rm pr}^{\mathbf{n}+\mathbf{1}}_{\mathbf{n}}(\Delta_{\mathcal{U}_{\mathbf{n+1}}}) = \mathbb{U}_p' \, \Delta_{\mathcal{U}_{\mathbf{n}}}.\] 
\end{proposition}
\begin{proof}
   This follows from \cite[Proposition 4.5.2]{LoefflerSpherical}, and we give a sketch of the proof for the convenience of the reader. We first claim that it is enough to verify the statement when ${\mathbf{n} } = (n)$. Indeed, if we suppose that and let $n = \max\{ n_1,n_2,n_3 \}$, we have $$ {\rm pr}^{\mathbf{n}+\mathbf{1}}_{\mathbf{n}} \circ {\rm pr}^{(n + 1)}_{\mathbf{n}+\mathbf{1}} =  {\rm pr}^{(n)}_{\mathbf{n}} \circ {\rm pr}^{(n + 1)}_{(n)}.$$  This implies that
 \begin{align*}
 {\rm pr}^{\mathbf{n}+\mathbf{1}}_{\mathbf{n}}(\Delta_{\mathcal{U}_{\mathbf{n+1}}}) &=  {\rm pr}^{\mathbf{n}+\mathbf{1}}_{\mathbf{n}} \circ {\rm pr}^{(n + 1)}_{\mathbf{n}+\mathbf{1}}(\Delta_{\mathcal{U}_{(n+1)}})   \\ &=  {\rm pr}^{(n)}_{\mathbf{n}} \circ {\rm pr}^{(n + 1)}_{(n)}(\Delta_{\mathcal{U}_{(n+1)}}) \\ &=  {\rm pr}^{(n)}_{\mathbf{n}} \circ \mathbb{U}_p' ( \Delta_{\mathcal{U}_{(n)}} ) \\ 
 &= \mathbb{U}_p' \circ {\rm pr}^{(n)}_{\mathbf{n}}  ( \Delta_{\mathcal{U}_{(n)}} )\\
  &= \mathbb{U}_p' \,  \Delta_{\mathcal{U}_{\mathbf{n}}}, 
 \end{align*} 
  where the fourth equality follows from the fact that  $\mathbb{U}_p' $ commutes with the pushforward maps ${\rm pr}^{(n)}_{\mathbf{n}}$ (cf. \cite{LoefflerSpherical}, Proposition 4.5.3).

  We now prove the formula in the case where ${\mathbf{n} } = (n)$. By Lemma \ref{lemma_cartesian_diagram}, 
   \[ (\varpi_n^{E})^\ast \circ \iota_{n,\ast}^u = (\mu_n^{E})_\ast\, \circ \,\iota_{n+1,\ast}^u\, \circ\,  \varpi_{Z,n}^{\ast}.\]
As the formation of the fundamental class $\mathbf{1}_n$ is compatible with base-change, we have
\[ (\varpi_n^{E})^* \circ \iota_{n,*}^u(\mathbf{1}_n) = (\mu_n^{E})_* \,\circ\, \iota_{n+1,*}^u (\mathbf{1}_{n+1}).\]
Applying the pushforward of \[\varpi_{2,n}^{E}:\mathbf{X}_{\mathcal{U}_{\circ}^{(n)}} \xrightarrow{\,\mathcal{T}_{\iota(\eta_p)}\,} \mathbf{X}_{\iota(\eta_p^{-1})\mathcal{U}_{\circ}^{(n)} \iota(\eta_p)} \xrightarrow{\,{\rm pr}\,} \mathbf{X}_{\mathcal{U}^{(n)}}\,, \] 
where ${\rm pr}$ is the natural degeneracy map, we deduce that 
\[ \varpi_{2,n,\ast}^{E} \circ (\varpi_n^{E})^* \,\circ\, \iota_{n,*}^u(\mathbf{1}_n) = \varpi_{2,n,\ast}^{E} \,\circ \,(\mu_n^{E})_* \,\circ\,\iota_{n+1,*}^u (\mathbf{1}_{n+1}).\]
Since $\mathbb{U}_p' = \varpi_{2,n,\ast}^{E}\, \circ \,(\varpi_n^{E})^*$, the final equality reads  
\[ \mathbb{U}_p' \, \iota_{n,*}^u(\mathbf{1}_n) = \varpi_{2,n, \ast}^{E} \circ (\mu_n^{E})_* \circ \iota_{n+1,*}^u (\mathbf{1}_{n+1}).\]
We now apply $\iota(\eta_p^n)_\ast$ to both sides on the left. While this commutes with $\mathbb{U}_p' $, an easy matrix computation shows that \[ \iota(\eta_p^n)_\ast \,\circ\,   \varpi_{2,n, \ast}^{E}\, \circ \,(\mu_n^{E})_* =   {\rm pr}^{(n+1)}_{(n)} \,\circ \,\iota(\eta_p^{n+1})_\ast \,. \]   
This shows that 
\[ \mathbb{U}_p' \circ \iota(\eta_p^n)_\ast \,\circ \,\iota_{n,*}^u(\mathbf{1}_n) = {\rm pr}^{(n+1)}_{(n)}\, \circ \,\iota(\eta_p^{n+1})_\ast \, \circ\, \iota_{n+1,*}^u (\mathbf{1}_{n+1})\,, \]
as desired.
\end{proof}

For every ${\mathbf{n} } = (n_1,n_2,n_3) \in \mathbb N^3$, recall that we have set 
$$\mathbf{X}_{\mathcal{U}_{\mathbf{n}}}  \xrightarrow{\,{\rm b}_{\mathbf{n}}\,} \mathbf{X}_{\mathbf{n}}=\mathbf{X}_{n_1,n_2,n_3} := X({U_{n_1}}) \times X({U_{n_2}}) \times X({U_{n_3}})$$ to denote the natural projection map. 
\begin{defn}
    We put 
 $$\Delta_{\mathbf{n}} : = {\rm b}_{\mathbf{n},\ast} \Delta_{\mathcal{U}_{\mathbf{n}}},$$
 and we let (with a slight abuse of notation)
$${\rm pr}^{\mathbf{n}+\mathbf{r}}_{\mathbf{n}}\,:\, {\rm CH}^2(\mathbf{X}_{\mathbf{n+r}}) \lra {\rm CH}^2(\mathbf{X}_{\mathbf{n}})$$ 
denote the pushforward of the natural degeneracy map.
\end{defn}
\begin{corollary}
\label{cor_2025_07_02_1244}
    For every triple ${\mathbf{n} } = (n_1, n_2, n_3)$ of positive integers, we have 
 \[ {\rm pr}^{\mathbf{n}+\mathbf{1}}_{\mathbf{n}}(\Delta_{{\mathbf{n + 1} }}) = \mathbb{U}_p' \, \Delta_{\mathbf{n} }.\] 
\end{corollary}
\begin{proof}
    Since ${\rm pr}^{\mathbf{n}+\mathbf{1}}_{\mathbf{n}} \circ  {\rm b}_{\mathbf{n}+\mathbf{1},\ast} ={\rm b}_{\mathbf{n},\ast} \circ {\rm pr}^{\mathbf{n}+\mathbf{1}}_{\mathbf{n}}$, Proposition \ref{prop:norm_comp_indefinite} yields
    \begin{align*}
 {\rm pr}^{\mathbf{n}+\mathbf{1}}_{\mathbf{n}}(\Delta_{{\mathbf{n + 1} }}) =  {\rm pr}^{\mathbf{n}+\mathbf{1}}_{\mathbf{n}} \circ  {\rm b}_{\mathbf{n}+\mathbf{1},\ast}(\Delta_{\mathcal{U}_{\mathbf{n+1}}}) = {\rm b}_{\mathbf{n},\ast} \circ {\rm pr}^{\mathbf{n}+\mathbf{1}}_{\mathbf{n}}(\Delta_{\mathcal{U}_{\mathbf{n+1}}}) = {\rm b}_{\mathbf{n},\ast} \circ \mathbb{U}_p' ( \Delta_{\mathcal{U}_{\mathbf{n}}} ) &=\mathbb{U}_p' \circ {\rm b}_{\mathbf{n},\ast}  ( \Delta_{\mathcal{U}_{\mathbf{n}}} )\\
  &= \mathbb{U}_p' \,  \Delta_{\mathbf{n}}, 
 \end{align*} 
 where the penultimate equality follows from Proposition \ref{prop_2025_05_15_0828}.
\end{proof}

\subsection{\texorpdfstring{$p$-adic\,}\ cycle class maps} 
Given an algebraic variety $Y$ defined over $\QQ$, let us put $\overline{Y}:={Y}\times_\QQ \overline{\QQ}$. Suppose that $X$ is an algebraic threefold defined over $\QQ$. We denote by
$${\rm CH}^2_0(X; K) := \ker \left( {\rm CH}^2(X;K) \xrightarrow{{\rm cl}_0} H^4_{\etale} \left(\overline{X}, \ZZ_p(2)\right)^{G_K}\right)$$ 
the group of null-homologous algebraic cycles defined over an algebraic number field $K$ (which is the kernel of the cycle class map ${\rm cl}_0$). We let ${\rm CH}^2_0(X;K)_{\ZZ_p}$ denote its $p$-adic completion. 

Since the target of ${\rm cl}_0$ is $p$-adically complete, ${\rm cl}_0$ factors through the $p$-adic completion ${\rm CH}^2(X;K)_{\ZZ_p}$ of ${\rm CH}^2(X;K)$, and we denote the induced map on  ${\rm CH}^2(X;K)_{\ZZ_p}$ still by   ${\rm cl}_0$. Moreover, since  the $\ZZ$-module $H^4_{\etale} \left(\overline{X}, \ZZ_p(2)\right)$ is torsion-free\footnote{Note that $\ZZ$ acts on $H^4_{\etale} (\overline{X}, \ZZ_p(2))$ via its image under the canonical injection $\ZZ\hookrightarrow \ZZ_p$.}, it follows that ${\rm im}({\rm cl}_0)$ is also torsion-free, hence it is flat. It then follows from \cite[\href{https://stacks.math.columbia.edu/tag/0315}{Lemma 0315}]{stacks-project} that 
$${\rm CH}^2_0(X;K)_{\ZZ_p}=\ker \left( {\rm CH}^2(X;K)_{\ZZ_p} \xrightarrow{{\rm cl}_0} H^4_{\etale} \left(\overline{X}, \ZZ_p(2)\right)\right)\,.$$

\subsubsection{} We return to the setting of our paper. Let $\TT^{B}_{n}$ denote the $p$-adic completion of the Hecke algebra $\mathrm{T}_{B}(Np^n)$, acting on $H^\ast_{\textup{\'{e}t}}({X}(U_n)\times_{\QQ}\overline{\QQ},\ZZ_p)$. Then $\TT^{B}_{n}$ is a finite $\ZZ_p$-algebra and it follows from \cite[\S7.2, Lemma 1]{Hida_Elementary} that $e_{{\rm U}_p'}:=\lim ({\rm U}_p')^{n!}$ belongs to $\TT^{B}_{n}$. Similarly, we denote by $\TT^{B_E}_{\mathbf{n}}$ the $p$-adic completion of $\mathrm{T}_{B_E}(Np^{\mathbf{n}})$, which is isomorphic to the completed tensor product of the $\TT^{B}_{n_i}$'s, where $\mathbf{n}=(n_1,n_2,n_3)$.

Note that $\TT^{B_E}_{\mathbf{n}}$ acts on ${\rm CH}^2(\mathbf{X}_{\mathbf{n}};K)_{\ZZ_p}$, and as a result, so does $(e_{{\rm U}_p'},{\rm id},e_{{\rm U}_p'})$. This allows us to define
\begin{equation}
    \label{eqn_2023_11_08_1351_appendix}
    \Delta_{\mathbf{n}} ^{\circ}:=(e_{{\rm U}_p'},{\rm T}_{q}'-(q+1),e_{{\rm U}_p'}) \Delta_{\mathbf{n}}\in {\rm CH}^2(\mathbf{X}_{\mathbf{n}};\QQ)_{\ZZ_p}\,,
\end{equation}
where $q\nmid Np$ is an auxiliary prime (whose choice will not concern us for the moment).

\begin{lemma}
    \label{DR_Lemma_2_5_appendix}
    We have $\Delta_{\mathbf{n}} ^{\circ}\in {\rm CH}^2_0(\mathbf{X}_{\mathbf{n}};\QQ)_{\ZZ_p}$. 
\end{lemma}

\begin{proof}
The argument we present closely follows the proof of \cite[Lemma 2.5]{DarmonRotger}. The correspondence $\mathrm{U}_{p}'$ (resp. $\mathrm{T}'_q$) acts as multiplication by $p$ (resp. $q+1$) on 
$H^2_{\etale}({X}(U_{n_i})\times_{\QQ}\overline{\QQ},\ZZ_p(2))$. 
Therefore Hida's idempotent $e_{\mathrm{U}_{p}'}$ and $\mathrm{T}'_q-(q+1)$ annihilate this module. As a result, $(e_{\mathrm{U}_{p}'}, \mathrm{T}_{q}'-(q+1), e_{\mathrm{U}_{p}'})$  annihilates all the direct summands that appear in the K\"unneth decomposition of $H^4_{\etale}(\mathbf{X}_{\mathbf{n}}\times_{\QQ}\overline{\QQ},\ZZ_p(2))$, therefore also $H^4_{\etale}(\mathbf{X}_{\mathbf{n}}\times_{\QQ}\overline{\QQ},\ZZ_p(2))$ itself, which is the target of the Hecke equivariant map ${\rm cl}_0$ (defined as in \cite[\S1.2]{nekovarbanff} and denoted by $cl_{0,X}$ in op. cit.). This concludes the proof of our lemma. 
\end{proof}

\subsubsection{} Thanks to Lemma~\ref{DR_Lemma_2_5_appendix}, we may put 
\begin{equation}
\label{eqn_2023_11_08_1512_bis}
   \Delta_{\mathbf{n}}^{\etale} := {\rm cl}_1(\Delta_{{\mathbf{n}}}^{\circ})\in H^1(\Q,  H^3_{\etale} (\mathbf{X}_{\mathbf{n}}\times_{\QQ} \overline{\QQ}, \ZZ_p(2)))\,.
\end{equation} 
By a slight abuse of notation, we will also denote by $\Delta_{\mathbf{n}}^{\etale}$ the image of this cohomology class under the compositum
\small 
\begin{align*}
&(e_{\mathrm{U}_{p}'}, \mathrm{T}_{q}'-(q+1), e_{\mathrm{U}_{p}'})\cdot {\rm CH}^2(\mathbf{X}_{\mathbf{n}};\QQ)_{\ZZ_p}\,\subset {\rm CH}^2_0(\mathbf{X}_{\mathbf{n}};\QQ)_{\ZZ_p}\xrightarrow{{\rm cl}_1}  H^1(\QQ,  (e_{\mathrm{U}_{p}'}, {\rm id}, e_{\mathrm{U}_{p}'})\cdot  H^3_{\etale} (\mathbf{X}_{\mathbf{n}}\times_{\QQ}
\overline{\QQ}, \ZZ_p(2)))\\
&\hspace{1.8cm}\xrightarrow{\,\,\texttt{K\"u}\,\,} H^1(\QQ,  e_{\mathrm{U}_{p}'} H^1_{\etale} (X(U_{n_1})\times_{\QQ}
\overline{\QQ}, \ZZ_p)\otimes H^1_{\etale} (X(U_{n_2})\times_{\QQ}
\overline{\QQ}, \ZZ_p)\otimes e_{\mathrm{U}_{p}'} H^1_{\etale} (X(U_{n_3})\times_{\QQ}
\overline{\QQ}, \ZZ_p)(2))
\end{align*}
\normalsize
where the final projection $\texttt{K\"u}$ is induced from the K\"unneth decomposition.

\subsubsection{} 
\label{subsubsec_2025_07_02_1527}
Let $g \in S_2^B(U_1(N))$ be an eigenform that is residually irreducible at $p$.  
Let us choose forever a prime $q$ so that $C_q:=a_q(g)-(q+1)$ is a $p$-adic unit (which is possible thanks to the residual irreducibility hypothesis). There exists a Galois equivariant surjective morphism 
$${\rm pr}_g\,:\,H^1_{\etale} (X(U_1(N))\times_{\QQ}
\overline{\QQ}, \QQ_p) \twoheadrightarrow V_g^*,$$
where $V_g^*$ is the $2$-dimensional $p$-adic Galois representation associated to $g$ (i.e. an irreducible $G_\QQ$-constituent of the maximal $g$-Hecke-isotypic quotient of $H^1_{\etale} (X(U_1(N))\times_{\QQ}
\overline{\QQ}, \QQ_p)$), and let $T_g^*$ denote the Galois-stable $\ZZ_p$-lattice obtained as the image of $H^1_{\etale} (X(U_1(N))\times_{\QQ}
\overline{\QQ}, \ZZ_p)$ under ${\rm pr}_g$. For a triple $\mathbf{n}=(n_1,n_2,n_3)$ of natural numbers, we define the class
$$\Delta_{\mathbf{n}}^{\ddagger}:=C_q^{-1}((\mathrm{U}_{p}')^{-n_1},{\rm id},(\mathrm{U}_{p}')^{-n_3})\Delta_{\mathbf{n}}^{\etale}\,.$$ 
We set
$$H^3_{\etale}(\mathbf{X}_{n_1,0,n_3}\times_{\QQ}
\overline{\QQ},\ZZ_p)^{\circ \,g\, \circ} := e_{\mathrm{U}_{p}'} H^1_{\etale} (X(U_{n_1})\times_{\QQ}
\overline{\QQ}, \ZZ_p)\otimes T_g^* \otimes e_{\mathrm{U}_{p}'} H^1_{\etale} (X(U_{n_3})\times_{\QQ}
\overline{\QQ}, \ZZ_p)$$
 and define the class 
$$\Delta_{n_1,0,n_3}^{\etale}(g)\in  H^1(\QQ,  H^3_{\etale}(\mathbf{X}_{n_1,0,n_3}\times_{\QQ}
\overline{\QQ},\ZZ_p)^{\circ \,g\, \circ}(2))$$ as the image of the class 
$\Delta_{n_1,0,n_3}^{\ddagger}$
under the morphism induced from 
\begin{align*}
    (e_{\mathrm{U}_{p}'}, {\rm id}, e_{\mathrm{U}_{p}'})\cdot  H^3_{\etale} (\mathbf{X}_{n_1,0,n_3}\times_{\QQ}
\overline{\QQ}, \ZZ_p(2)) \,\xrightarrow{({\rm id},\,{\rm pr}_g\,,\,{\rm id})}  H^3_{\etale}(\mathbf{X}_{n_1,0,n_3}\times_{\QQ}
\overline{\QQ},\ZZ_p)^{\circ \,g\, \circ}(2)\,.\end{align*}
composed with $\texttt{K\"u}$. Note then that $\Delta_{n_1,0,n_3}^{\etale}(g)$ does not depend on the choice of the auxiliary prime $q$.

\subsubsection{}
\label{subsubsec_2025_10_20_1513}
In our argument below to prove that diagonal cycles form a $Q$-system (in the sense of Perrin-Riou), we will also need the following modification: For $g$ as above and a pair $(n_1,n_3)$ of positive integers, let us define 
$$\widetilde{\Delta}_{n_1,0,n_3}^{\etale}(g)\in H^1(\QQ, H^3_{\etale}({\mathbf{X}}_{n_1,0,n_3}\times_\QQ \overline{\QQ},\ZZ_p)^{\circ\,g\,\circ}(2))$$ 
as the image of the class $\Delta_{n_1,1,n_3}^{\ddagger}$ under the morphism induced from
\begin{align*}
          (e_{\mathrm{U}_{p}'}, {\rm id},e_{\mathrm{U}_{p}'})\cdot H^3_{\etale}({\mathbf{X}}_{n_1,1,n_3}\times_\QQ\overline{\QQ},\ZZ_p)
        &\,\xrightarrow{({\rm id},\,{\,\rm pr}_g\,\circ\,\pi_{2,*},\,{\rm id})} H^3_{\etale}({\mathbf{X}}_{n_1,0,n_3}\times_\QQ\overline{\QQ},\ZZ_p)^{\circ\,g\,\circ}\,.
\end{align*}

\subsubsection{}

The following proposition will be used to establish that the collection of diagonal cycles forms a $Q$-system. In what follows, we shall write $\Delta_{n}^{\etale}(g)$ in place of $\Delta_{n,0,n}^{\etale}(g)$. We similarly put $\Delta_n^\ddagger:=\Delta_{n,0,n}^\ddagger$, $\Delta_{n}^{\etale}:=\Delta_{n,0,n}^{\etale}$, and $\widetilde{\Delta}_{n}^{\etale}(g):=\widetilde{\Delta}_{n,0,n}^{\etale}(g)$.

\begin{proposition}
\label{prop_2025_03_08_1203}
   The following relation holds for all positive integers $n$:
    \[
   (\pi_1, {\rm id},\pi_1)_*\,\Delta_{n+1}^{\etale}(g)=a_p(g) \Delta_{n} ^{\etale}(g)-\chi_g(p)\widetilde{\Delta}_{n}^{\etale}(g).
    \]
\end{proposition}

\begin{proof}
    This follows from a direct calculation:
    \begin{align}
     \notag    (\pi_1, {\rm id},\pi_1)_*\,\Delta_{n+1}^{\etale}(g)&= ({\rm id},\,{\rm pr}_g\,,\,{\rm id})\circ {\rm pr}^{n+1,0,n+1}_{n,0,n} \Delta_{n+1}^{\ddagger}\\
 \notag  &= C_q^{-1}({\rm id},\,{\rm pr}_g\,,\,{\rm id})\circ {\rm pr}^{n+1,0,n+1}_{n,0,n}\circ  (\mathrm{U}_{p}',{\rm id},\mathrm{U}_{p}')^{-n-1} \Delta_{n+1,0,n+1}^{\etale}\\
         \notag  &= C_q^{-1}({\rm id},\,{\rm pr}_g\,,\,{\rm id})\circ {\rm pr}^{n+1,0,n+1}_{n,0,n}\circ  (\mathrm{U}_{p}',{\rm id},\mathrm{U}_{p}')^{-n-1} {\rm pr}^{(n+1)}_{n+1,0,n+1}\Delta_{(n+1)}^{\etale}\\
         \label{eqn_2025_07_02_1501}  &= C_q^{-1}({\rm id},\,{\rm pr}_g\,,\,{\rm id})\circ {\rm pr}^{n+1,0,n+1}_{n,0,n}\circ {\rm pr}^{(n+1)}_{n+1,0,n+1} (\mathrm{U}_{p}',{\rm id},\mathrm{U}_{p}')^{-n-1}\Delta_{(n+1)}^{\etale}\\
            \notag &= C_q^{-1}({\rm id},\,{\rm pr}_g\,,\,{\rm id})\circ {\rm pr}^{(n+1)}_{n,0,n} (\mathrm{U}_{p}',{\rm id},\mathrm{U}_{p}')^{-n-1}\Delta_{(n+1)}^{\etale}\\
          \notag &= C_q^{-1}({\rm id},\,{\rm pr}_g\,,\,{\rm id})\circ {\rm pr}^{(n)}_{n,0,n}\circ {\rm pr}^{(n+1)}_{(n)} (\mathrm{U}_{p}',{\rm id},\mathrm{U}_{p}')^{-n-1}\Delta_{(n+1)}^{\etale}\\
           \label{eqn_2025_07_02_1502}   &= C_q^{-1}({\rm id},\,{\rm pr}_g\,,\,{\rm id})\circ {\rm pr}^{(n)}_{n,0,n} \circ (\mathrm{U}_{p}',{\rm id},\mathrm{U}_{p}')^{-n-1}\circ {\rm pr}^{(n+1)}_{(n)}\Delta_{(n+1)}^{\etale}\\
       \notag \hbox{(Corollary~\ref{cor_2025_07_02_1244})}\quad\quad\qquad   \notag    &= C_q^{-1}({\rm id},\,{\rm pr}_g\,,\,{\rm id})\circ {\rm pr}^{(n)}_{n,0,n} \circ (\mathrm{U}_{p}',{\rm id},\mathrm{U}_{p}')^{-n-1}(\mathrm{U}_{p}',\mathrm{U}_{p}',\mathrm{U}_{p}') \,\Delta_{(n)}^{\etale}\\
       \notag      &= C_q^{-1}({\rm id},\,{\rm pr}_g\,,\,{\rm id})\circ {\rm pr}^{(n)}_{n,0,n} \circ (\mathrm{U}_{p}',{\rm id},\mathrm{U}_{p}')^{-n}({\rm id},\mathrm{U}_{p}',{\rm id}) \,\Delta_{(n)}^{\etale}\\
        \notag     &= ({\rm id},\,{\rm pr}_g\,,\,{\rm id})\circ {\rm pr}^{(n)}_{n,0,n} \circ ({\rm id},\mathrm{U}_{p}',{\rm id}) \,\Delta_{(n)}^{\ddagger}\\
       \notag      &= ({\rm id},\,{\rm pr}_g\,,\,{\rm id})\circ {\rm pr}^{n,1,n}_{n,0,n} \circ {\rm pr}^{(n)}_{n,1,n} \circ ({\rm id},\mathrm{U}_{p}',{\rm id}) \,\Delta_{(n)}^{\ddagger}\\
       \label{eqn_2025_03_08_1450} &=({\rm id},\,{\rm pr}_g\,,\,{\rm id})\circ {\rm pr}^{n,1,n}_{n,0,n} \circ({\rm id},\mathrm{U}_{p}',{\rm id}) \circ {\rm pr}^{(n)}_{n,1,n}  \,\Delta_{(n)}^{\ddagger}\\
       \label{eqn_2025_03_08_1451} &=({\rm id},\,{\rm pr}_g\,,\,{\rm id})\circ ({\rm id},\mathrm{T}_{p}',{\rm id})\circ  {\rm pr}^{n,1,n}_{n,0,n} \circ {\rm pr}^{(n)}_{n,1,n}  \,\Delta_{(n)}^{\ddagger}\\
       \notag& \hspace{2cm} -({\rm id},\,{\rm pr}_g\,,\,{\rm id})\circ ({\rm id},\mathrm{S}_p^{-1},{\rm id})\circ ({\rm id},\pi_2,{\rm id})_*\circ {\rm pr}^{(n)}_{n,1,n}  \,\Delta_{(n)}^{\ddagger}\\
       \notag & =  ({\rm id},\,{\rm pr}_g\,,\,{\rm id})\circ ({\rm id},\mathrm{T}_{p}',{\rm id})\circ {\rm pr}^{(n)}_{n,0,n}  \,\Delta_{(n)}^{\ddagger} \\
       \notag &\hspace{4cm}- ({\rm id},\,{\rm pr}_g\,,\,{\rm id})\circ ({\rm id},\mathrm{S}_p^{-1},{\rm id})\circ ({\rm id},\pi_2,{\rm id})_*\circ {\rm pr}^{(n)}_{n,1,n}  \,\Delta_{(n)}^{\ddagger}\\
      \notag & =  a_p(g)\cdot \Delta_{n}^{\etale}(g)- \chi_g(p)\cdot  \,\widetilde{\Delta}_{n}^{\etale}(g)\,.
    \end{align}
    Here, \eqref{eqn_2025_07_02_1501}, \eqref{eqn_2025_07_02_1502}, and \eqref{eqn_2025_03_08_1450} are valid since 
    ${\rm pr}^j_{j-1}\,\mathrm{U}_{p}'=\mathrm{U}_{p}'\,{\rm pr}^j_{j-1}$ whenever $j>1$, whereas \eqref{eqn_2025_03_08_1451} holds because ${\rm pr}^1_{0}\,\mathrm{U}_{p}'=\mathrm{T}_{p}'\,{\rm pr}^1_{0}-\mathrm{S}_p^{-1}\,(\pi_2)_*$ (cf. Proposition~\ref{Prop_of_KLZ17}).
\end{proof}

\begin{remark}
    A more general form of the asserted identity in Proposition~\ref{prop_2025_03_08_1203} holds: for any pair $(n_1,n_3)$ of positive integers, we have  
    \[
   (\pi_1, {\rm id},\pi_1)_*\,\Delta_{n_1+1,0,n_3+1}^{\etale}(g)=a_p(g) \Delta_{n_1,0,n_3} ^{\etale}(g)-\chi_g(p)\widetilde{\Delta}_{n_1,0,n_3}^{\etale}(g)\,.
    \]
    The proof of this generalization is virtually identical to that of Proposition~\ref{prop_2025_03_08_1203}. Since we do not need this extension in what follows, we leave the details to the interested reader.
\end{remark}

\begin{lemma}
    \label{lemma_2025_03_08_1548}
    Let $n\geq 2$ be an integer. Then,
    $$(\pi_1,{\rm id},\pi_1)_*\,\widetilde{\Delta}_{n}^{\etale}(g)=p\cdot {\Delta}_{n-1}^{\etale}(g)\,.$$
\end{lemma}

\begin{proof}
    The proof of this lemma is similar to that of Proposition~\ref{prop_2025_03_08_1203}, and it follows from a direct calculation:
    \begin{align}
     \notag    (\pi_1, {\rm id},\pi_1)_*\,\widetilde{\Delta}_{n}^{\etale}(g)&= ({\rm id},\,{\rm pr}_g\,,\,{\rm id})\circ {\rm pr}^{n,0,n}_{n-1,0,n-1}\circ ({\rm id},\pi_2,{\rm id})_*\circ {\rm pr}^{(n)}_{n,1,n} \,\Delta_{(n)}^{\ddagger}\\
         \notag  &= C_q^{-1} ({\rm id},\,{\rm pr}_g\,,\,{\rm id})\circ {\rm pr}^{n,0,n}_{n-1,0,n-1}\circ ({\rm id},\pi_2,{\rm id})_*\circ {\rm pr}^{(n)}_{n,1,n} \circ (\mathrm{U}_{p}',{\rm id},\mathrm{U}_{p}')^{-n}\Delta_{(n)}^{\etale} \\
          \notag  &= C_q^{-1} ({\rm id},\,{\rm pr}_g\,,\,{\rm id})\circ ({\rm id},\pi_2,{\rm id})_*\circ {\rm pr}^{n,1,n}_{n-1,1,n-1} \circ  {\rm pr}^{(n)}_{n,1,n} \circ (\mathrm{U}_{p}',{\rm id},\mathrm{U}_{p}')^{-n}\Delta_{(n)}^{\etale} \\
         \notag  &= C_q^{-1} ({\rm id},\,{\rm pr}_g\,,\,{\rm id})\circ  ({\rm id},\pi_2,{\rm id})_*\circ {\rm pr}^{(n-1)}_{n-1,1,n-1} \circ {\rm pr}^{(n)}_{(n-1)}\circ \, (\mathrm{U}_{p}',{\rm id},\mathrm{U}_{p}')^{-n}\Delta_{(n)}^{\etale} \\
         \notag  &= C_q^{-1} ({\rm id},\,{\rm pr}_g\,,\,{\rm id})\circ  ({\rm id},\pi_2,{\rm id})_*\circ {\rm pr}^{(n-1)}_{n-1,1,n-1} \circ \, (\mathrm{U}_{p}',{\rm id},\mathrm{U}_{p}')^{-n}\,{\rm pr}^{(n)}_{(n-1)}\,\Delta_{(n)}^{\etale} \\
     \notag \hbox{(Corollary~\ref{cor_2025_07_02_1244})}\qquad  &= C_q^{-1} ({\rm id},\,{\rm pr}_g\,,\,{\rm id})\circ  ({\rm id},\pi_2,{\rm id})_*\circ {\rm pr}^{(n-1)}_{n-1,1,n-1} \circ \, (\mathrm{U}_{p}',{\rm id},\mathrm{U}_{p}')^{-n}(\mathrm{U}_{p}',\mathrm{U}_{p}',\mathrm{U}_{p}')\,\Delta_{(n-1)}^{\etale} \\
     \notag   &= C_q^{-1} ({\rm id},\,{\rm pr}_g\,,\,{\rm id})\circ  ({\rm id},\pi_2,{\rm id})_*\circ {\rm pr}^{(n-1)}_{n-1,1,n-1} \circ \, (\mathrm{U}_{p}',{\rm id},\mathrm{U}_{p}')^{-n+1}({\rm id},\mathrm{U}_{p}',{\rm id})\,\Delta_{(n-1)}^{\etale} \\
     \notag   &= ({\rm id},\,{\rm pr}_g\,,\,{\rm id})\circ  ({\rm id},\pi_2,{\rm id})_*\circ {\rm pr}^{(n-1)}_{n-1,1,n-1} \circ \, ({\rm id},\mathrm{U}_{p}',{\rm id})\,\Delta_{(n-1)}^{\ddagger} \\
      \notag   &= ({\rm id},\,{\rm pr}_g\,,\,{\rm id})\circ  ({\rm id},\pi_2,{\rm id})_*\circ ({\rm id},\mathrm{U}_{p}',{\rm id})\circ {\rm pr}^{(n-1)}_{n-1,1,n-1} \,\Delta_{(n-1)}^{\ddagger}\\
      \notag   &=p\cdot  ({\rm id},\,{\rm pr}_g\,,\,{\rm id})\circ   ({\rm id},\pi_1,{\rm id})_*\circ  {\rm pr}^{(n-1)}_{n-1,1,n-1} \,\Delta_{(n-1)}^{\ddagger}\\
      \notag   &=p\cdot  ({\rm id},\,{\rm pr}_g\,,\,{\rm id})\circ    {\rm pr}^{(n-1)}_{n-1,0,n-1} \,\Delta_{(n-1)}^{\ddagger}= p\cdot \Delta_{n-1}^{\etale}(g)\,.
    \end{align}
    Here, the fifth and ninth equalities are valid since    ${\rm pr}^j_{j-1}\,\mathrm{U}_{p}'=\mathrm{U}_{p}'\,{\rm pr}^j_{j-1}$ whenever $j>1$, while the 10th holds because $\pi_{2,*}  \,\mathrm{U}_{p}'=p\, \pi_{1,*}$.
\end{proof}

\subsubsection{}
Let $f$ be an eigenform on $B$ of level $U_{1}$ given as in \eqref{eqn_new_level}. We assume that $f$ is $\mathrm{U}_p'$-ordinary and we define the Carayol--Deligne Galois representation $V_f^*$ (and the lattice $T_f^*\subset V_f^*$) as well as a projection ${\rm pr}_f$ in a manner analogous to \S\ref{subsubsec_2025_07_02_1527}. 
We set 
$$H^3_{\etale}(\mathbf{X}_{1,0,n}\times_{\QQ}
\overline{\QQ},\ZZ_p)^{f \,g\, \circ} := T_f^*\otimes T_g^* \otimes e_{\mathrm{U}_{p}'} H^1_{\etale} (X(U_{n})\times_{\QQ}
\overline{\QQ}, \ZZ_p)\,.$$
\begin{defn}
    \label{defn_2025_03_21_2027}
    We set 
    $$\Delta_{n}^{\etale}(f,g):=({\rm pr}_{f},\id,\id)\circ ({\rm pr}^{n,0,n}_{1,0,n})\,\,\Delta_{n}^{\etale}(g)\in  H^1(\QQ, H^3_{\etale}(\mathbf{X}_{1,0,n}\times_{\QQ}
\overline{\QQ},\ZZ_p)^{f \,g\, \circ} (2))$$
    $$\widetilde{\Delta}_{n}^{\etale}(f,g):=({\rm pr}_{f},\id,\id)\circ ({\rm pr}^{n,0,n}_{1,0,n})\,\,\widetilde{\Delta}_{n}^{\etale}(g)\in  H^1(\QQ, H^3_{\etale}(\mathbf{X}_{1,0,n}\times_{\QQ}
\overline{\QQ},\ZZ_p)^{f \,g\, \circ} (2))\,.$$
\end{defn}

\begin{corollary}
        \label{cor_2025_03_22_0955}
    For any positive integer $n$, we have
    \begin{equation}
    \label{eqn_2025_03_22_1111}
        ({\rm id}, {\rm id},\pi_1)_*\,\Delta_{n+1}^{\etale}(f, g)=a_p(g) \Delta_{n}^{\etale}(f,g)-\chi_g(p)\widetilde{\Delta}_{n}^{\etale}(f,g)\,,
    \end{equation}
        \begin{equation}
    \label{eqn_2025_03_22_1112}
        ({\rm id}, {\rm id},\pi_1)_*\,\widetilde\Delta_{n+1}^{\etale}(f, g)=p\cdot \Delta_{n} ^{\etale}(f,g)\,.
    \end{equation}
\end{corollary}

\begin{proof}
    We apply $({\rm pr}_{f},\id,\id)\circ ({\rm pr}^{n,0,n}_{1,0,n})$ to both sides of the asserted equality in Proposition~\ref{prop_2025_03_08_1203}, to see that
    \begin{align*}
        ({\rm pr}_{f},\id,\id)\circ ({\rm pr}^{n,0,n}_{1,0,n})\circ(\pi_1, {\rm id},\pi_1)_*\,\Delta_{n+1}^{\etale}(g)&\,=a_p(g)({\rm pr}_{f},\id,\id)\circ ({\rm pr}^{n,0,n}_{1,0,n})\, \Delta_{n}^{\etale}(g)\\
        &\qquad\qquad\qquad \qquad  -\chi_g(p)({\rm pr}_{f},\id,\id)\circ ({\rm pr}^{n,0,n}_{1,0,n})\,\widetilde{\Delta}_{n}^{\etale}(g)\\
        &\,=a_p(g) \Delta_{n}^{\etale}(f,g)-\chi_g(p)\widetilde{\Delta}_{n}^{\etale}(f,g).
    \end{align*}
    The proof of \eqref{eqn_2025_03_22_1111} follows from noting that 
    \begin{align*}
    ({\rm pr}_{f},\id,\id)\circ ({\rm pr}^{n,0,n}_{1,0,n})\circ(\pi_1, {\rm id},\pi_1)_*\,\Delta_{n+1}^{\etale}(g)&=\, ({\rm pr}_{f},\id,\id)\circ({\rm id}, {\rm id},\pi_1)_*\,\circ \, ({\rm pr}^{n+1,0,n+1}_{1,0,n+1})\,\Delta_{n+1}^{\etale}(g)\\
    &\,=({\rm id}, {\rm id},\pi_1)_*\,\circ\,({\rm pr}_{f},\id,\id)\circ \, ({\rm pr}^{n+1,0,n+1}_{1,0,n+1})\,\Delta_{n+1}^{\etale}(g)\\
    &\,=:({\rm id}, {\rm id},\pi_1)_*\,\circ\,\Delta_{n+1}^{\etale}(f,g)\,.
        \end{align*}
        The proof of \eqref{eqn_2025_03_22_1112} can be deduced from Lemma~\ref{lemma_2025_03_08_1548} in a similar fashion.
\end{proof}

\subsubsection{}\label{sec:specialization} We recall from \S\ref{subsubsec_2027_07_02_1557} the diamond operator $\langle x, y, z \rangle$ attached to the element $\left(\left( \begin{smallmatrix}
    1 & \\ & x
\end{smallmatrix}\right),\left( \begin{smallmatrix}
    1 & \\ & y
\end{smallmatrix}\right),\left( \begin{smallmatrix}
    1 & \\ & z
\end{smallmatrix}\right)\right)$ for $x,y,z\in \ZZ_p^\times$. 
\begin{proposition}
        \label{prop_2025_03_08_1713}
Let $n\geq 2$ be an integer. For any integer $d\equiv 1 \mod p^{n-1}$, we have
    $$\langle 1,1,d\rangle\,\widetilde{\Delta}_{n}^{\etale}(f, g)=\widetilde{\Delta}_{n}^{\etale}(f,g)\,.$$
\end{proposition}
\begin{proof}
    We recall that 
$$\widetilde{\Delta}_{n}^{\etale}(f,g)=({\rm pr}_f,\,{\rm pr}_g\,,\,{\rm id})\circ ({\rm id},\pi_2,{\rm id})_*\circ {\rm pr}^{(n)}_{1,1,n}  \,\Delta_{(n)}^{\ddagger}\,.$$
Hence, by obvious commutation relations (including $\pi_{2,*}\circ \pi_{1,*}=\pi_{1,*}\circ\pi_{2,*}$), we have
\begin{align}
\label{eqn_2025_03_09_1004}
\begin{aligned}
    \langle 1,1,d^{-1}\rangle\,\widetilde{\Delta}_{n}^{\etale}(f,g) =({\rm pr}_f,\,{\rm pr}_g\,,\,{\rm id})\circ ({\rm id},\pi_1,{\rm id})_*\circ \,{\rm pr}^{n-1,n-1,n}_{1,1,n}\circ \langle 1,1,d^{-1}\rangle \circ (\pi_1,\pi_2,{\rm id})_*\,\Delta_{(n)}^{\ddagger}\,.
\end{aligned}
\end{align}
Recall also that 
$$\Delta_{(n)}^{\ddagger}:=C_q^{-1}(\mathrm{U}_{p}',{\rm id},\mathrm{U}_{p}')^{-n}\Delta_{(n)}^{\etale}\,,$$
so that, since $({\rm U}_p', {\rm id},{\rm U}_p')$ commutes with diamond action and $(\pi_1,\pi_2,{\rm id})_*$ (note that $n\geq 2$),
\begin{equation}
\label{eqn_2025_03_09_1005}
\langle 1,1,d^{-1}\rangle \circ (\pi_1,\pi_2,{\rm id})_* \,\Delta_{(n)}^{\ddagger}=C_q^{-1}(\mathrm{U}_{p}',{\rm id},\mathrm{U}_{p}')^{-n}\langle 1,1,d^{-1}\rangle  \circ (\pi_1,\pi_2,{\rm id})_*\,\Delta_{(n)}^{\etale}\,.
\end{equation}
Combining \eqref{eqn_2025_03_09_1004} and \eqref{eqn_2025_03_09_1005}, it therefore suffices (in order to check the claimed equality in the first part of our proposition) to prove that
\begin{equation}
\label{eqn_2025_03_22_1357}
     \langle 1,1,d^{-1}\rangle \circ (\pi_1,\pi_2,{\rm id})_*\,\Delta_{(n)}^{\etale}=(\pi_1,\pi_2,{\rm id})_*\,\Delta_{(n)}^{\etale}\,.
\end{equation}
Furthermore, as we have (by definition)
\begin{equation}
\label{eqn_2025_03_22_1414}
    \Delta_{(n)}^{\etale}:= 
{\rm cl}_1 \circ (e_{\mathrm{U}_{p}'}, \mathrm{T}_{q}'-(q+1), e_{\mathrm{U}_{p}'}) \cdot\,\Delta_{(n)} 
\end{equation}
and $\langle 1,1,d^{-1}\rangle \circ (\pi_1,\pi_2,{\rm id})_*$ commutes with the
cycle class map ${\rm cl}_1$ as well as the ordinary projector (we use the fact that $n\geq 2$ once again), it suffices to show that
\begin{equation}
\label{eqn_2025_07_02_1642}
\langle 1,1,d^{-1}\rangle \circ (\pi_1,\pi_2,{\rm id})_*\,\Delta_{(n)}=(\pi_1,\pi_2,{\rm id})_*\,\Delta_{(n)}\,. 
\end{equation}



Since $d\equiv 1 \pmod{p^{n-1}}$, the elements $\left ( \begin{smallmatrix}
    d & \\ & 1
\end{smallmatrix}\right), \left ( \begin{smallmatrix}
    d & (d-1)p^{1-n}\\ & 1
\end{smallmatrix}\right) \in \mathrm{GL}_2(\ZZ_p)$ belong to $U_{n-1}$, and hence
\begin{align*}
    \langle 1,1,d^{-1}\rangle \circ (\pi_1,\pi_2,{\rm id})_*\,\Delta_{(n)} &= \left (\left ( \begin{smallmatrix}
    d & \\ & 1
\end{smallmatrix}\right), \left ( \begin{smallmatrix}
    d & (d-1)p^{1-n}\\ & 1
\end{smallmatrix}\right), \left ( \begin{smallmatrix}
    1 & \\ & d
\end{smallmatrix}\right)\right)_* (\pi_1,\pi_2,{\rm id})_*\,\Delta_{(n)} \\ &= (\pi_1,\pi_2,{\rm id})_* \left (\left ( \begin{smallmatrix}
    d & \\ & 1
\end{smallmatrix}\right), \left ( \begin{smallmatrix}
    d & (d-1)p^{-n}\\ & 1
\end{smallmatrix}\right), \left ( \begin{smallmatrix}
    1 & \\ & d
\end{smallmatrix}\right)\right)_* \,\Delta_{(n)}.
\end{align*}
Moreover, as $\Delta_{(n)} = {\rm b}_{(n),*}\Delta_{\mathcal{U}_{(n)}}$, the sought-after equality \eqref{eqn_2025_07_02_1642} follows (thanks to Proposition \ref{prop_2025_05_15_0828}) once we show that
\begin{equation}
\label{eqn_2025_07_03_1204} 
\left (\left ( \begin{smallmatrix}
    d & \\ & 1
\end{smallmatrix}\right), \left ( \begin{smallmatrix}
    d & (d-1)p^{-n}\\ & 1
\end{smallmatrix}\right), \left ( \begin{smallmatrix}
    1 & \\ & d
\end{smallmatrix}\right)\right)_* \,\Delta_{\mathcal{U}_{(n)}}=\,\Delta_{\mathcal{U}_{(n)}}\,. 
\end{equation}
By Definition \ref{def_2025_07_03_1206}, we have
\begin{align*}
    \left (\left ( \begin{smallmatrix}
    d & \\ & 1
\end{smallmatrix}\right), \left ( \begin{smallmatrix}
    d & (d-1)p^{-n}\\ & 1
\end{smallmatrix}\right), \left ( \begin{smallmatrix}
    1 & \\ & d
\end{smallmatrix}\right)\right)_* \,\Delta_{\mathcal{U}_{(n)}} & = \left (\left ( \begin{smallmatrix}
    d & \\ & 1
\end{smallmatrix}\right), \left ( \begin{smallmatrix}
    d & (d-1)p^{-n}\\ & 1
\end{smallmatrix}\right), \left ( \begin{smallmatrix}
    1 & \\ & d
\end{smallmatrix}\right)\right)_* \circ  \iota(\eta_p^n)_\ast \circ \iota_{n,\ast}^u(\mathbf{1}_{n}) \\ 
& =\iota(\eta_p^n)_\ast \circ \left (\left ( \begin{smallmatrix}
    d & \\ & 1
\end{smallmatrix}\right), \left ( \begin{smallmatrix}
    d & d-1\\ & 1
\end{smallmatrix}\right), \left ( \begin{smallmatrix}
    1 & \\ & d
\end{smallmatrix}\right)\right)_* \circ  \iota_{n,\ast}^u(\mathbf{1}_{n}) \\ 
& =\iota(\eta_p^n)_\ast \circ  u_*  \circ \iota \left (\left ( \begin{smallmatrix}
    d & \\ & 1
\end{smallmatrix}\right)\right)_* \circ \iota_{U_{Z,n}\,,\, u\, \mathcal{U}^{(n)} \,u^{-1},\ast}(\mathbf{1}_{n}) 
\\ 
& =\iota(\eta_p^n)_\ast \circ  u_*   \circ \iota_{U_{Z,n}\,,\, u\, \mathcal{U}^{(n)}\, u^{-1},\ast}(\left ( \begin{smallmatrix}
    d & \\ & 1
\end{smallmatrix}\right)_*\mathbf{1}_{n})\,. 
\end{align*}
Recall that $$\mathbf{1}_{n}= \sum_{\mathcal{C} \in \pi_0(X_{Z,n})} [\mathcal{C}] \in {\rm CH}^0(X_{Z,n})$$  
is the fundamental cycle of $X_{Z,n}$.
Since $d\equiv 1 \mod p^{n-1}$ and $n>1$, it follows that $d$ is a square mod $p^n$. As a result, the action of $\left ( \begin{smallmatrix}
    d & \\ & 1
\end{smallmatrix}\right)_*$ preserves the set of connected components of $X_{Z,n}$. Then the equality \eqref{eqn_2025_07_03_1204} follows from the fact that 
$$\left ( \begin{smallmatrix}
    d & \\ & 1
\end{smallmatrix}\right)_*[\mathcal{C}]=\left ( \begin{smallmatrix}
    d^{-1} & \\ & 1
\end{smallmatrix}\right)^*[\mathcal{C}]=[\mathcal{C}]\,,$$
where the last equality is because the fundamental cycle is compatible with (flat) base change and $\left ( \begin{smallmatrix}
    d^{-1} & \\ & 1
\end{smallmatrix}\right)^*$ is an isomorphism (in particular, its degree equals $1$). This concludes the proof.
\end{proof}
 
\vspace{-0.3cm}

\section{Diagonal cycles for anticyclotomic twists of Rankin--Selberg products}\label{S:anticyclotomic}

We keep the definitions in \S\ref{sec:families_of_indefinite_cycles}. Let $K$ be an imaginary quadratic field of discriminant $-D<0$ and let $\varepsilon_{K}$ be the corresponding quadratic character. Assume that $p$ splits in $K$ as $(p)=\frk{p}\overline{\frk{p}}$. Let $\frk{p}$ be the prime of $K$ above $p$ induced by our fixed embedding $\iota_p:\overline{\bb{Q}}\hookrightarrow\overline{\bb{Q}}_p$. Assume also that $p$ does not divide the class number of $K$. Building upon the construction of the classes ${\Delta}_{n}^{\etale}(f,g)$ and $\widetilde{\Delta}_{n}^{\etale}(f,g)$ in \S\ref{sec:families_of_indefinite_cycles}, we define a $p$-adic family of cohomology classes $z_{n,m}$ over ring class fields of $K$ with coefficients in a conjugate self-dual twist of $T_f^\ast\otimes T_g^\ast$ and prove that they satisfy Euler-system tame norm relations and $Q$-system relations. 

\vspace{-0.2cm}
\subsection{CM Hida families}
\label{subsec_CM_Hida_fam}

Let $\psi$ be a Hecke character of $K$ of conductor $\frk{c}$ coprime to $p$ and infinity type $(1-k,0)$ for some even integer $k\geq 2$. Let $\chi_\psi$ be the unique Dirichlet character modulo $N_{K/\bb{Q}}(\frk{c})$ such that 
\[
\psi((n))=n^{k-1} \chi_\psi(n)
\]
for integers $n$ coprime to $N_{K/\bb{Q}}(\frk{c})$. Put $N_\psi=N_{K/\bb{Q}}(\frk{c})D$, and let $\theta_{\psi} \in S_k(N_{\psi}, \chi_\psi\varepsilon_{K})$ be the theta series attached to $\psi$, i.e. $\theta_{\psi} := \sum_{(\frk{a}, \frk{c})=1} \psi(\frk{a}) q^{N_{K/\bb{Q}}(\frk{a})}$.
\subsubsection{}
\label{subsubsec_2025_11_12_1049}
Let $L$ be a finite extension of $\bb{Q}_p$ containing (the image under $\iota_p$ of) $K$, the Fourier coefficients of $f$ and $g$, the roots of the Hecke polynomials of $f$ and $g$ at $p$, and the image of the Hecke character $\psi$. We denote by $\mathcal{O}$ the ring of integers of $L$ and by $\frk{P}$ its maximal ideal. Let $\psi_{\mathfrak{P}}$ be the $p$-adic avatar of $\psi$, i.e. the continuous $L$-valued character of $K^\times\backslash\mathbb{A}_{K,f}^\times$ defined by
\[
\psi_{\mathfrak{P}}(x)=x_\mathfrak{p}^{1-k} \psi(x).
\]
We also denote by $\psi_\mathfrak{P}$ the corresponding character of $G_K$ obtained via the (geometrically normalised) Artin map. The $p$-adic representation attached to the eigenform $\theta_\psi$ is given by $\Ind_{K}^\bb{Q} L(\psi_\frk{P})$. Since $(\frk{c},p)=1$ and $k$ is even, it easily follows that the corresponding residual representation is absolutely irreducible and $p$-distinguished (see \cite[Prop.~5.1.2]{LLZ2} and \cite[Remark~5.1.3]{LLZ2}).

\subsubsection{} Let $\Gamma_{\frk{p}}$ be the Galois group of the unique $\bb{Z}_p$-extension of $K$ unramified outside $\frk{p}$. We can identify $\Gamma_{\frk{p}}$ with $\mathcal{O}_{K,\frk{p}}^{(1)}\cong 1+p\bb{Z}_p$ via the Artin map. Let $\lambda$ denote the unique Hecke character of infinity type $(-1,0)$ and conductor $\frk{p}$ whose $p$-adic avatar $\lambda_\mathfrak{P}:K^\times\backslash \bb{A}_{K,f}^{\times}\rightarrow L^\times$, defined by
\[
\lambda_{\frk{P}}(x)=x_{\frk{p}}^{-1} \lambda(x),
\]
factors through $\Gamma_\frk{p}$. Under the identification $\Gamma_{\frk{p}}\cong 1+p\bb{Z}_p$, the map $\lambda_{\frk{P}}$ is given by $u\mapsto u^{-1}$. We write $\psi=\psi_0 \lambda^{k-2}$. Then $\psi_0$ is a Hecke character of $K$ of conductor dividing $\frk{cp}$ and infinity-type $(-1,0)$. We denote by $\psi_{0,\frk{P}}$ its $p$-adic avatar. 
\vspace{-0.2cm}
\subsubsection{} Let $\Lambda=\mathcal{O}\llbracket1+p\bb{Z}_p\rrbracket$. Let $\bm{\kappa}:1+p\bb{Z}_p\rightarrow \Lambda^\times$ be the tautological character given by $u\mapsto [u]$. Let $\bm{\kappa}^{1/2}:1+p\bb{Z}_p\rightarrow \Lambda^\times$ denote the unique square root of $\bm{\kappa}$. Let $\bm{\lambda}:\Gamma_{\frk{p}}\rightarrow \Lambda^\times$ be the character defined by $\sigma\mapsto [\lambda_{\frk{P}}(\sigma)]$ (i.e. $\bm{\lambda}=\bm{\kappa}\circ\lambda_{\frk{P}}$). We use the same notation for the corresponding character of $G_K$, as well as for the character of $K^\times\backslash \bb{A}_{K,f}$ obtained by composition with the Artin map. Consider the $q$-expansion
\[
\hh = \sum_{(\frk{a},\frk{cp})=1}\psi_0(\frk{a})\bm{\lambda}(\frk{a})q^{N_{K/\bb{Q}}(\frk{a})}\in \Lambda\llbracket q\rrbracket.
\]
Then $\hh$ is a primitive Hida family of tame level $N_\psi$ and character $\chi_\psi\varepsilon_{K}\omega^{k-2}$ passing through the ordinary $p$-stabilization of $\theta_\psi$. The big Galois representation attached to $\hh$ is given by $\Ind_{K}^\bb{Q} \Lambda(\psi_{0,\frk{P}}\bm{\lambda})$.
\vspace{-0.2cm}

\subsection{Jacquet--Langlands correspondence for CM forms}

Recall that $B$ denotes an indefinite quaternion algebra over $\bb{Q}$ of discriminant $N^-$. The main goal of this subsection is to give conditions under which the eigenform $\theta_\psi$ introduced above can be seen as the Jacquet--Langlands lift of an eigenform $\varphi$ on $B^\times$.

Let $\pi=\otimes'_{v} \pi_v$ be the automorphic representation of $\GL_2$ defined by $\theta_\psi$. Since $\theta_\psi$ is the theta series attached to the Hecke character $\psi$, the automorphic representation $\pi$ is the automorphic induction of $\psi$. In particular, we have the following description of its local constituents:
\begin{enumerate}[(i)]
    \item At finite places $v$ of $\bb{Q}$ which split in $K$, say $v=w\overline{w}$, the local representation $\pi_v$ is the principal series $\pi(\psi_w,\psi_{\overline{w}})$.
    \item At finite places $v$ of $\bb{Q}$ which do not split in $K$, the local representation $\pi_v$ is the one corresponding to the Langlands parameter $\Ind_{W_{K_v}}^{W_{\bb{Q}_v}}\psi_v$. Let $c$ denote the non-trivial element of $\Gal(K_v/\bb{Q}_v)$ and let $\psi_v^c=\psi_v\circ c$. We consider the following two subcases:
    \begin{enumerate}[(a)]
        \item Suppose that $\psi_v\neq \psi_v^c$. Then the representation $\Ind_{W_{K_v}}^{W_{\bb{Q}_v}}\psi_v$ is irreducible and therefore $\pi_v$ is supercuspidal.
        \item Suppose that $\psi_v= \psi_v^c$. This is equivalent to assuming that there exists a character $\mu_v$ of $\bb{Q}_v^\times$ such that $\psi_v=\mu_v\circ N_{K_v/\bb{Q}_v}$. Let $\omega_{K_v/\bb{Q}_v}$ denote the quadratic character of $\bb{Q}_v^\times$ attached to the extension $K_v/\bb{Q}_v$. Then $\Ind_{W_{K_v}}^{W_{\bb{Q}_v}}\psi_v=\mu_v\oplus\omega_{K_v/\bb{Q}_v}\mu_v$ and $\pi_v$ is the principal series $\pi(\mu_v,\omega_{K_v/\bb{Q}_v}\mu_v)$.
    \end{enumerate}
    \item At the infinite place $\infty$ of $\bb{Q}$, the local representation $\pi_\infty$ is the holomorphic discrete series $\sigma_k$. 
\end{enumerate}
In case (ii), note that if $\psi_v$ is unramified, then $\psi_v=\psi_v^c$, so we are in subcase (ii)(b).

Let $N_{K/\bb{Q}}(\frk{c})=N_{\frk{c}}^+N_{\frk{c}}^-$, where all primes dividing $N_{\frk{c}}^+$ split in $K$ and all primes dividing $N_{\frk{c}}^-$ do not split in $K$. It is known that an eigenform $f$ on $\GL_2$ is the Jacquet--Langlands lift of an eigenform $\varphi$ on $B^\times$ if and only if for all places $v\mid N^-$ the corresponding local representation is a discrete series. Therefore, based on the previous discussion, we have the following result.

\begin{proposition}
    The eigenform $\theta_\psi$ is the Jacquet--Langlands lift of an eigenform $\varphi$ on $B^\times$ if and only if the following condition holds:
    \begin{itemize}
        \item[\mylabel{item_JLCM}{$\mathbf{JL}_{\rm CM}$})] For all places $v\mid N^-$, we have $v\mid N_{\frk{c}}^-$ and $\psi_v\neq \psi_v^c$.
    \end{itemize}
\end{proposition}

Assume that \eqref{item_JLCM} holds. Let $N=N^+N^-$ be a positive integer as in \S\ref{subsec:ShimuraCurves} and assume that $N_\psi\mid N$. For each natural number $n$ and for each positive integer $m$ coprime to $N p$, we put
\[
U_n := \{ g \in U_1(N) \,:\, \Psi_p(g_p) \equiv \left( \begin{smallmatrix} a & \star  \\ 0 & a \end{smallmatrix}\right)\, (\text{mod } p^n)\, ,\,a \in\Zp^\times\}\,,
\]
\[
U^m:=\{ g \in U_1(N)\,:\, \Psi_q(g_q) \equiv \left( \begin{smallmatrix}
    \star & \star \\ 0 & \star
\end{smallmatrix} \right)\, (\text{mod } m^2 \ZZ_q ) \text{ for } q| m \}\,,
\]
\[
U_n^m:=U_n\cap U^m\,.
\]
We denote by $X(U_n^m)$ the Shimura curve of level $U_n^m$, and we denote by $T_B(U_n^m)$ the algebra of Hecke correspondences of level $U_n^m$ generated by the operators $\lbrace \mathrm{T}_\ell',\mathrm{S}_\ell, \mathrm{S}_\ell^{-1}\,\vert\, \ell\nmid Np\rbrace$ and $\lbrace \mathrm{U}_p',\langle x\rangle\,\vert\, x\in \bb{Z}_p^\times\rbrace$. We define
\[
H^1_{\Iw}(X(U_\infty^m)\times_{\bb{Q}}\overline{\bb{Q}},\mathcal{O}):=\varprojlim_n H^1_{\etale}(X(U_n^m)\times_{\bb{Q}}\overline{\bb{Q}},\mathcal{O}).
\]
Let $\mathbb{T}^{\mathrm{a.o.}}_{B,\mathcal{O}}(U_{\infty}^m)$ denote the subalgebra of $\End_{\mathcal{O}}(e_{\mathrm{U}_p'}H^1_{\Iw}(X(U_\infty^m)\times_{\bb{Q}}\overline{\bb{Q}},\mathcal{O}))$ generated by the Hecke operators $\lbrace \mathrm{T}_\ell',\mathrm{S}_\ell, \mathrm{S}_\ell^{-1}\,\vert\, \ell\neq p\rbrace$ and $\lbrace \mathrm{U}_p',\langle x\rangle\,\vert\, x\in \bb{Z}_p^\times\rbrace$.

In the next subsection, following \cite{LLZ2}, we carry out a patching argument to obtain a family of cohomology classes satisfying suitable norm relations. However, the results of \cite{LLZ2} rely on the following crucial input (stated here for the levels that will be relevant in our case).

\begin{theorem}
    Assume that $B=\mathrm{M}_2(\bb{Q})$. Let $\frk{m}$ be a non-Eisenstein $p$-distinguished maximal ideal of $\mathbb{T}^{\mathrm{a.o.}}_{B,\mathcal{O}}(U_{\infty}^m)$. Then $e_{\mathrm{U}_p'}H^1_{\Iw}(X(U_\infty^m)\times_{\bb{Q}}\overline{\bb{Q}},\mathcal{O})$ is a free $\mathbb{T}^{\mathrm{a.o.}}_{B,\mathcal{O}}(U_{\infty}^m)$-module of rank $2$.
\end{theorem}
\begin{proof}
    With a different choice of tower of level structures at $p$, this follows from \cite[Prop.~3.3.1]{EPW}. The comparison of the relevant cohomology groups for the choice of level structures at $p$ in \emph{op.\,cit.} and the choice in the present article is explained in Appendix~\ref{sec:comparison_of_towers}.
\end{proof}

At present, we are not aware of an analogous result for arbitrary Shimura curves over $\bb{Q}$ in the level of generality that we need. Therefore, in the rest of this article we assume that $B=\mathrm{M}_2(\bb{Q})$.

\subsection{Patching}

We keep the definitions and notations in the previous subsections and consider the Hecke character $\varphi=\psi\vert\cdot\vert^{k/2-1}$ of $K$, which has conductor $\frk{c}$ and infinity-type $(-k/2,k/2-1)$. We denote by $\varphi_{\frk{P}}$ the $p$-adic avatar of $\varphi$.

Let $\Gamma_{\ac}$ be the Galois group of the anticyclotomic $\bb{Z}_p$-extension of $K$. We can identify $\Gamma_{\ac}$  with the anti-diagonal in $(1+p\bb{Z}_p)\times (1+p\bb{Z}_p)\cong \mathcal{O}_{K,\frk{p}}^{(1)}\times \mathcal{O}_{K,\overline{\frk{p}}}^{(1)}$ via the Artin map. Let 
\[
\kappa_{\ac}:\Gamma_{\ac}\longrightarrow \bb{Z}_p^\times
\]
be the character defined by mapping the element $((1+p)^{-1},(1+p))$ to $1+p$ and let $\bm{\kappa}_{\ac}:\Gamma_{\ac}\rightarrow \Lambda^\times$ be the character defined by mapping the element $((1+p)^{-1}, (1+p))$ to the group-like element $[1+p]$. We use the same notation for the corresponding characters of $G_K$.

Write $\varphi_{\frk{P}}=\varphi_{0,\frk{P}}\kappa_{\ac}^{2-k}$. Then $\varphi_{0,\frk{P}}$ is the $p$-adic avatar of a Hecke character $\varphi_0$ of $K$ of conductor dividing $\frk{cp}$, infinity type $(-1,0)$ and central character $\chi_{\varphi_0}=\chi_\psi\omega^{k-2}$.

For each positive integer $m$, let $K[m]$ denote the maximal $p$-subextension of the ring class field of $K$ of conductor $m$ and let $R_m=\Gal(K[m]/K)$. For each positive integer $n$, let $\Gamma_n=(1+p\bb{Z}_p)/(1+p\bb{Z}_p)^{p^{n-1}}$ and $\Lambda_n=\mathcal{O}[\Gamma_n]$. Note that, via the isomorphism $\Gamma_{\ac}\cong 1+p\bb{Z}_p$ defined by $\kappa_{\ac}$, we have the identifications $\Gamma_n\cong \Gamma_{\ac}/\Gamma_{\ac}^{p^{n-1}}\cong R_{p^{n}}$.

\begin{proposition}
There exists a homomorphism $\phi^m_n:T_B(U_n^m)\rightarrow \Lambda_n[R_m]$ defined on generators as follows:
\begin{itemize}
    \item For every rational prime $q\nmid Np$,
    \[
    \phi^m_n(\mathrm{T}_q')=\sum_{\frk{q}}\varphi_0(\frk{q})\bm{\kappa}_{\ac}(\frk{q})[\frk{q}],
    \]
    where the sum runs over ideals of $\mathcal{O}_K$ coprime to $m\frk{c}$ of norm $q$ and $[\frk{q}]$ denotes the image of $\frk{q}$ in $R_m$ by the Artin map.
    \item For every rational prime $q\nmid Np$,
    \[
    \phi_n^m(\mathrm{S}_q^{-1})=(\chi_\psi\epsilon_K)(q).
    \]
    \item $\phi_{n}^m(\mathrm{U}_p')=\psi(\overline{\frk{p}})[\overline{\frk{p}}]$, where $[\overline{\frk{p}}]$ denotes the image of $\overline{\frk{p}}$ in $R_m$ by the Artin map.
    \item for all $x\in \mathbb{Z}_p^\times$,
    \[
    \phi_n^m(\langle x\rangle)=\omega(x)^{1-k/2}\bm{\kappa}^{-1/2}(x).
    \]
\end{itemize}
\end{proposition}
\begin{proof}
    This follows as in \cite[Prop.~3.2.1]{LLZ2}.
\end{proof}

Let $m$ be a positive integer coprime to $Np$, let $q$ be a rational prime coprime to $Np$ and let $n\in \bb{Z}_{\geq 1}$. Let $\varpi_{11,\ast}$, $\varpi_{12,\ast}$ and $\varpi_{22,\ast}$ denote the degeneracy maps
\[
\varpi_{ij,\ast}=\varpi_{i,q,\ast}\circ \varpi_{j,q,\ast}: H^1_{\etale}(X(U_n^{mq})\times_\bb{Q}\overline{\bb{Q}},\mathcal{O}) \longrightarrow H^1_{\etale}(X(U_n^{m})\times_\bb{Q}\overline{\bb{Q}},\mathcal{O}).
\]
Following \cite[\S~3.3]{LLZ2}, we define norm maps
\[
\mathcal{N}_{m,n}^{mq,n}:\Lambda_n[R_{mq}]\otimes_{\phi_{n}^{mq}} H^1_{\etale}(X(U_n^{mq})\times_\bb{Q}\overline{\bb{Q}},\mathcal{O})\longrightarrow \Lambda_n[R_{m}]\otimes_{\phi_n^{m}} H^1_{\etale}(X(U_n^{m})\times_\bb{Q}\overline{\bb{Q}},\mathcal{O})
\]
by the following formulae:
\begin{itemize}
    \item if $q\mid m$, $\mathcal{N}_{m,n}^{mq,n}=1\otimes \varpi_{11,\ast}$.
    \item if $q\nmid m$ and $q$ splits in $K$ as $(q)=\frk{q}\overline{\frk{q}}$,
    \begin{align*}
    \mathcal{N}_{m,n}^{mq,n}&=1\otimes \varpi_{11,\ast}-\left(\frac{\varphi_{0}(\frk{q})\bm{\kappa}_{\ac}(\frk{q})}{q}[\frk{q}]+\frac{\varphi_{0}(\overline{\frk{q}})\bm{\kappa}_{\ac}(\overline{\frk{q}})}{q}[\overline{\frk{q}}]\right)\otimes \varpi_{12,\ast}+\frac{\chi_\psi(q)}{q}\otimes \varpi_{22,\ast}\,.
    \end{align*}
    \item if $q\nmid n$ and $q$ is inert in $K$,
    \[
    \mathcal{N}_{m,n}^{mq,n}=1\otimes \varpi_{11,\ast}-\frac{\chi_\psi(q)}{q}\otimes \varpi_{22,\ast}\,.
    \]
\end{itemize}

Also, for any positive integer $m$ coprime to $N p$ and for integers $n'\geq n\geq 1$, we define the norm map
\[
\mathcal{N}_{m,n}^{m,n'}:\Lambda_{n'}[R_{m}]\otimes_{\phi_{n'}^{m}} H^1_{\etale}(X(U_{n'}^{m})\times_\bb{Q}\overline{\bb{Q}},\mathcal{O})\longrightarrow \Lambda_n[R_{m}]\otimes_{\phi_n^{m}} H^1_{\etale}(X(U_n^{m})\times_\bb{Q}\overline{\bb{Q}},\mathcal{O})
\]
by $\mathcal{N}_{m,n}^{m,n'}=1\otimes \pi_{1,\ast}$. As in \cite[Prop.~3.3.2]{LLZ2}, it can be verified that all the norm maps above are well-defined. More generally, given positive integers $m\mid m'$ coprime to $Np$ and integers $n'\geq n\geq 1$, we define
\[
\mathcal{N}_{m,n}^{m',n'}:\Lambda_{n'}[R_{m'}]\otimes_{\phi_{n'}^{m'}} H^1_{\etale}(X(U_{n'}^{m'})\times_\bb{Q}\overline{\bb{Q}},\mathcal{O})\longrightarrow \Lambda_n[R_{m}]\otimes_{\phi_n^{m}} H^1_{\etale}(X(U_n^{m})\times_\bb{Q}\overline{\bb{Q}},\mathcal{O})
\]
by composing the previously defined norm maps in the natural way.

\begin{proposition}\label{lemma:LLZ}
There exists a family of $G_\bb{Q}$-equivariant homomorphisms of $\Lambda_n[R_m]$-modules
\begin{equation*}
\nu_{n}^m:\Lambda_n[R_{m}]\otimes_{\phi_n^m} e_{\mathrm{U}_p'}H^1_{\etale}(X(U_n^m)\times_\bb{Q}\overline{\bb{Q}},\mathcal{O}) \longrightarrow  \Ind_{K[m]}^\bb{Q}\Lambda_n(\varphi_{0,\frk{P}}^c\bm{\kappa}_{\ac}^{-1}),
\end{equation*}
for positive integers $m$ coprime to $N p$ and for positive integers $n$, such that for integers $m\mid m'$ coprime to $N p$ and for integers $n'\geq n\geq 1$, the diagram
\[
\begin{tikzcd}
\Lambda_{n'}[R_{m'}]\otimes_{\phi_{n'}^{m'}} e_{\mathrm{U}_p'}H^1_{\etale}(X(U_{n'}^{m'})\times_\bb{Q}\overline{\bb{Q}},\mathcal{O}) \arrow[r, "\nu_{n'}^{m'}"] \arrow[d, "\mathcal{N}_{m,n}^{m',n'}"] & \Ind_{K[m']}^\bb{Q}\Lambda_{n'}(\varphi_{0,\frk{P}}^c\bm{\kappa}_{\ac}^{-1}) \arrow[d, "\mathrm{Norm}"] \\
\Lambda_n[R_{m}]\otimes_{\phi_n^m} e_{\mathrm{U}_p'} H^1_{\etale}(X(U_n^m)\times_{\bb{Q}}\overline{\bb{Q}},\mathcal{O}) \arrow[r, "\nu_n^m"] & \Ind_{K[m]}^\bb{Q}\Lambda_n(\varphi_{0,\frk{P}}^c\bm{\kappa}_{\ac}^{-1})
\end{tikzcd}
\]
commutes.
\end{proposition}
\begin{proof}
    This follows from \cite[Cor.~5.2.6]{LLZ2}.
\end{proof}

\subsection{Norm relations}

For the remaining of this section, we assume that the self-duality condition \eqref{item_self_duality} holds. Let $r$ be a positive integer coprime to $Np$. For $\bullet \in \{ 1, Z , \emptyset, {\rm spl} \}$ and for $n\in \mathbb{Z}_{\geq 0}$, we put
\[
U_{\bullet,n}(r)=U_{\bullet,n}\cap U_0(r)\subseteq \widehat{B}^{\times}.
\]
Similarly, for $\bullet \in \{ 1, Z , \emptyset, {\rm spl} \}$ and for $\mathbf{n}=(n_1,n_2,n_3) \in \mathbb{Z}_{\geq 0}^3$, we put
\[
\mathcal{U}_{\bullet,\mathbf{n}}(r)=\mathcal{U}_{\bullet,\mathbf{n}}\cap \mathcal{U}_0(r) \subseteq \widehat{B}_E^{\times}
\]
and we define
\[
\mathbf{X}_{\mathbf{n}}(r)=X(U_{n_1}(r))\times X(U_{n_2}(r))\times X(U_{n_3}(r))
\]
for each positive integer $r$ coprime to $Np$.

Recall that in \S\ref{sec:families_of_indefinite_cycles} we constructed cycles $\Delta_{\mathbf{n}}\in \mathrm{CH}^2(\mathbf{X}_{\mathbf{n}})$. Using the same recipe, one can also construct
\[
{\Delta}_{\mathbf{n}}(r)^\circ\in \mathrm{CH}^2(\mathbf{X}_{\mathbf{n}}(r)).
\]

\begin{lemma}\label{lem:deg-maps}
    Let $r$ be a positive integer coprime to $Np$ and let $q$ be a rational prime coprime to $Np$. Then, for every $\mathbf{n}=(n_1,n_2,n_3) \in \mathbb{Z}_{\geq 0}^3$,
    \begin{align*}
    &(\varpi_{2,q},\varpi_{1,q},\varpi_{1,q})_\ast \,{\Delta}_{\mathbf{n}}(rq)^\circ=(\mathrm{T}_q,1,1) {\Delta}_{\mathbf{n}}(r)^\circ;\\
    &(\varpi_{1,q},\varpi_{2,q},\varpi_{1,q})_\ast \,{\Delta}_{\mathbf{n}}(rq)^\circ=(1,\mathrm{T}_q,1) {\Delta}_{\mathbf{n}}(r)^\circ;\\
    &(\varpi_{1,q},\varpi_{1,q},\varpi_{2,q})_\ast \,{\Delta}_{\mathbf{n}}(rq)^\circ=(1,1,\mathrm{T}_q) {\Delta}_{\mathbf{n}}(r)^\circ;\\
    &(\varpi_{1,q},\varpi_{2,q},\varpi_{2,q})_\ast \,{\Delta}_{\mathbf{n}}(rq)^\circ=(\mathrm{T}_q',1,1) {\Delta}_{\mathbf{n}}(r)^\circ;\\
    &(\varpi_{2,q},\varpi_{1,q},\varpi_{2,q})_\ast \,{\Delta}_{\mathbf{n}}(rq)^\circ=(1,\mathrm{T}_q',1) {\Delta}_{\mathbf{n}}(r)^\circ;\\
    &(\varpi_{2,q},\varpi_{2,q},\varpi_{1,q})_\ast \,{\Delta}_{\mathbf{n}}(rq)^\circ=(1,1,\mathrm{T}_q') {\Delta}_{\mathbf{n}}(r)^\circ.
    \end{align*}
    If $q$ is coprime to $r$ we have
    \begin{align*}
    &(\varpi_{1,q},\varpi_{1,q},\varpi_{1,q})_\ast \,{\Delta}_{\mathbf{n}}(rq)^\circ=(q+1) {\Delta}_{\mathbf{n}}(r)^\circ;\\
    &(\varpi_{2,q},\varpi_{2,q},\varpi_{2,q})_\ast \,{\Delta}_{\mathbf{n}}(rq)^\circ=(q+1) {\Delta}_{\mathbf{n}}(r)^\circ;
    \end{align*}
    otherwise, if $q\mid r$, we have
    \begin{align*}
    &(\varpi_{1,q},\varpi_{1,q},\varpi_{1,q})_\ast \,{\Delta}_{\mathbf{n}}(rq)^\circ=q {\Delta}_{\mathbf{n}}(r)^\circ;\\
    &(\varpi_{2,q},\varpi_{2,q},\varpi_{2,q})_\ast \,{\Delta}_{\mathbf{n}}(rq)^\circ=q {\Delta}_{\mathbf{n}}(r)^\circ.
    \end{align*}
\end{lemma}
\begin{proof}
    Let $n=\max\lbrace n_1,n_2,n_3\rbrace$. We have the commutative diagram
    \[
    \xymatrix{  
     & & X(U_{n}(rq))\times X(U_{n}(rq))\times X(U_{n}(rq)) \ar[d]^-{(1,\varpi_{1,q},\varpi_{1,q})} \\ X_{Z,n}(rq) \ar[d]_-{\varpi_{1,q}} \ar[urr]^-{\mathrm{b}_{(n)}\circ\iota(\eta_p^n)\circ\iota_{n}^u} \ar[rr] & & X(U_{n}(rq))\times X(U_{n}(r))\times X(U_{n}(r)) \ar[d]^-{(\varpi_{1,q},1,1)} \\ 
     X_{Z,n}(r) \ar[rr]_-{\mathrm{b}_{(n)}\circ\iota(\eta_p^n)\circ\iota_{n}^u} & & X(U_{n}(r))\times X(U_{n}(r))\times X(U_{n}(r))\,.}
     \]
     Since the horizontal arrows are closed embeddings and the two bottom vertical arrows are finite morphisms of the same degree, it follows that the bottom square is Cartesian. 
     Let $\bm{1}_{n,rq}\in \CH^0(X_{Z,n}(rq))$ denote the fundamental cycle of $X_{Z,n}(rq)$ and let $\bm{1}_{n,r}\in\CH^0(X_{Z,n}(r))$ denote the fundamental cycle of $X_{Z,n}(r)$. Then $\Delta_{(n)}(rq)^\circ=\mathrm{b}_{(n),\ast}\circ\iota(\eta_p^n)_{\ast}\circ\iota_{n,\ast}^u(\bm{1}_{n,rq})$ and $\Delta_{(n)}(r)^\circ=\mathrm{b}_{(n),\ast}\circ\iota(\eta_p^n)_{\ast}\circ\iota_{n,\ast}^u(\bm{1}_{n,r})$. By the Cartesianness of the bottom square in the previous diagram, we have
     \begin{equation}\label{eq:202510241545}
         (1,\varpi_{1,q},\varpi_{1,q})_\ast \,\Delta_{(n)}(rq)^\circ=(\varpi_{1,q},1,1)^\ast \, \Delta_{(n)}(r)^\circ.
     \end{equation}
     Applying $(\varpi_{2,q},1,1)_{\ast}$ to both sides of \eqref{eq:202510241545}, we obtain
     \[
     (\varpi_{2,q},\varpi_{1,q},\varpi_{1,q})_\ast \,\Delta_{(n)}(rq)^\circ=(\mathrm{T}_q,1,1) \Delta_{(n)}(r)^\circ,
     \]
     whereas applying $(\varpi_{1,q},1,1)_{\ast}$ to both sides of \eqref{eq:202510241545}, we obtain
     \[
     (\varpi_{1,q},\varpi_{1,q},\varpi_{1,q})_\ast \,\Delta_{(n)}(rq)^\circ=\deg(\varpi_{1,q}) \Delta_{(n)}(r)^\circ.
     \]
     Note that $\Delta_{\mathbf{n}}(rq)^\circ=\pr_{\mathbf{n}}^{(n)} \Delta_{(n)}(rq)^\circ$ and $\Delta_{\mathbf{n}}(r)^\circ=\pr_{\mathbf{n}}^{(n)} \Delta_{(n)}(r)^\circ$. Since
     \[
     \pr_{\mathbf{n}}^{(n)}\circ (\varpi_{i,q},\varpi_{1,q},\varpi_{1,q})_\ast=(\varpi_{i,q},\varpi_{1,q},\varpi_{1,q})_\ast\circ \pr_{\mathbf{n}}^{(n)}
     \]
     for $i\in\lbrace 1,2\rbrace$, we deduce that
     \begin{align*}
     (\varpi_{2,q},\varpi_{1,q},\varpi_{1,q})_\ast \,\Delta_{\mathbf{n}}(rq)^\circ &=(\mathrm{T}_q,1,1) \Delta_{\mathbf{n}}(r)^\circ\,, \qquad
     (\varpi_{1,q},\varpi_{1,q},\varpi_{1,q})_\ast \,\Delta_{\mathbf{n}}(rq)^\circ &=\deg(\varpi_{1,q}) \Delta_{\mathbf{n}}(r)^\circ\,.
     \end{align*}
     By the second equality, we have
     \[
     (\varpi_{1,q},\varpi_{1,q},\varpi_{1,q})_\ast \,\Delta_{\mathbf{n}}(rq)^\circ=q \Delta_{\mathbf{n}}(r)^\circ\,, \qquad \hbox{if $q\mid r$ }
     \]
     \[
     (\varpi_{1,q},\varpi_{1,q},\varpi_{1,q})_\ast \,\Delta_{\mathbf{n}}(rq)^\circ=(q+1) \Delta_{\mathbf{n}}(r)^\circ\,,\qquad \hbox{if $q\nmid r$.}
     \]

     The remaining identities in the statement can be proved in a similar way.
\end{proof}

Let $m$ be a square-free product of primes coprime to $Np$. For $i=1,2$ and $n\in \mathbb{Z}_{\geq 0}$, consider the degeneracy maps
\[
\varpi_{i,m^2}=\prod_{q\mid m} \varpi_{i,q}^2 : X(U_n^m)\rightarrow X(U_n), 
\]
where the product is taken over all primes dividing $m$. We then define the cycles
\[
\Delta_{\mathbf{n}}(m)=(\varpi_{1,m^2},\varpi_{2,m^2},1)_\ast\,\Delta_{\mathbf{n}}(m^2)^\circ\in\mathrm{CH}^2(X(U_{n_1})\times X(U_{n_2})\times X(U_{n_3}^m)).
\]

\begin{lemma}\label{lemma:norm-relations}
Let $m$ be a square-free product of primes coprime to $Np$, and let $q$ be a rational prime coprime to $Nmp$. Then, for every $\mathbf{n}=(n_1,n_2,n_3) \in \mathbb{Z}_{\geq 0}^3$,
\begin{align*}
(1,1,\varpi_{11,q})_\ast \Delta_{\mathbf{n}}(mq)&=\left\{(1,\mathrm{S}_q^2 (\mathrm{T}_{q}')^2,1)-(q+1)(1,\mathrm{S}_q,1)\right\}\Delta_{\mathbf{n}}(m),\\
(1,1,\varpi_{12,q})_\ast \Delta_{\mathbf{n}}(mq)&=\left\{(\mathrm{T}_q',\mathrm{S}_q\mathrm{T}_{q}',1)-(\mathrm{S}_q^{-1},1,\mathrm{T}_q')\right\}\Delta_{\mathbf{n}}(m),\\
(1,1,\varpi_{22,q})_\ast \Delta_{\mathbf{n}}(mq)&=\left\{((\mathrm{T}_q')^2,1,1)-(q+1)(\mathrm{S}_q^{-1},1,1)\right\}\Delta_{\mathbf{n}}(m).
\end{align*}
\end{lemma}
\begin{proof}
In view of Lemma~\ref{lem:deg-maps}, the argument is identical to that of \cite[Lem.~4.4]{ACR}.
\end{proof}

For each positive integer $n$ and for each square-free positive integer $m$ coprime to $Np$, we set
$$H^3_{\etale}(\mathbf{X}_{1,0,n}^m\times_{\QQ}
\overline{\QQ},\ZZ_p)^{f \,g\, \circ} := T_f^*\otimes T_g^* \otimes e_{\mathrm{U}_{p}'} H^1_{\etale} (X(U_{n}^m)\times_{\QQ}
\overline{\QQ}, \ZZ_p)\,.$$
Following the construction in \S\ref{sec:families_of_indefinite_cycles}, we use the cycles $\Delta_{\mathbf{n}}(m)$ to define cohomology classes
\[
\Delta_{n,m}^{\etale}(f,g),\,\widetilde{\Delta}_{n,m}^{\etale}(f,g) \in H^1(\bb{Q},H^3_{\etale}(\mathbf{X}_{1,0,n}^m\times_{\QQ}
\overline{\QQ},\ZZ_p)^{f \,g\, \circ}(2))\,.
\]
Applying the homomorphisms $\nu_n^m$ of Lemma~\ref{lemma:LLZ} to these classes, we obtain classes
\[
z_{n,m},\,\tilde{z}_{n,m}\in H^1(\bb{Q},T_f^*\otimes T_g^*\otimes \Ind_{K[m]}^{\bb{Q}}\Lambda_n(\varphi_{0,\frk{P}}^c\bm{\kappa}_{\ac}^{-1})(2))\,.
\]
By Shapiro's isomorphism, these classes can also be regarded as classes in
\[
H^1(K[m],T_f^\ast\otimes T_g^\ast\otimes \Lambda_n(\varphi_{0,\frk{P}}^c\bm{\kappa}_{\ac}^{-1})(2))
\simeq H^1(K[mp^n],T_f^\ast\otimes T_g^\ast(\varphi_{0,\frk{P}}^c)(2))\,.
\]

\begin{proposition}\label{prop:horizontal-norm}
    Let $n$ be a positive integer, let $m$ be a square-free product of primes coprime to $Np$ and split in $K$ and let $q$ be a prime coprime to $Nmp$ and split in $K$. Then
    \begin{align*}
        \mathrm{cor}_{K[mqp^n]/K[mp^n]}(z_{n,mq}) &= \chi_g^{-1}(q)\Bigg\{ q\chi_f(q)\chi_g(q) \left(\frac{\varphi_0(\frk{q})}{q}\Fr_{\frk{q}}\right)^2 -a_q(f)a_q(g) \left(\frac{\varphi_0(\frk{q})}{q}\Fr_{\frk{q}}\right) \\
        & \hspace{2cm} +\frac{\chi_f^{-1}(q)a_q(f)^2}{q}+ \chi_{g}^{-1}(q)a_{q}(g)^2-\frac{q^2+1}{q} \\
        &\hspace{2.8cm} -a_q(f)a_q(g) \left(\frac{\varphi_0(\overline{\frk{q}})}{q}\Fr_{\overline{\frk{q}}}\right) +q\chi_f(q)\chi_g(q) \left(\frac{\varphi_0(\overline{\frk{q}})}{q}\Fr_{\overline{\frk{q}}}\right)^2 \Bigg\} z_{n,m},
    \end{align*}
    where $\Fr_\frk{q}$ denotes a geometric Frobenius at $\frk{q}$.
\end{proposition}
\begin{proof}
The corestriction map in Galois cohomology corresponds to the norm map in Lemma~\ref{lemma:LLZ} under Shapiro's isomorphism. Hence, the result follows from Lemma~\ref{lemma:LLZ} and Lemma~\ref{lemma:norm-relations} after straightforward computations.
\end{proof}

\begin{proposition}
\label{prop:2025_08_11_1553}
    Let $n$ be a positive integer and let $m$ be a square-free product of primes coprime to $Np$ and split in $K$; then
    \begin{equation*}
        \mathrm{cor}_{K[mp^{n+1}]/K[mp^n]}(z_{n+1,m})=a_p(g) z_{n,m}-\chi_g(p)\tilde{z}_{n,m}\,,
    \end{equation*}
    \begin{equation*}
        \mathrm{cor}_{K[mp^{n+1}]/K[mp^n]}(\tilde{z}_{n+1,m})=p\cdot z_{n,m}\,.
    \end{equation*}
\end{proposition}
\begin{proof}
    This follows from Corollary~\ref{cor_2025_03_22_0955}.
\end{proof}

We will henceforth work under the following assumption. We note that this hypothesis will be replaced by its stronger version \eqref{item_BI0} in \S\ref{subsubsec_2025_10_30_0947}.

\begin{assumption}
    \label{ass:irred}
    The residual $G_K$-representation $\overline{T}_f^\ast\otimes \overline{T}_g^\ast(\varphi_{0,\frk{P}}^c)(2)$ satisfies $H^0(K,\overline{T}_f^\ast\otimes \overline{T}_g^\ast(\varphi_{0,\frk{P}}^c)(2))=0$.
\end{assumption}

\begin{remark}\label{rk:irred}
    As a consequence of \eqref{item_H0}, we have for any finite $p$-extension $\mathcal{K}$ of $K$ that
    \[
    H^0(\mathcal{K},\overline{T}_f^\ast\otimes \overline{T}_g^\ast(\varphi_{0,\frk{P}}^c)(2))=0.
    \]
\end{remark}

\begin{proposition}
\label{prop:2025_08_11_1554}
    For all positive integers $n$ and for all square-free products $m$ of primes coprime to $Np$ and split in $K$, we have
    \[
    \tilde{z}_{n+1,m}=\mathrm{res}_{K[mp^n]/K[mp^{n+1}]}(z_{n,m}).
    \]
\end{proposition}
\begin{proof}
First we prove that
\[
\tilde{z}_{n+1,m}\in \mathrm{im}\left(H^1(K[mp^n],T_f^\ast\otimes T_g^\ast(\varphi_{0,\frk{P}}^c)(2))\rightarrow H^1(K[mp^{n+1}],T_f^\ast\otimes T_g^\ast(\varphi_{0,\frk{P}}^c)(2))\right).
\]
By \eqref{item_H0}  and Remark~\ref{rk:irred}, this is equivalent to showing that
\[
\tilde{z}_{n+1,m}\in H^1(K[mp^{n+1}],T_f^\ast\otimes T_g^\ast(\varphi_{0,\frk{P}}^c)(2))^{G_{K[mp^n]}},
\]
which holds by Proposition~\ref{prop_2025_03_08_1713}. We therefore know that $\tilde{z}_{n+1,m}=\mathrm{res}_{K[mp^n]/K[mp^{n+1}]}(c)$ for some
\[
c\in H^1(K[mp^n],T_f^\ast\otimes T_g^\ast(\varphi_{0,\frk{P}}^c)(2)).
\]
Using the second equation in Proposition~\ref{prop:2025_08_11_1553} and noting that the composition
\[
\mathrm{cor}_{K[mp^{n+1}]/K[mp^{n}]}\circ\mathrm{res}_{K[mp^n]/K[mp^{n+1}]},
\]
equals multiplication by $p$, we deduce that
\[
p \cdot z_{n,m} =\mathrm{cor}_{K[mp^{n+1}]/K[mp^n]}(\tilde{z}_{n+1,m})=p\cdot c.
\]
Finally, again by \eqref{item_H0}  and Remark~\ref{rk:irred}, we infer that $H^1(K[mp^{n+1}],T_f^\ast\otimes T_g^\ast(\varphi_{0,\frk{P}}^c)(2))$ is torsion-free, so we obtain $z_{n,m}=c$ and therefore $\tilde{z}_{n+1,m}=\mathrm{res}_{K[mp^n]/K[mp^{n+1}]}(z_{n,m})$, as required.
\end{proof}

\begin{proposition}
\label{prop_thm_2025_07_03_1627}
    Let $n\geq 2$ be an integer and let $m$ be a square-free product of primes coprime to $Np$ and split in $K$. Then
    \[
    \mathrm{cor}_{K[mp^{n+1}]/K[mp^n]}(z_{n+1,m})=a_p(g)z_{n,m}-\chi_g(p)\mathrm{res}_{K[mp^{n-1}]/K[mp^n]}(z_{n-1,m}).
    \]
\end{proposition}
\begin{proof}
    This follows by combining the first equation in Proposition~\ref{prop:2025_08_11_1553} and Proposition~\ref{prop:2025_08_11_1554}.
\end{proof}

\begin{remark}
    Note that the previous proposition may also be proved using the more conceptual strategy in the proof of Proposition~\ref{prop_2025_10_15_1640}.
\end{remark}

\section{Iwasawa cohomology classes}
As in \S\ref{S:anticyclotomic}, $K$ denotes an imaginary quadratic field in which $p$ splits.  Let $\psi$ be a Hecke character of $K$ of conductor $\frk{c}$ coprime to $p$ and infinity type $(1-k,0)$ for some even integer $k\geq 2$. Throughout this section, $m$ denotes a square-free product of primes that are coprime to $Np$ and split in $K$. Furthermore, we assume that \eqref{item_H0}  holds. We explain how the collection $\left\{z_{n,m}\right\}_{n}$ gives rise to classes that take values in tempered distributions and bounded signed classes.

\subsection{Classes with values in tempered distributions}
We first recall definitions of distributions and the associated norms. 
\begin{defn}
\label{defn_PR_rings}
Let $0\le \lambda<1$ be a real number.
\item[i)]  We define the Banach space $\cC_\lambda$ of $L$-valued order-$\lambda$ functions on $1+p\Zp$ as in \cite[\S I.5]{Colmez2010Ast}, after identifying the multiplicative group $1+p\Zp$ with the additive group $\Zp$.
   
\item[ii)] We let $\cD_\lambda$ denote the continuous $L$-dual of $\cC_\lambda$.
   \item[iii)] Given $\mu\in\cD_\lambda$, we define the norm
   \[
   ||\mu||_\lambda=\sup_{n\ge 0}\,\,\sup_{a\in 1+p\Zp} p^{-\lfloor \lambda n\rfloor}\left|\left|\int_{a(1+p^{n+1}\Zp)} \mu\right|\right|.
   \]
   \item[iv)] Let $\sL$ denote the $L$-dual of the space of locally constant $L$-valued functions on $1+p\Zp$.
\end{defn}

 Via the Amice transform, we may consider elements of $\sL$ as power series in $L\llbracket X\rrbracket$ that converge on the open unit disk that are $O(\log)$. Those in $\cD_\lambda$ correspond to the power series that are $O(\log^\lambda)$. In particular, there is a natural injection $\cD_{\lambda}\hookrightarrow\sL$ for any $0\le \lambda<1$.

 Let $n\ge1$ be an integer. Recall that $\Gamma_n=(1+p\Zp)/(1+p\Zp)^{p^{n-1}}=(1+p\Zp)/(1+p^n\Zp)$. We write $\sL_n$ for the dual of the set of functions on $1+p\Zp$ that are spanned by the indicator functions of the cosets in $\Gamma_n$. By an abuse of notation, we again write $||\cdot||$ for the sup-norm on $\sL_n$. There is a natural projection $\sL_{n'}\to \sL_n$ for all $n'\ge n$. When viewed as power series, this is the map given by the projection modulo $(1+X)^{p^{n-1}}-1$.

\begin{lemma}\label{lem:characterizing-distributions}
    Let $0\le \lambda<1$ be a real number. Let $(\mu_n)_{n\ge1}$ be a sequence of elements in $\sL_n$ such that for all $n\ge1$, we have
    \begin{itemize}
        \item[a)] $||p^{\lambda (n-1)}\mu_n||\le 1$;
        \item[b)] the image of $\mu_{n+1}$ in $\sL_n$ is $\mu_n$.
    \end{itemize}
    Then there exists a unique $\mu\in\cD_{\lambda}$ such that the image of $\mu$ in $\sL_n$ is $\mu_n$ for all $n$, and $||\mu||_\lambda\le 1$.
\end{lemma}
\begin{proof}
    This is a special case of \cite[Lemma~2.2]{BFSuper}.
\end{proof}

 Recall that $z_{n,m}$ belongs to 
\[
H^1(K[m],T_f^\ast\otimes T_g^\ast\otimes \Lambda_n(\varphi_{0,\frk{P}}^c\bm{\kappa}_{\ac}^{-1})(2))
\simeq H^1(K[mp^{n}],T_f^\ast\otimes T_g^\ast(\varphi_{0,\frk{P}}^c)(2))\,.
\]
For notational simplicity, we write $M$ for the $G_{K[m]}$-representation $T_f^\ast\otimes T_g^\ast(\varphi_{0,\frk{P}}^c)(2)$. We shall identify $z_{n,m}$ with its image under the natural map
\[
H^1(K[mp],M\otimes \Lambda_n\bk)\to H^1(K[m],M\otimes\sL_n).
\]

 We write $||\cdot||_M$ for the natural norm on $M$ induced by the $\cO$-module structure of $M$ and the $p$-adic norm on $\cO$. For $G\subseteq G_{K[m]}$, the continuous cohomology group $H^1(G,M\otimes \sL)$ is endowed with a supremum semi-norm, which we still denote by $||\cdot||_M$. We use the same notation for the semi-norm on $H^1(G,M\otimes \sL_n)$.



The following proposition is a slight alteration of a special case of \cite[Proposition~2.3.2]{LZ1}.

\begin{proposition}\label{prop:inject-H1}
    For all real numbers $\lambda$ such that $0\le \lambda<1$, there is an injection $$\Phi_\lambda:H^1(K[mp^\infty],M\otimes \cD_{\lambda})\longrightarrow H^1(K[mp^\infty],M\otimes \sL).$$
    
    Let $\mu\in H^1(K[mp^\infty],M\otimes \sL)$, and write $\mu_n$ for its image in $H^1(K[mp^\infty],M\otimes \sL_n)$ induced by the natural map $\sL\rightarrow\sL_n$. Then $\mu$ is in the image of $\Phi_\lambda$ if and only if the sequence
    \[
    p^{-\lfloor \lambda (n-1)\rfloor}\left|\left|\mu_n\right|\right|_M
    \]
    is bounded as $n\rightarrow\infty$. Moreover, if $\left|\left|\mu_n\right|\right|_M\le p^{ \lambda (n-1)}$ for all $n\ge1$, then $||\mu||_{M,\lambda}\le 1$,
    where $||\cdot||_{M,\lambda}$ is the natural semi-norm on $H^1(K[mp^\infty],M\otimes\cD_{\lambda})$ induced by the norms on $M$ and $\cD_{\lambda}$.
\end{proposition}
\begin{proof}
The existence of the morphism $\Phi_\lambda$ follows from the natural injection $\cD_{\lambda}\hookrightarrow\sL$.
By \cite[Proposition~2.1.2]{LZ1}, the set of continuous coboundaries $B^1(G_{K[mp^\infty]},M)$ is closed in the set of continuous cocycles $Z^1(G_{K[mp^\infty]},M)$. Consequently, the injectivity of $\Phi_\lambda$ follows from \cite[proof of Proposition~II.2.1]{colmez98}. The remainder of the proposition follows by adapting the proof of \cite[Proposition~2.3.2]{LZ1}, replacing the use of Lemma~2.2.6 in op. cit. by Lemma~\ref{lem:characterizing-distributions} (which tells us that the constant $C$ appearing there can be taken to be $1$ in our setting). 
\end{proof}

For $n\ge2$, we write 
\begin{align*}
    \pr^n_{n-1}&:H^1({K[m]},M\otimes\sL_n)\rightarrow H^1({K[m]},M\otimes\sL_{n-1}),\\
    \Tr^n_{n-1}&:H^1({K[m]},M\otimes\sL_{n-1})\rightarrow H^1({K[m]},M\otimes\sL_{n})
\end{align*}
for the natural maps induced by the projection $\sL_n\rightarrow \sL_{n-1}$ and the trace map $\sL_{n-1}\rightarrow\sL_{n}$ induced by $\Gamma_{n-1}\ni\sigma\mapsto \displaystyle\sum_{\pr^n_{n-1}(\tau)=\sigma}\tau$. 
\begin{proposition}\label{prop:exist-H1}
    Let $\lambda$ be a real number such that $0\le \lambda<1$. Suppose we are given a sequence of elements $\mu_n\in H^1(K[m],M\otimes \sL_n)$, $n\ge0$ such that
    \begin{itemize}
        \item[i)] For all $n\ge2$, $\pr^{n}_{n-1}\mu_{n}=\mu_{n-1}$,
        \item[ii)] $||\mu_n||_M\le p^{ \lambda (n-1)}$ for all $n\ge1$.
    \end{itemize}
    Then there is a unique element $\mu\in H^1(K[m],M\otimes \cD_{\lambda})$ whose image in $H^1(K[m],M\otimes \sL_{n})$ is equal to $\mu_n$ for all $n$. Furthermore, $||\mu||_{M,\lambda}\le1$.
\end{proposition}
\begin{proof}
Under \eqref{item_H0}, we have $H^0(K[mp^\infty],M)=0$. Thus, by the inflation-restriction exact sequence and Shapiro's lemma, we have 
\[
H^1(K[mp^\infty],M\otimes \cD_{\lambda})^\Gamma\cong H^1(K[m],M\otimes \cD_{\lambda}),\quad H^1(K[mp^\infty],M\otimes L)^{\Gamma^{p^{n-1}}}\cong H^1(K[m],M\otimes \sL_{n}).
\]
In particular, we can identify $\mu_n$ as elements of
   $H^1(K[mp^\infty],M\otimes L)^{\Gamma^{p^{n-1}}}$. The proposition follows from \cite[proof of Proposition~2.3.3]{LZ1} after replacing Proposition 2.3.2 therein by Proposition~\ref{prop:inject-H1}.
\end{proof}

\begin{defn}
\label{def_2025_10_15_1649}
Let $\alpha$ and $\beta$ denote the roots of the Hecke polynomial of $g$ at $p$.
 For $\xi\in\{\alpha,\beta\}$ and $n\ge2$, define 
 \[z_{m,n,\xi}=\frac{1}{\xi^{n}}\left(z_{m,n}-\frac{\chi_g(p)}{\xi}\Tr_{n-1}^nz_{m,n-1}\right) \,\in\, H^1(K[m], M\otimes\sL_n).\]
\end{defn}

\begin{theorem}\label{thm:non-integral-classes}
    For $\xi\in\{\alpha,\beta\}$ with $\ord_p(\xi)<1$, there exists a unique element $$z_{m,\infty,\xi}\in H^1(K[m],M\otimes\cD_{\ord_p(\xi)})$$ such that its image under the natural morphism $
    H^1(K[m],M\otimes\sL)\longrightarrow H^1(K[m],M\otimes\sL_n)$ is equal to $z_{m,n,\xi}$. Furthermore, $||z_{m,\infty,\xi}||_{M,\ord_p(\xi)}\le |\xi|^{-2}$. 
\end{theorem}
\begin{proof}
    It follows from Proposition~\ref{prop_thm_2025_07_03_1627} that 
\[
\pr_{n-1}^nz_{m,n,\xi}=z_{m,n-1,\xi}.
\]
Therefore, the theorem is a direct consequence of Proposition~\ref{prop:exist-H1}. 
\end{proof}

\subsection{Integral signed classes}\label{subsec:integralsignedclasses}

We now explain how the classes given by Theorem~\ref{thm:non-integral-classes} decompose into integral ones under the assumption that $g$ is $p$-nonordinary. We first define the logarithmic matrix attached to $g$ (which finds its origin in \cite{sprung09}). The notation we use here is based on the discussion in \cite[\S2.2]{BFSuper}. We identify $\Lambda$ with $\cO\llbracket X\rrbracket$. For each $n\ge1$, let $\omega_n=(1+X)^{p^n}-1$.
\begin{defn}\label{defn_01_11_25}
    Let 
    \[A_g=\begin{bmatrix}
    0&\frac{-1}{\chi_g(p)p}\\ \\ 1&\frac{a_p(g)}{\chi_g(p)p}
    \end{bmatrix},\quad  Q_g=\begin{bmatrix}
        \alpha&-\beta\\ -\chi_g(p)p& \chi_g(p)p
    \end{bmatrix},\quad C_{g,n}=\begin{bmatrix}
    a_p(g)&1\\ -\chi_g(p)\Phi_n&0
    \end{bmatrix}\] for $n\ge1$, where $\Phi_n=\sum_{j=0}^{p-1}(1+X)^{jp^{n-1}}$ denotes the $p^n$-th cyclotomic polynomial in $1+X$. 

    The limit of the sequence $((\alpha-\beta)(A_g)^{n+1}C_{g,n}\cdots C_{g,1})_{n\ge1}$ is denoted by $M_{\log,g}$.
\end{defn}

It follows from \cite[Lemma~2.8]{BFSuper} that the entries in the first column of $Q_g^{-1}M_{\log,g}$ are $O(\log^{\ord_p(\alpha)})$, while those in the second column are $O(\log^{\ord_p(\beta)})$.

\begin{lemma}\label{lem:inverse-limit}
    For $n\ge1$, let $\Theta_{g,n}$ denote the map
    \[
   \Lambda^{\oplus 2}\rightarrow \cO[\Gamma_n]^{\oplus2},\quad \begin{bmatrix}
        F\\ G
    \end{bmatrix}\mapsto C_{g,n-1}\cdots C_{g,1}\begin{bmatrix}
        F\\ G
    \end{bmatrix}\mod \omega_{n-1}.
    \]
    Then $\displaystyle\varprojlim \Lambda^{\oplus 2}/\ker \Theta_{g,n}=\Lambda^{\oplus2}$.
\end{lemma}
\begin{proof}
See \cite[Lemma~2.9]{BFSuper}.
\end{proof}

For integers $n'\ge n\ge1$, let
\[
\pr^{n'}_{n}:\cO[\Gamma_{n'}]\to\cO[\Gamma_{n}]
\]
denote the natural projection map and write
\[
\Tr^{n'}_{n}:\cO[\Gamma_{n}]\to\cO[\Gamma_{n'}]
\]
for the trace map defined by $\Gamma_{n}\ni\sigma\mapsto \displaystyle\sum_{\pr^{n'}_n(\tau)=\sigma}\tau$.  By an abuse of notation, we use the same notation to denote the maps on the cohomological groups
\begin{align*}
\pr^{n'}_n&:H^1(K[m],M\otimes\Lambda_{n'}\bk)\to H^1(K[m],M\otimes\Lambda_{n}\bk),\\
\Tr_{n}^{n'}&:H^1(K[m],M\otimes\Lambda_{n}\bk)\to H^1(K[m],M\otimes\Lambda_{n'}\bk).
\end{align*}

\begin{lemma}\label{lem:image-trace}
    Let $n'\ge n\ge1$. The images of $\Tr_n^{n'}:H^1\left(K[m],M\otimes\Lambda_n\bk\right)\to H^1\left(K[m],M\otimes\Lambda_{n'}\bk\right)$ and the natural map $$H^1\left(K[m],M\otimes\left(\omega_{n'-1}/\omega_{n-1}\right)\Lambda_{n'}\bk\right)\to H^1\left(\QQ,M\otimes\Lambda_{n'}\bk\right)$$ coincide. Furthermore, if $z\in H^1(\QQ,M\otimes\Lambda_{n'}\bk)$, then $\omega_{n'-1}/\omega_{n-1}\cdot z=\Tr_{n}^{n'}\circ\pr^{n'}_n(z)$.
\end{lemma}
\begin{proof}
Let $\gamma$ be a topological generator of $\Gal(K[mp^\infty]/K[m])$ that corresponds to $1+X$ under the identification $\cO\llbracket\Gamma_\ac\rrbracket=\cO\llbracket X\rrbracket$. 
By Shapiro's lemma
\[
H^1\left(K[m],M\otimes\Lambda_r\bk\right)\cong H^1\left(K[mp^r],M\right), \quad r\in\{n,n'\}.
\]
The morphisms $\Tr_n^{n'}$ and $\pr_n^{n'}$ correspond to corestriction and restriction maps, respectively. The corestriction map can be realized by the action of $$\sum_{i=0}^{p^{n'-n}-1}\gamma^{ip^{n-1}}=\frac{\gamma^{p^{n'-1}}-1}{\gamma^{p^{n-1}}-1}=\omega_{n'-1}/\omega_{n-1},$$ 
 from which the lemma follows.    \end{proof}

In what follows, $\pr^\infty_n:H^1(K[m],M\otimes\Lambda\bk)\to H^1(K[m],M\otimes\Lambda_n\bk)$ denotes the natural projection map.

\begin{theorem}\label{thm:generic-decomposition}
  Let $(\kappa_n)_{n\ge 0}$ be a collection of classes such that $\kappa_n\in H^1(K[m],M\otimes\Lambda_n\bk)$ for all $n$ and
    \begin{equation}
    \pr^n_{n-1}\kappa_n=a_p(g)\kappa_{n-1}-\chi_g(p)\Tr_{n-2}^{n-1}\kappa_{n-2},\quad n\ge3.    \label{eq:trace-relation}
    \end{equation}
       Then there exist unique $\kappa_\sharp,\kappa_\flat\in H^1(K[m],M\otimes\Lambda\bk)$ such that
    \[
    \begin{bmatrix}
        \kappa_n\\-\chi_g(p)\Tr_{n-1}^n\kappa_{n-1}
    \end{bmatrix}
    =C_{g,n-1}\cdots C_{g,1}\begin{bmatrix}
        \pr^\infty_n \kappa_\sharp\\\pr^\infty_n \kappa_\flat
    \end{bmatrix}.
    \]\end{theorem}
\begin{proof}
Our proof is based on a slight modification of the proof presented in \cite[Theorem~3.4]{BBL1}.
Let $C_{g,i}'=\begin{bmatrix}
    0&-1\\ \chi_g(p)\Phi_i&a_p(g)
\end{bmatrix}$ be the adjugate matrix of $C_{g,i}$ so that $C_{g,i}C_{g,i}'=\det C_{g,i}$. We show inductively that for all $n\ge2$
\begin{equation}
C_{g,1}'\cdots C_{g,n-1}'\begin{bmatrix}
    \kappa_n\\-\chi_g(p)\Tr_{n-1}^n\kappa_{n-1}
\end{bmatrix}\in\image\Tr_{1}^n.
\label{eq:induction}    
\end{equation}

When $n=2$,
\[C_{g,1}'\begin{bmatrix}
    \kappa_2\\-\chi_g(p)\Tr_{1}^2\kappa_{1}
\end{bmatrix}=
\chi_g(p)\begin{bmatrix}
    \Tr_{1}^2\kappa_{1}\\ 
    \Phi_1\kappa_2-a_p(g)\Tr_{1}^2\kappa_{1}
\end{bmatrix}.\]
As $\Phi_1=\omega_1/\omega_0$, Lemma~\ref{lem:image-trace} says that $\Phi_1\kappa_2$ is in the image of $\Tr^2_1$. Thus, \eqref{eq:induction} holds for $n=2$.

Suppose that \eqref{eq:induction} holds for some $n\ge2$. Then
\begin{align*}
C_{g,1}'\cdots C_{g,n}'\begin{bmatrix}
    \kappa_{n+1}\\-\chi_g(p)\Tr_n^{n+1}\kappa_{n}
\end{bmatrix}&=\chi_g(p)C_{g,1}'\cdots C_{g,n-1}'\begin{bmatrix}
    \Tr_{n}^{n+1}\kappa_{n}\\
    \Phi_{n}\kappa_{n+1}-a_p(g)\Tr_{n}^{n+1}\kappa_{n}
\end{bmatrix}\\
    &=\chi_g(p)C_{g,1}'\cdots C_{g,n-1}'\begin{bmatrix}
    \Tr_{n}^{n+1}\kappa_{n}\\
    \Tr_{n}^{n+1}\circ\pr_n^{n+1}\kappa_{n+1}-a_p(g)\Tr_{n}^{n+1}\kappa_{n}
\end{bmatrix}\\
&=\chi_g(p)C_{g,1}'\cdots C_{g,n-1}'\begin{bmatrix}
    \Tr_{n}^{n+1}\kappa_n\\
    -\chi_g(p)\Tr_{n}^{n+1}\circ\Tr_{n-1}^{n}\kappa_{n-1}
\end{bmatrix}\\
&=\chi_g(p)\Tr_{n}^{n+1}\left(C_{g,1}'\cdots C_{g,n-1}'\begin{bmatrix}
    \kappa_{n}\\
    -\chi_g(p)\Tr_{n-1}^{n}\kappa_{n-1}
\end{bmatrix}\right),
\end{align*}
where the second equality follows from Lemma~\ref{lem:image-trace} and the third equality is a consequence of the trace relation given by \eqref{eq:trace-relation}.
Therefore, by the inductive hypothesis, \eqref{eq:induction} holds for $n+1$.

For each $n\ge0$, we can choose a $M\otimes\Lambda_n\bk$-valued cocycle on $G_{K[m]}$ recursively so that
\[
\pr^n_{n-1}c_n=a_p(g)c_{n-1}-\chi_g(p)\Tr_{n-2}^{n-1}c_{n-2},\quad n\ge2.
\]
Indeed, after $c_{n-1}$ and $c_{n-2}$ are chosen, any representative $c_n$ of $\kappa_n$ satisfies
\[
\pr^n_{n-1}c_n-a_p(g)c_{n-1}+\chi_g(p)\Tr_{n-2}^{n-1}c_{n-2}=b_x,
\]
where $b_x:\sigma\mapsto\sigma\cdot x-x$ is a coboundary defined by some element $x\in M\otimes\Lambda_{n-1}\bk$. We can replace $c_n$ by $c_n-b_{x'}$, where $b_{x'}$ is the coboundary defined by any pre-image $x'\in M\otimes\Lambda_n$ of $x$. Then
\begin{equation}
\label{eq:cocycle}\pr^n_{n-1}c_n-a_p(g)c_{n-1}+\chi_g(p)\Tr_{n-2}^{n-1}c_{n-2}=0,    
\end{equation}

For any $\sigma\in G_{K[m]}$,  \eqref{eq:induction} implies that there exist $c_{\sharp,n}(\sigma),c_{\flat,n}(\sigma)\in M\otimes\Lambda\bk$ such that
\[\begin{bmatrix}
    c_n(\sigma)\\-\chi_g(p)\Tr_{n-1}^nc_{n-1}(\sigma)
\end{bmatrix}=C_{g,n-1}\cdots C_{g,1}\begin{bmatrix}
    c_{\sharp,n}(\sigma)\\c_{\flat,n}(\sigma)
\end{bmatrix}\mod\omega_{n-1}
\]
since $\det(C_{g,n-1}\cdots C_{g,1})=\chi_g(p)^{n-1}\prod_{i=1}^{n-1}\Phi_i$. Let $\Theta_{g,n}$ denote the map on $\Lambda^2_n$ given as multiplication by $C_{g,n-1}\cdots C_{g,1}$ as in Lemma~\ref{lem:inverse-limit}. 
As $c_n$ and $c_{n-1}$ are cocycles on $G_{K[m]}$, the map $\sigma\mapsto(c_{\sharp,n}(\sigma),c_{\flat,n}(\sigma))$ defines a $M\otimes\Lambda_n^2\bk/\ker\Theta_{g,n}$-valued cocycle on $G_{K[m]}$.
Furthermore,
\begin{align*}
    \begin{bmatrix}
    c_{n+1}\\-\chi_g(p)\Tr_{n}^{n+1}c_{n}
\end{bmatrix}&\equiv C_{g,n}C_{g,n-1}\cdots C_{g,1}\begin{bmatrix}
    c_{\sharp,n+1}\\c_{\flat,n+1}
\end{bmatrix}&\mod\omega_{n}\\
\begin{bmatrix}
    a_p(g)c_{n}-\chi_g(p)\Tr_{n-1}^nc_{n-1}\\-\chi_g(p)pc_{n}
\end{bmatrix}&\equiv \begin{bmatrix}
    a_p(g)&1\\-\chi_g(p)p&0
\end{bmatrix}C_{g,n-1}\cdots C_{g,1}\begin{bmatrix}
    \pr_n^{n+1}c_{\sharp,n+1}\\\pr_n^{n+1}c_{\flat,n+1}
\end{bmatrix}&\mod\omega_{n-1}\\
\begin{bmatrix}
    c_{n}\\-\chi_g(p)\Tr_{n-1}^nc_{n-1}
\end{bmatrix}&\equiv C_{g,n-1}\cdots C_{g,1}\begin{bmatrix}
    \pr_n^{n+1}c_{\sharp,n+1}\\\pr_n^{n+1}c_{\flat,n+1}
\end{bmatrix}&\mod\omega_{n-1},
\end{align*}
where the second congruence is a consequence of \eqref{eq:cocycle}.
Thus, $(c_{\sharp,n},c_{\flat,n})$ form an inverse system of cocycles taking values in 
$\displaystyle\varprojlim_n M\otimes\Lambda_n^2\bk/\ker\Theta_{g,n}=M\otimes\Lambda^2\bk$ (where the validity of the equality follows from Lemma~\ref{lem:inverse-limit}). This concludes the proof.
\end{proof}

\begin{theorem}\label{thm:decompo}
    There exist unique $z_m^\sharp,z_m^\flat\in H^1(K[m],M\otimes\LL\bk)$ such that
    \[
    \begin{bmatrix}
      z_{m,\infty,\alpha}\\  z_{m,\infty,\beta}\end{bmatrix}=Q_g^{-1}M_{\log,g}
\begin{bmatrix}
       z_m^\sharp\\  z_m^\flat
\end{bmatrix}    .
    \]
    Furthermore,  let $q$ be a prime coprime to $Nmp$ and split in $K$. Then
    \begin{align*}
        \mathrm{cor}_{K[mq]/K[m]}z_{m q}^\bullet &= \chi_g^{-1}(q)\Bigg\{ q\chi_f(q)\chi_g(q) \left(\frac{\varphi_0(\frk{q})}{q}\Fr_{\frk{q}}\right)^2 -a_q(f)a_q(g) \left(\frac{\varphi_0(\frk{q})}{q}\Fr_{\frk{q}}\right) \\
        & +\frac{\chi_f^{-1}(q)a_q(f)^2}{q}+ \chi_{g}^{-1}(q)a_{q}(g)^2-\frac{q^2+1}{q} \\
        &-a_q(f)a_q(g) \left(\frac{\varphi_0(\overline{\frk{q}})}{q}\Fr_{\overline{\frk{q}}}\right) +q\chi_f(q)\chi_g(q) \left(\frac{\varphi_0(\overline{\frk{q}})}{q}\Fr_{\overline{\frk{q}}}\right)^2 \Bigg\} z_{m}^\bullet
    \end{align*}
    for $\bullet\in\{\sharp,\flat\}$.
\end{theorem}
\begin{proof}
Recall that for $\xi\in\{\alpha,\beta\}$, the image of $ z_{m,\infty,\xi}$ in $H^1(K[m],M\otimes\sL_n)$ is given by
\[z_{m,n,\xi}=\frac{1}{\xi^{n}}\left(z_{m,n}-\frac{\chi_g(p)}{\xi}\Tr_{n-1}^nz_{m,n-1}\right).\]
As $Q_g^{-1}A_gQ_g=\begin{bmatrix}
    \alpha^{-1}&0\\0&\beta^{-1}
\end{bmatrix}$, we have
\begin{align*}
A_g^{-n}Q_g\begin{bmatrix}
    z_{m,n,\alpha}\\z_{m,n,\beta}
\end{bmatrix}&=\begin{bmatrix}
    \alpha&-\beta\\-\alpha\beta&\alpha\beta
\end{bmatrix}\begin{bmatrix}
    \alpha^{n}&0\\0&\beta^{n}
\end{bmatrix}\begin{bmatrix}
    \alpha^{-n}&\alpha^{-n-1}\\  \beta^{-n}&\beta^{-n-1}
\end{bmatrix}\begin{bmatrix}
    z_{m,n}\\-\chi_g(p) \Tr_{n-1}^nz_{m,n-1}
\end{bmatrix}\\
    &=(\alpha-\beta)\begin{bmatrix}
    z_{m,n}\\-\chi_g(p) \Tr_{n-1}^nz_{m,n-1}
\end{bmatrix}.
\end{align*}
Therefore, by Proposition~\ref{prop_thm_2025_07_03_1627} and Theorem \ref{thm:generic-decomposition}, there exist unique classes $z_m^\sharp,z_m^\flat\in H^1(K[m],M\otimes\LL\bk)$ such that 
\[
A_g^{-n}Q_g\begin{bmatrix}
    z_{m,n,\alpha}\\z_{m,n,\beta}
\end{bmatrix}=(\alpha-\beta)C_{g,n-1}\cdots C_{g,1}\begin{bmatrix}
\pr_n^\infty z_m^\sharp\\ \pr_n^\infty z_m^\flat
\end{bmatrix}.
\]
Letting $n\rightarrow \infty$, the first assertion of the theorem follows.

    Let
     \begin{align*}
        \cP_q &= \chi_g^{-1}(q)\Bigg\{ q\chi_f(q)\chi_g(q) \left(\frac{\varphi_0(\frk{q})}{q}\Fr_{\frk{q}}\right)^2 -a_q(f)a_q(g) \left(\frac{\varphi_0(\frk{q})}{q}\Fr_{\frk{q}}\right) \\
        & +\frac{\chi_f^{-1}(q)a_q(f)^2}{q}+ \chi_{g}^{-1}(q)a_{q}(g)^2-\frac{q^2+1}{q} \\
        &-a_q(f)a_q(g) \left(\frac{\varphi_0(\overline{\frk{q}})}{q}\Fr_{\overline{\frk{q}}}\right) +q\chi_f(q)\chi_g(q) \left(\frac{\varphi_0(\overline{\frk{q}})}{q}\Fr_{\overline{\frk{q}}}\right)^2 \Bigg\} .
    \end{align*}
    It follows from Proposition~\ref{prop:horizontal-norm} and  the first assertion of the theorem applied to $mq$ that
     \[
    \begin{bmatrix}
      \cP_q\cdot z_{m,\infty,\alpha}\\   \cP_q  \cdot z_{m,\infty,\beta}\end{bmatrix}=Q_g^{-1}M_{\log,g}
\begin{bmatrix}
       \cor_{K[mq]/K[m]}z_{mq}^\sharp\\\cor_{K[mq]/K[m]}  z_{mq}^\flat
\end{bmatrix}    .
    \]
After applying $\cP_q$ to both sides of the equation given by the first assertion of the theorem, we have
      \[
    \begin{bmatrix}
      \cP_q\cdot z_{m,\infty,\alpha}\\   \cP_q  \cdot z_{m,\infty,\beta}\end{bmatrix}=Q_g^{-1}M_{\log,g}
\begin{bmatrix}
       \cP_q\cdot z_{m}^\sharp\\\cP_q \cdot z_{m}^\flat
\end{bmatrix}    .
    \]
    Therefore, by the uniqueness of the signed classes, 
    \[
    \begin{bmatrix}
       \cor_{K[mq]/K[m]}z_{mq}^\sharp\\\cor_{K[mq]/K[m]}  z_{mq}^\flat
\end{bmatrix}   = \begin{bmatrix}
       \cP_q\cdot z_{m}^\sharp\\\cP_q \cdot z_{m}^\flat
\end{bmatrix} ,
    \]
        and the second assertion follows.
\end{proof}

\section{Arithmetic applications}
\label{sec:applications}
In this section, we use the signed classes constructed in \S\ref{subsec:integralsignedclasses} to obtain applications towards signed main conjectures. We keep the assumptions and notations introduced in previous sections, and for a technical reason (cf. Remark~\ref{rem_2025_10_13_0908}), we restrict our discussion to the case $a_p(g)=0$. 

\subsection{Coleman maps and Selmer groups}
\label{subsec_5_1_2025_11_02_1128}
In this subsection, we introduce Coleman maps and use them to define the Selmer groups that we will consider in our applications. 

Since $f$ is $p$-ordinary, we have an exact sequence of $G_{\bb{Q}_p}$-representations
\[
0\rightarrow \mathscr{F}^+V_f^\ast\rightarrow V_f^\ast \rightarrow \mathscr{F}^- V_f^\ast\rightarrow 0,
\]
where $\mathscr{F}^+V_f^\ast$ and $\mathscr{F}^-V_f^\ast$ are one-dimensional with Hodge--Tate weights $0$ and $-1$, respectively, with the convention that the Hodge--Tate weight of the cyclotomic character is $+1$. Also, $\mathscr{F}^+V_f^\ast$ is unramified with the geometric Frobenius $\Fr_p$ acting as multiplication by $\alpha_f$ (the unit root of the Hecke polynomial of $f$ at the prime $p$). We define
\[
\mathscr{F}^+T_f^\ast=T_f^\ast\cap \mathscr{F}^+V_f^\ast \quad \text{and} \quad \mathscr{F}^-T_f^\ast=T_f^\ast/\mathscr{F}^+T_f^\ast.
\]
We also define
\[
\mathscr{F}^+M=\mathscr{F}^+T_f^\ast\otimes T_g^\ast(\varphi_{0,\frk{P}}^c)(2) \quad \text{and} \quad \mathscr{F}^-M=\mathscr{F}^-T_f^\ast\otimes T_g^\ast(\varphi_{0,\frk{P}}^c)(2)=M/\mathscr{F}^+M.
\]
We remark that the $G_{K_{\overline{\frk{p}}}}$-representation $\mathscr{F}^+M\vert_{G_{K_{\overline{\frk{p}}}}}$ and the $G_{K_{\frk{p}}}$-representation $\mathscr{F}^-M\vert_{G_{K_{\frk{p}}}}$ both have Hodge–Tate weights $1$ and $0$.

Let $(\fq,T)\in\{(\overline\fp,\sF^+M),(\fp,\sF^-M)\}$. Let $F$ be a finite unramified extension of $K_\fq$. We write 
\[
\cL_{T,\fq,F}:H^1(F,T\otimes\Lambda_\cyc^{\iota})\to\cH(\Gamma_\cyc)\otimes\Dcris(F,T)
\]
for the Perrin-Riou map, where $\Gamma_\cyc=\Gal(K_\fq(\mu_{p^\infty})/K_\fq)$,  $\Lambda_\cyc$ is the Iwasawa algebra $\Zp\llbracket\Gamma_\cyc\rrbracket$, $\iota$ denotes the Galois action of $G_F$ which is the inverse of the one induced by $G_F\rightarrow \Gamma_\cyc\hookrightarrow\Lambda_\cyc$ and $\cH(\Gamma_\cyc)$ is the algebra of tempered distributions. Following the construction in \cite{LZ0}, after taking inverse limit of $\cL_{T,\fq,F_n}$, where $F_n$ is the unramified extension of degree $p^n$ over $F$, we obtain a ``two-variable'' Perrin-Riou map:
\[
\cL_{T,\fq,F_\infty}:H^1(F,T\otimes\Lambda_\cyc^{\iota}\hat\otimes\Lambda\bk)\to \cH(\Gamma_\cyc)\hat\otimes\Lambda\otimes\Dcris(F,T).
\]
Here, we implicitly fix a basis  $\Omega_F$ of the Yager module, as in \cite[\S2.3]{BFSuper}, and we identify $\Lambda_\cyc\hat\otimes\Lambda$ with the Iwasawa algebra $\Zp\llbracket\Gal(F_\infty(\mu_{p^\infty})/F)\rrbracket\cong\Zp\llbracket\Gamma_\cyc\times\Gamma_\ac\rrbracket$.

Given a finite extension $\mathcal{K}$ of $K$ and a place $\lambda$ of $K$, we write
\[
H^1(\mathcal{K}_\lambda,T):=\bigoplus_{v\mid \lambda} H^1(\mathcal{K}_v,T).
\]
Let $m$ be a square-free product of primes coprime to $Np$ and split in $K$. 
We denote by
\[
\mathcal{L}_{T,\fq,m}:H^1(K[m]_{\fq},T\otimes \Lambda(\bm{\kappa}_{\ac}^{-1})) \longrightarrow \mathcal{O}[\Gal(K[m]/K)]\otimes\mathcal{H}(\Gamma_{\ac})\otimes\Dcris(T\vert_{G_{K_{\frk{q}}}})
\]
the natural map induced by $\bigoplus\cL_{T,\fq,K[m]_{v,\infty}}$ under projection, where $v$ runs over the places lying above $\fq$.

Let $v_0$ be an $\mathcal{O}$-basis of the $\mathcal{O}$-module $\Fil^1\Dcris(T_g\vert_{G_{\bb{Q}_p}})$. Let $\alpha$ and $\beta$ denote the roots of the Hecke polynomial of $g$ at $p$. For $\xi\in\lbrace \alpha,\beta\rbrace$, define the $\varphi$-eigenvector
\[
v_{g,\xi}=\varphi(v_0)-\frac{1}{\xi}v_0 \in \Dcris(V_g\vert_{G_{\bb{Q}_p}})
\]
with $\varphi$-eigenvalue $1/{\xi'}$, where $\xi'\in\lbrace \alpha,\beta\rbrace\setminus\lbrace \xi\rbrace$.

Fix an $\mathcal{O}$-basis $w_{\overline{\frk{p}},+}$ of the $\mathcal{O}$-module $\Dcris(\Hom(\mathscr{F}^+ T_f^\ast(\varphi_{0,\frk{P}}^c),\mathcal{O}(-1))\vert_{G_{K_{\overline{\frk{p}}}}})$ and an $\mathcal{O}$-basis $w_{\frk{p},-}$ of the $\mathcal{O}$-module $\Dcris(\Hom(\mathscr{F}^- T_f^\ast(\varphi_{0,\frk{P}}^c),\mathcal{O}(-1))\vert_{G_{K_\frk{p}}})$.
Let $v_{\overline{\frk{p}},+,\xi}=v_{g,\xi}^c\otimes w_{\overline{\frk{p}},+}$ and $v_{\frk{p},-,\xi}=v_{g,\xi}\otimes w_{\frk{p},-}$. 
Then $\lbrace v_{\overline{\frk{p}},+,\alpha},v_{\overline{\frk{p}},+,\beta}\rbrace$ is a $\varphi$-eigenbasis of $\Dcris(\Hom(\mathscr{F}^{+}M,L(1))\vert_{G_{K_{\overline{\frk{p}}}}})$ and $\lbrace v_{\frk{p},-,\alpha},v_{\frk{p},-,\beta}\rbrace$ is a $\varphi$-eigenbasis of $\Dcris(\Hom(\mathscr{F}^{-}M,L(1))\vert_{G_{K_{\frk{p}}}})$.
The calculations in \cite[\S2.3]{BFSuper} show that
\begin{equation}
\begin{bmatrix}
    \langle \mathcal{L}_{\mathscr{F}^+ M,\overline{\frk{p}},m}(-),v_{\overline{\frk{p}},+,\alpha}\rangle \\
    \langle \mathcal{L}_{\mathscr{F}^+ M,\overline{\frk{p}},m}(-),v_{\overline{\frk{p}},+,\beta}\rangle
\end{bmatrix}
= Q_g^{-1}M_{\log,g}
\begin{bmatrix}
    \col_{\mathscr{F}^+ M,\overline{\frk{p}},m}^{\sharp} \\
    \col_{\mathscr{F}^+ M,\overline{\frk{p}},m}^{\flat}
\end{bmatrix}
\label{eq:Col-decompo+}    
\end{equation}
for certain homomorphisms of $\Lambda[\Gal(K[m]/K)]$-modules
\[
\col_{\mathscr{F}^+ M,\overline{\frk{p}},m}^{\sharp/\flat}:H^1(K[m]_{\overline{\frk{p}}},\mathscr{F}^+ M\otimes \Lambda(\bm{\kappa}_{\ac}^{-1}))\longrightarrow \Lambda[\Gal(K[m]/K)]
\]
and similarly
\begin{equation}
\begin{bmatrix}
    \langle \mathcal{L}_{\mathscr{F}^- M,\frk{p},m}(-),v_{\frk{p},-,\alpha}\rangle \\
    \langle \mathcal{L}_{\mathscr{F}^- M,\frk{p},m}(-),v_{\frk{p},-,\beta}\rangle
\end{bmatrix}
= Q_g^{-1}M_{\log,g}
\begin{bmatrix}
    \col_{\mathscr{F}^- M,\frk{p},m}^{\sharp} \\
    \col_{\mathscr{F}^- M,\frk{p},m}^{\flat}
\end{bmatrix}
\label{eq:Col-decompo-}    
\end{equation}
for certain homomorphisms of $\Lambda[\Gal(K[m]/K)]$-modules
\[
\col_{\mathscr{F}^- M,\frk{p},m}^{\sharp/\flat}:H^1(K[m]_{\frk{p}},\mathscr{F}^- M\otimes \Lambda(\bm{\kappa}_{\ac}^{-1}))\longrightarrow \Lambda[\Gal(K[m]/K)].
\]

We assume that the following hypothesis holds until the end of \S\ref{sec:applications}.
\begin{itemize}
        \item[(\mylabel{item_H0minus}{\textbf{H}$^0_-$})]     We have  $H^0(K_{\overline{\frk{p}}},\mathscr{F}^-\overline{M})=0$.
    \end{itemize}

\begin{remark}\label{rk:irred2}
    As a consequence of \eqref{item_H0minus}, for any finite $p$-extension $F$ of $K_{\overline{\frk{p}}}$ we have $H^0(F,\mathscr{F}^-\overline{M})=0$.
\end{remark}

\begin{lemma}
    Let $m$ be a square-free product of primes coprime to $Np$ and split in $K$. The map
    \[
    H^1(K[m]_{\overline{\frk{p}}},\mathscr{F}^+M\otimes \Lambda(\bm{\kappa}_{\ac}^{-1}))\rightarrow H^1(K[m]_{\overline{\frk{p}}},M\otimes \Lambda(\bm{\kappa}_{\ac}^{-1}))
    \]
    induced by the inclusion $\mathscr{F}^+M\hookrightarrow M$ is injective.
\end{lemma}
\begin{proof}
    We have an exact sequence
    \[
    H^0(K[m]_{\overline{\frk{p}}},\mathscr{F}^-M\otimes \Lambda(\bm{\kappa}_{\ac}^{-1}))\rightarrow H^1(K[m]_{\overline{\frk{p}}},\mathscr{F}^+M\otimes \Lambda(\bm{\kappa}_{\ac}^{-1}))\rightarrow H^1(K[m]_{\overline{\frk{p}}},M\otimes \Lambda(\bm{\kappa}_{\ac}^{-1})), 
    \]
    so it suffices to show that $H^0(K[m]_{\overline{\frk{p}}},\mathscr{F}^-M\otimes \Lambda(\bm{\kappa}_{\ac}^{-1}))=0$. Now, by \eqref{item_H0minus} and Remark~\ref{rk:irred2}, we have that $H^0(K[mp^n]_{\overline{\frk{p}}},\mathscr{F}^-M)=0$ for all positive integers $n$. Therefore
    \[
    H^0(K[m]_{\overline{\frk{p}}},\mathscr{F}^-M\otimes \Lambda(\bm{\kappa}_{\ac}^{-1}))=\varprojlim_n H^0(K[mp^n]_{\overline{\frk{p}}},\mathscr{F}^-M)=0,
    \]
    which concludes the proof.
\end{proof}

In light of the previous lemma, we can regard the module $H^1(K[m]_{\overline{\frk{p}}},\mathscr{F}^+M\otimes \Lambda(\bm{\kappa}_{\ac}^{-1}))$ as a submodule of $H^1(K[m]_{\overline{\frk{p}}},M\otimes \Lambda(\bm{\kappa}_{\ac}^{-1}))$.

\begin{defn}
    Let $m$ be a square-free product of primes coprime to $Np$ and split in $K$. Let $\bullet\in\lbrace \sharp,\flat\rbrace$.
        \item[i)] The local condition $H^1_{\ord,\bullet}(K[mp^\infty]_{\frk{p}},M)\subseteq H^1_{\Iw}(K[mp^\infty]_{\frk{p}},M)$ is the kernel of the compositum
        \[
        H^1_{\Iw}(K[mp^\infty]_{\frk{p}},M)\longrightarrow H^1_{\Iw}(K[mp^\infty]_{\frk{p}},\mathscr{F}^-M)\xrightarrow{\col_{\mathscr{F}^-M,\frk{p},m}^\bullet} \Lambda[\Gal(K[m]/K)],
        \]
        where the first arrow is induced by the canonical projection $M\rightarrow \mathscr{F}^-M$.
        \item[ii)] The local condition $H^1_{\ord,\bullet}(K[mp^\infty]_{\overline{\frk{p}}},M)\subseteq H^1_{\Iw}(K[mp^\infty]_{\overline{\frk{p}}},M)$  is the kernel of the map
        \[
        \col_{\mathscr{F}^+M,\overline{\frk{p}},m}^\bullet:H^1_{\Iw}(K[mp^\infty]_{\overline{\frk{p}}},\mathscr{F}^+M)\longrightarrow \Lambda[\Gal(K[m]/K)].
        \]
\end{defn}

At primes $\lambda$ of $K$ not dividing $p$, we will usually work with the following unramifed condition:
\[
H^1_{\ur}(K[mp^\infty]_\lambda,M)=\ker\left(H^1(K[m]_\lambda,M\otimes \Lambda(\bm{\kappa}_{\ac}^{-1}))\longrightarrow H^1(K[m]^{\ur}_\lambda,M\otimes \Lambda(\bm{\kappa}_{\ac}^{-1}))\right).
\]

\begin{defn}\label{def:Selmergroups}
    Let $\Sigma$ be a finite set of finite places of $K$ including all places dividing $Np$. Let $m$ be a square-free product of primes coprime to $Np$ and split in $K$. Let $K[m](\Sigma)$ denote the maximal extension of $K[m]$ unramified outside the places of $K[m]$ lying above $\Sigma$. Let $\bullet,\star\in\lbrace \sharp,\flat\rbrace$. Then we define
    \begin{align*}
        H^{1,\Sigma}_{\emptyset,\emptyset}(K[mp^\infty],M)&=H^1(K[m](\Sigma)/K[m],M\otimes \Lambda(\bm{\kappa}_{\ac}^{-1})), \\
        H^{1}_{\emptyset,\emptyset}(K[mp^\infty],M)&=\ker\left(H^{1,\Sigma}_{\emptyset,\emptyset}(K[mp^\infty],M)\longrightarrow \prod_{\lambda\in \Sigma\setminus \lbrace \frk{p},\overline{\frk{p}}\rbrace}\frac{H^1_{\Iw}(K[mp^\infty]_{\lambda}, M)}{H^1_{\ur}(K[mp^\infty]_{\lambda},M)}\right) \\
        H^{1}_{\bullet,\star}(K[mp^\infty],M)&=\ker\left(H^{1}_{\emptyset,\emptyset}(K[mp^\infty],M)\longrightarrow \frac{H^1_{\Iw}(K[mp^\infty]_{\frk{p}}, M)}{H^1_{\ord,\bullet}(K[mp^\infty]_{\frk{p}},M)} \bigoplus \frac{H^1_{\Iw}(K[mp^\infty]_{\overline{\frk{p}}}, M)}{H^1_{\ord,\star}(K[mp^\infty]_{\overline{\frk{p}}},M)} \right).
    \end{align*}
\end{defn}

\subsection{Verifying the local conditions} 
\label{subsec_5_2_2025_11_02_1128}
The goal of this subsection is to prove that, for $\bullet\in\lbrace \sharp,\flat\rbrace$,
the classes $z_m^\bullet\in H^1_{\Iw}(K[mp^\infty],M)$ introduced in Theorem~\ref{thm:decompo} belong to the Selmer groups $H^{1}_{\bullet,\bullet}(K[mp^\infty],M)$.

\begin{proposition}\label{prop:verify-local}
    Let $n$ be a positive integer and let $m$ be a square-free product of primes coprime to $Np$ and split in $K$. Let $v$ be a place of $K[mp^n]$ above $\overline{\frk{p}}$. Then
    \[
    \res_v(z_{n,m})\in H^1(K[mp^n]_v,\mathscr{F}^+M).
    \]
\end{proposition}
\begin{proof}
    It suffices to show that the image of $\res_v(z_{n,m})$ in $H^1(K[mp^n]_v,\mathscr{F}^-M)$ is zero. By \cite[Thm.~5.9]{NN16}, we have that $\res_v(z_{n,m})\in H^1_g(K[mp^n]_v,M)$, so the image of $\res_v(z_{n,m})$ in $H^1(K[mp^n]_v,\mathscr{F}^-M)$ lies in $H^1_g(K[mp^n]_v,\mathscr{F}^-M)$. We will prove that $H^1_g(K[mp^n]_v,\mathscr{F}^-M)=0$.
    
    Let $V=\mathscr{F}^-M\otimes_{\mathcal{O}}L$, which we regard as a $G_{K[mp^n]_v}$-representation. Since $V$ has Hodge--Tate weights $0$ and $-1$, we have $\DdR(V)/\DdR^+(V)=0$. By \eqref{item_H0minus} and Remark~\ref{rk:irred2}, we have $H^0(K[mp^n]_v,V)=0$. Also, since $0<\ord_p(\alpha),\ord_p(\beta)<1$, we have that $\Dcris(V^\vee(1))^{\varphi=1}=0$, where $\varphi$ denotes the crystalline Frobenius. Hence,
    \[
    \dim H^1_g(K[mp^n]_v,V)=\dim \DdR(V)/\DdR^+(V)+\dim H^0(K[mp^n]_v,V) +\dim \Dcris(V^\vee(1))^{\varphi=1}=0. 
    \]
    Therefore the $\mathcal{O}$-module
    \[
    H^1_g(K[mp^n]_v,\mathscr{F}^-M)=\ker\left(H^1(K[mp^n]_v,\mathscr{F}^-M)\longrightarrow H^1(K[mp^n]_v,V)/H^1_g(K[mp^n]_v,V)\right)
    \]
    is torsion. Since $H^1(K[mp^n]_v,\mathscr{F}^-M)$ is a free $\mathcal{O}$-module by hypothesis \eqref{item_H0minus} and Remark~\ref{rk:irred2}, it follows that $H^1_g(K[mp^n]_v,\mathscr{F}^-M)=0$.
\end{proof}

\begin{corollary}\label{cor:F+}
    Let $m$ be a square-free product of primes coprime to $Np$ and split in $K$. Then,
    \[
    \res_{\overline{\frk{p}}}(z_{m,\infty,\xi})\in H^1(K[m]_{\overline{\frk{p}}},\mathscr{F}^+M\otimes\cD_{\ord_p(\xi)})\,,\qquad  \xi\in\lbrace \alpha,\beta\rbrace\,,
    \]
    \[
    \res_{\overline{\frk{p}}}(z_{m}^\bullet)\in H^1(K[m]_{\overline{\frk{p}}},\mathscr{F}^+M\otimes\Lambda(\bm{\kappa}_{\ac}^{-1}))\,,\qquad \bullet\in\lbrace \sharp,\flat\rbrace\,.
    \]
\end{corollary}
\begin{proof}
    This follows easily from the previous proposition and the construction of the classes $z_{m,\infty,\xi}$, $z_{m}^\bullet$.
\end{proof}

In what follows, we assume that the following hypothesis holds:
 \begin{itemize}
        \item[(\mylabel{item_LOC}{\textbf{LOC}})] $a_p(g)=0$.
    \end{itemize}

\begin{proposition}\label{prop:vanishing-col}
    Let $m$ be a square-free product of primes coprime to $Np$ and split in $K$. Suppose that \eqref{item_LOC} holds. Then, for $\bullet\in\lbrace \sharp,\flat\rbrace$, we have 
    \[
    \col^\bullet_{\mathscr{F}^-M,\frk{p},m}(\res_{\frk{p}}(z_m^\bullet))=0 \quad \text{and} \quad \col^\bullet_{\mathscr{F}^+M,\overline{\frk{p}},m}(\res_{\overline{\frk{p}}}(z_m^\bullet))=0.
    \]
\end{proposition}
\begin{proof}
We prove the first equality; the second follows by a completely analogous argument.

Since we are assuming that $a_p(g)=0$, the matrix $M_{\log,g}$ is of the form $\begin{bmatrix}
    *\log^-&*\log^+\\
    *\log^-&*\log^+
\end{bmatrix}$, where $*$ represents a non-zero constant and $\log^\pm$ are Pollack's plus and minus logarithms defined in \cite{pollack03} (see for example \cite[Appendix~4]{buyukboduklei0}).

Let $\eta$ be a character of order $p^n$ of $\Gamma_\ac$. Note that 
\begin{align*}
&M_{\log,g}\begin{bmatrix}
     \col^\sharp_{\mathscr{F}^-M,\frk{p},m}(\res_{\frk{p}}(z_m^\sharp)) & \col^\sharp_{\mathscr{F}^-M,\frk{p},m}(\res_{\frk{p}}(z_m^\flat)) \\
      \col^\flat_{\mathscr{F}^-M,\frk{p},m}(\res_{\frk{p}}(z_m^\sharp)) & \col^\flat_{\mathscr{F}^-M,\frk{p},m}(\res_{\frk{p}}(z_m^\flat)) 
\end{bmatrix}M_{\log,g}^T \\
&= \begin{bmatrix}
    \langle \mathcal{L}_{\mathscr{F}^-M,\frk{p},m}(\res_{\frk{p}}(z_{m,\infty,\alpha})),v_{\frk{p},-,\alpha} \rangle & \langle \mathcal{L}_{\mathscr{F}^-M,\frk{p},m}(\res_{\frk{p}}(z_{m,\infty,\beta})),v_{\frk{p},-,\alpha} \rangle \\
    \langle \mathcal{L}_{\mathscr{F}^-M,\frk{p},m}(\res_{\frk{p}}(z_{m,\infty,\alpha})),v_{\frk{p},-,\beta} \rangle & \langle \mathcal{L}_{\mathscr{F}^-M,\frk{p},m}(\res_{\frk{p}}(z_{m,\infty,\beta})),v_{\frk{p},-,\beta} \rangle
\end{bmatrix}.
\end{align*}
Since $\res_{\frk{p}}(z_{m,n})\in H^1_g(K[mp^n]_{\frk{p}},M)$ by \cite[Thm.~5.9]{NN16}, the above matrix vanishes when evaluated at $\eta$ by the interpolation property of the Perrin-Riou map $\mathcal{L}_{\mathscr{F}^-M,\frk{p},m}$. If $n$ is even, we have $\log^-(\eta)=0$. In particular, 
\[
\log^+(\eta)^2\col^\flat_{\mathscr{F}^-M,\frk{p},m}(\res_{\frk{p}}(z_m^\flat)) (\eta)=0.
\]
But $\log^+(\eta)\ne0$, which implies that $\col^\flat_{\mathscr{F}^-M,\frk{p},m}(\res_{\frk{p}}(z_m^\flat)) (\eta)=0$. As this holds for infinitely many $\eta$, we deduce that $\col^\flat_{\mathscr{F}^-M,\frk{p},m}(\res_{\frk{p}}(z_m^\flat))=0$. On considering odd $n$, we deduce by the same argument that $\col^\sharp_{\mathscr{F}^-M,\frk{p},m}(\res_{\frk{p}}(z_m^\sharp))=0$.
\end{proof}

\begin{proposition}\label{prop:local-ap=0}
    Let $m$ be a square-free product of primes coprime to $Np$ and split in $K$. If {\rm \eqref{item_LOC}} holds, then for $\bullet\in\lbrace \sharp,\flat\rbrace$, we have that $z_m^\bullet\in H^{1}_{\bullet,\bullet}(K[mp^\infty],M)$.
\end{proposition}
\begin{proof}
    Let $n$ be a positive integer. Let $V$ denote the $p$-adic $G_{K[mp^n]}$-representation $M\otimes_{\mathcal{O}}L\vert_{G_{K[mp^n]}}$, which is pure of motivic weight $-1$. Let $v$ be a finite place of $K[mp^n]$ not dividing $Np$ and let $I_v$ denote the inertia subgroup of $G_{K[mp^n]_v}$. By the weight-monodromy conjecture, established for modular curves by Saito \cite{Saito97}, the eigenvalues of a geometric Frobenius $\Fr_v$ acting on $V^{I_{v}}$ have complex absolute value $q_v^{1/2}$, where $q_v$ is the cardinality of the residue field at $v$. In particular, it follows that $H^0(K[mp^n]_v,V)=0$. By the same argument we have that $H^0(K[mp^n]_{\overline{v}},V)=0$. Therefore
    \begin{align*}
    \dim H^2(K[mp^n]_v,V)&=\dim H^2(K[mp^n]_{\overline{v}},V^c) \\
    &=\dim H^2(K[mp^n]_{\overline{v}},V^\vee(1))=\dim H^0(K[mp^n]_{\overline{v}},V)=0.
    \end{align*}
    By the Euler--Poicar\'e characteristic formula, we deduce that
    \[
    \dim H^1(K[mp^n]_v,V)=\dim H^0(K[mp^n]_v,V)+\dim H^2(K[mp^n]_v,V)=0.
    \]
    Thus, it trivially holds that
    \[
    \res_v(z_{m,n})\in H^1_{f}(K[mp^n]_v,M):= \ker\left( H^1(K[mp^n]_v,M)\longrightarrow \frac{H^1(K[mp^n]_v,V)}{H^1_{\ur}(K[mp^n]_v,V)}\right).
    \]
    
    Hence, for each finite place $\lambda\nmid Np$ of $K$, we have that $\res_{\lambda}(z_m^\bullet)\in H^1_{\ur}(K[mp^\infty]_\lambda,M)$. Therefore
    \[
    z_m^\bullet\in H^{1}_{\emptyset,\emptyset}(K[mp^\infty],M).
    \]
    Combined with Corollary~\ref{cor:F+} and  Proposition~\ref{prop:vanishing-col}, this implies that
    \[
    z_m^\bullet\in H^{1}_{\bullet,\bullet}(K[mp^\infty],M),
    \]
    as desired.
\end{proof}

\begin{remark}
    \label{rem_2025_10_13_0908}
 The only reason why our main results require the hypothesis \textup{\eqref{item_LOC}} is the verification of the local properties of signed diagonal cycles (as in Proposition~\ref{prop:local-ap=0}). The more general case can be treated once we compare our diagonal cycles with those of \cite{DarmonRotger, BSV}, relying on the reciprocity laws therein. To establish this comparison, however, one needs to allow variation in $f$ and $g$ as well. This will be treated in the PhD thesis (in progress) of Qingshen Lv. 
\end{remark}

\subsection{Euler systems}
\label{subsec:ES}

In this subsection, we formulate the signed (anticyclotomic) Iwasawa main conjectures for the trito-non-ordinary product ${\rm BC}_{K/\QQ}(\pi_f)\times {\rm BC}_{K/\QQ}(\pi_g)\times \psi$, and explain that the work \cite{JNS} implies one inclusion of these conjectures.

\subsubsection{} We write $V=M\otimes_{\mathcal{O}}L$ and $A=V/M$. Let us assume\footnote{In our main results towards anticyclotomic main conjectures (cf. Theorem~\ref{thm:mainconjectures} below), we will require the stronger ``big image'' hypothesis \eqref{item_BI0} on $M$.} until the end of \S\ref{sec:applications} that $f$ is residually non-Eisenstein. This, together with the fact that $g$ is readily so, implies that we have a symmetric $\mathcal{O}$-linear perfect pairing
\[
\langle \, , \, \rangle \, : \, M \times M \longrightarrow \mathcal{O}(1)
\]
such that $\langle x^\sigma,y^{c\sigma c}\rangle=\langle x,y\rangle^\sigma$ for all $x,y\in M$ and $\sigma\in G_K$, where $c\in G_{\bb{Q}}$ denotes the choice of complex conjugation determined by our fixed embedding $\iota:\overline{\bb{Q}}\hookrightarrow \bb{C}$. In particular, it follows that $M^c\simeq \Hom_{\mathcal{O}}(M,\mathcal{O}(1))$ and $A\simeq \Hom_{\mathcal{O}}(M^c,L/\mathcal{O}(1))$.

\subsubsection{} Let $\mathcal{K}$ be a finite extension of $K$, $v$ a finite place of $\mathcal{K}$, and $\overline{v}=c(v)$. Then, we have an isomorphism $H^1(\mathcal{K}_{\overline{v}},M)\cong H^1(\mathcal{K}_v,M^c)$. Therefore, cup-product gives a natural perfect pairing
\begin{equation*}
    H^1(\mathcal{K}_{\overline{v}},M)\times H^1(\mathcal{K}_v,A) \longrightarrow L/\mathcal{O}.
\end{equation*}
Similarly, we have a perfect pairing
\begin{equation*}
    H^1(\mathcal{K}_{\overline{v}},M)\times H^1(\mathcal{K}_v,M) \longrightarrow\mathcal{O}.
\end{equation*}

\begin{defn}
    Let $m$ be a square-free product of primes coprime to $Np$ and split in $K$. Let $\bullet\in\lbrace \sharp,\flat\rbrace$.
        \item[i)] The submodule $H^1_{\bullet}(K[mp^\infty]_{\frk{p}},\mathscr{F}^-M)\subseteq H^1_{\Iw}(K[mp^\infty]_{\frk{p}},\mathscr{F}^-M)$ is defined as the kernel of the Coleman map
        \[
        \mathrm{Col}_{\mathscr{F}^-M,\frk{p},m}^\bullet:H^1_{\Iw}(K[mp^\infty]_{\frk{p}},\mathscr{F}^-M)\longrightarrow \Lambda[\Gal(K[m]/K)].
        \]
        \item[ii)]  For each positive integer $n$, the submodule $H^1_{\bullet}(K[mp^n]_{\frk{p}},\mathscr{F}^-M)\subseteq H^1(K[mp^n]_{\frk{p}},\mathscr{F}^-M)$ is defined as the image of $H^1_{\bullet}(K[mp^\infty]_{\frk{p}},\mathscr{F}^-M)$ under the natural projection
        \[
        H^1_{\Iw}(K[mp^\infty]_{\frk{p}},\mathscr{F}^-M)\longrightarrow H^1(K[mp^n]_{\frk{p}},\mathscr{F}^-M).
        \]
        \item[iii)]  For each positive integer $n$, the local condition $H^1_{\ord,\bullet}(K[mp^n]_{\frk{p}},M)\subseteq H^1(K[mp^n]_{\frk{p}},M)$ is defined as the kernel of the homomorphism
        \[
        H^1(K[mp^n]_{\frk{p}},M)\longrightarrow \frac{H^1(K[mp^n]_{\frk{p}},\mathscr{F}^-M)}{H^1_{\bullet}(K[mp^n]_{\frk{p}},\mathscr{F}^-M)}.
        \]
        \item[iv)]  The submodule $H^1_{\bullet}(K[mp^\infty]_{\overline{\frk{p}}},\mathscr{F}^+M)\subseteq H^1_{\Iw}(K[mp^\infty]_{\overline{\frk{p}}},\mathscr{F}^+M)$ is defined as the kernel of the Coleman map
        \[
        \mathrm{Col}_{\mathscr{F}^+M,\overline{\frk{p}},m}^\bullet:H^1_{\Iw}(K[mp^\infty]_{\overline{\frk{p}}},\mathscr{F}^+M)\longrightarrow \Lambda[\Gal(K[m]/K)].
        \]
        \item[v)]  For each positive integer $n$, the submodule $H^1_{\bullet}(K[mp^n]_{\overline{\frk{p}}},\mathscr{F}^+M)\subseteq H^1(K[mp^n]_{\overline{\frk{p}}},\mathscr{F}^+M)$ is defined as the image of $H^1_{\bullet}(K[mp^\infty]_{\overline{\frk{p}}},\mathscr{F}^+M)$ under the natural projection
        \[
        H^1_{\Iw}(K[mp^\infty]_{\overline{\frk{p}}},\mathscr{F}^+M)\longrightarrow H^1(K[mp^n]_{\overline{\frk{p}}},\mathscr{F}^+M).
        \]
        \item[vi)] For each positive integer $n$, the local condition $H^1_{\ord,\bullet}(K[mp^n]_{\overline{\frk{p}}},M)\subseteq H^1(K[mp^n]_{\overline{\frk{p}}},M)$ is defined by
        \[
        H^1_{\ord,\bullet}(K[mp^n]_{\overline{\frk{p}}},M):=H^1_{\bullet}(K[mp^n]_{\overline{\frk{p}}},\mathscr{F}^+M)\subseteq H^1(K[mp^n]_{\overline{\frk{p}}},M).
        \]
\end{defn}

For $\frk{q}\in\lbrace \frk{p},\overline{\frk{p}}\rbrace$ and $\bullet\in\lbrace \sharp,\flat\rbrace$, 
we define $H^1_{\ord,\bullet}(K[p^n]_{\frk{q}},A)\subseteq H^1(K[p^n]_{\frk{q}},A)$ as the orthogonal complement of $H^1_{\ord,\bullet}(K[p^n]_{\overline{\frk{q}}},M)$ under the pairing
\begin{equation*}
    H^1(K[p^n]_{\overline{\frk{q}}},M)\times H^1(K[p^n]_{\frk{q}},A) \longrightarrow L/\mathcal{O}.
\end{equation*}

\subsubsection{} Again, let $\mathcal{K}$ be a finite extension of $K$ and let $v$ be a finite place of $\mathcal{K}$ not dividing $p$. Then we define
\[
H^1_f(\mathcal{K}_v,V)=H^1_{\ur}(\mathcal{K}_v,V)=\ker\left(H^1(\mathcal{K}_v,V)\longrightarrow H^1(\mathcal{K}_v^{\ur},V)\right).
\]
We define $H^1_f(\mathcal{K}_v,M)$ and $H^1_f(\mathcal{K}_v,A)$ by propagation:
\begin{itemize}
    \item $H^1_f(\mathcal{K}_v,M)$ is the preimage of $H^1_f(\mathcal{K}_v,V)$ under the map $H^1(\mathcal{K}_v,M)\rightarrow H^1(\mathcal{K}_v,V)$;
    \item $H^1_f(\mathcal{K}_v,A)$ is the image of $H^1_f(\mathcal{K}_v,V)$ under the map $H^1(\mathcal{K}_v,V)\rightarrow H^1(\mathcal{K}_v,A)$.
\end{itemize}

\subsubsection{} For each positive integer $n$ and for $\bullet\in\lbrace \sharp,\flat\rbrace$, we define
\begin{align*}
    H^1_{\bullet,\bullet}(K[p^n],M)&=\ker\left(H^1(K[p^n],M)\longrightarrow \prod_{\lambda\nmid p} \frac{H^1(K[p^n]_{\lambda},M)}{H^1_f(K[p^n]_{\lambda},M)}\times \prod_{\frk{q}\mid p} \frac{H^1(K[p^n]_{\frk{q}},M)}{H^1_{\ord,\bullet}(K[p^n]_{\frk{q}},M)}\right), \\
    \Sel_{\bullet,\bullet}(K[p^n],A)&=\ker\left(H^1(K[p^n],A)\longrightarrow \prod_{\lambda\nmid p} \frac{H^1(K[p^n]_{\lambda},A)}{H^1_f(K[p^n]_{\lambda},A)}\times \prod_{\frk{q}\mid p} \frac{H^1(K[p^n]_{\frk{q}},A)}{H^1_{\ord,\bullet}(K[p^n]_{\frk{q}},A)}\right).
\end{align*}
We further define
\begin{align*}
    H^1_{\bullet,\bullet}(K[p^\infty],M)&=\varprojlim_n H^1_{\bullet,\bullet}(K[p^n],M), \\
    \Sel_{\bullet,\bullet}(K[p^\infty],A)&=\varinjlim_n \Sel_{\bullet,\bullet}(K[p^n],A) \\
    X_{\bullet,\bullet}(K[p^\infty],A)&=\Hom_{\mathcal{O}}(\Sel_{\bullet,\bullet}(K[p^\infty],A),L/\mathcal{O})
\end{align*}
Note that this definition of $H^1_{\bullet,\bullet}(K[p^\infty],M)$ agrees with the definition given in Definition~\ref{def:Selmergroups}.

\begin{lemma}
    We have the following equalities:
    \begin{align*}
        & H^1_\sharp(K_{\overline{\frk{p}}},\mathscr{F}^+M)=H^1_\flat(K_{\overline{\frk{p}}},\mathscr{F}^+M)=H^1_f(K_{\overline{\frk{p}}},\mathscr{F}^+M), \\
        & H^1_\sharp(K_{\frk{p}},\mathscr{F}^-M)=H^1_\flat(K_{\frk{p}},\mathscr{F}^-M)=H^1_f(K_{\frk{p}},\mathscr{F}^-M).
    \end{align*}
\end{lemma}
\begin{proof}
Let $z_\infty\in H^1(K_\fq,T\otimes\Lambda\bk)$, where $(\fq,T)$ is either $(\overline{\fp},\sF^+M)$ or $(\fp,\sF^-M)$. There is an isomorphism
\[
H^1(K_\fq,T\otimes\Lambda\bk)_\Gamma\cong H^1(K_\fq,T).
\]
Let $z$ be the image of $z_\infty$ in $H^1(K_\fq,T)$ under this isomorphism.
Let $\bullet\in\{\sharp,\flat\}$. Then $z\in H^1_\bullet(K_\fq,T)$ if and only if $\col^\bullet_{T,\fq}(z_\infty)(\mathbb{1})=0$, where $\mathbb{1}$ denotes the trivial character of $\Gamma$.

By the interpolation formula of the Perrin-Riou map and \eqref{eq:Col-decompo+} (or \eqref{eq:Col-decompo-}), we have
\[
\begin{bmatrix}
    \langle \mathcal{L}_{T,\fq}(z_\infty)(\mathbb{1}),v_{\alpha}\rangle \\
    \langle \mathcal{L}_{T,\fq}(z_\infty)(\mathbb{1}),v_{\beta}\rangle
\end{bmatrix}
=\begin{bmatrix}
 \langle(1-\vp)(1-p^{-1}\vp^{-1})^{-1}\exp^*(z),v_{\alpha}\rangle \\
  \langle(1-\vp)(1-p^{-1}\vp^{-1})^{-1}\exp^*(z),v_{\beta}\rangle
\end{bmatrix}
= (\alpha-\beta)Q_g^{-1}A_g\begin{bmatrix}
    \col_{T,\fq}^{\sharp}(z_\infty)(\mathbb{1}) \\
    \col_{T,\fq}^{\flat}(z_\infty)(\mathbb{1})
\end{bmatrix},
\]
where we have written $v_\xi=v_{g,\xi}\otimes w_{\fq,?}$ for $\xi\in\{\alpha,\beta\}$, $?\in\{+,-\}$ for simplicity.
Since the adjoint of $\vp$ is $p^{-1}\vp^{-1}$, we deduce that
\[
\chi_g(p)p\cdot\begin{bmatrix}
 \langle\exp^*(z),(1-\vp)^{-1}(1-p^{-1}\vp^{-1})(\vp^2(v_0)\otimes w_\fq)\rangle \\
  \langle\exp^*(z),-(1-\vp)^{-1}(1-p^{-1}\vp^{-1})(\vp(v_0)\otimes w_\fq)\rangle
\end{bmatrix}
= \begin{bmatrix}
    \col_{T,\fq}^{\sharp}(z_\infty)(\mathbb{1}) \\
    \col_{T,\fq}^{\flat}(z_\infty)(\mathbb{1})
\end{bmatrix}.
\]
On $\Dcris(T)$, we have $\vp^2+a\vp +b =0$ for some constants $a$ and $b$. Then, \cite[proof of Lemma~5.6]{LLZ0.5} tells us that
\[
(1-\vp)^{-1}(1-p^{-1}\vp^{-1})=\frac{(1+a+bp)\vp+a(1+a+pb)+b(p-1)}{pb(1+a+b)}.
\]
As $\langle u,v_0\otimes w_\fq\rangle=0$ if $u$ is in the image of $\exp^*$, we have
\[
\begin{bmatrix}
C_\sharp\langle\exp^*(z),\vp(v_0)\otimes w_\fq\rangle \\
 C_\flat\langle\exp^*(z),\vp(v_0)\otimes w_\fq\rangle
\end{bmatrix}
= \begin{bmatrix}
    \col_{T,\fq}^{\sharp}(z_\infty)(\mathbb{1}) \\
    \col_{T,\fq}^{\flat}(z_\infty)(\mathbb{1})
\end{bmatrix}
\]
for some non-zero constants $C_\sharp$ and $C_\flat$. Hence, we deduce that $\col_{T,\fq}^{\bullet}(z_\infty)(\mathbb{1})=0$ if and only if $\exp^*(z)=0$, which is equivalent to $z\in H^1_f(K_\fq,T)$.
\end{proof}

\begin{proposition}
    Let $\bullet\in \lbrace \sharp,\flat\rbrace$. The local conditions $H^1_{\ord,\bullet}(K_{\overline{\frk{p}}},M)$ and $H^1_{\ord,\bullet}(K_\frk{p},M)$ are exact orthogonal complements under the pairing
    \begin{equation*}
    H^1(K_{\overline{\frk{p}}},M)\times H^1(K_\frk{p},M) \longrightarrow\mathcal{O}.
    \end{equation*}
\end{proposition}
\begin{proof}
    It suffices to show that $H^1_{\bullet}(K_{\overline{\frk{p}}},\mathscr{F}^+M)$ and $H^1_{\bullet}(K_{\frk{p}},\mathscr{F}^-M)$ are exact orthogonal complements under the pairing
    \begin{equation*}
    H^1(K_{\overline{\frk{p}}},\mathscr{F}^+M)\times H^1(K_\frk{p},\mathscr{F}^-M) \longrightarrow\mathcal{O}.
    \end{equation*}
    But this follows from the previous lemma, since $H^1_f(K_{\overline{\frk{p}}},\mathscr{F}^+M)$ is the exact orthogonal complement of $H^1_f(K_{\frk{p}},\mathscr{F}^-M)$ under the above pairing by \cite[Proposition~3.8]{blochkato}.
\end{proof}

We can now formulate signed main conjectures for the Selmer groups introduced above.

\begin{conjecture}
    Let $\bullet\in \lbrace \sharp,\flat\rbrace$. The class $z_1^\bullet\in H^1(K[p^\infty],M)$ lies in $H^1_{\bullet,\bullet}(K[p^\infty],M)$. Moreover, $z_1^\bullet$ is not $\Lambda$-torsion, both $H^1_{\bullet,\bullet}(K[p^\infty],M)$ and $X_{\bullet,\bullet}(K[p^\infty],A)$ are rank-one $\Lambda$-modules and
    \[
    \Char_{\Lambda}\left(X_{\bullet,\bullet}(K[p^\infty],A)_{\mathrm{tors}}\right)=\Char_\Lambda\left(\frac{H^1_{\bullet,\bullet}(K[p^\infty],M)}{\Lambda \cdot z_1^\bullet}\right)^2.
    \]
\end{conjecture}

\subsubsection{} 
\label{subsubsec_2025_10_30_0947}
Let $K(p^\infty)$ denote the ray class field of $K$ of conductor $p^\infty$. In order to apply the Euler system machinery of \cite{JNS}, we assume that the following set of ``big image'' hypotheses hold.
 \begin{itemize}
        \item[(\mylabel{item_BI}{$\mathbf{BI}$})] The following three conditions hold.
         \begin{itemize}
\item[(\mylabel{item_BI0}{$\mathbf{BI}_0$})] The residual $G_K$-representation $\overline{M}$ is absolutely irreducible.
 \item[(\mylabel{item_BI1}{$\mathbf{BI}_1$}$)$]There exists $\sigma_0 \in \Gal(\overline{K}/K(p^\infty))$ such that $M/(\sigma_0-1)M \simeq \mathcal{O}$ is a free $\mathcal{O}$-module of rank one.
 \item[(\mylabel{item_BI2}{$\mathbf{BI}_2$}$)$] There exists $\tau_0 \in G_K$ such that $\tau_0-1$ acts on $M$ as multiplication by a unit $a_{\tau_0}\in\mathcal{O}^\times$ with $a_{\tau_0}-1\in\mathcal O^{\times}$.
  \end{itemize}
    \end{itemize}

Let $\mathcal{L}$ denote the set of rational primes not dividing $pN$ which split in $K$ and let $\mathcal{S}$ denote the set of square-free products of primes in $\mathcal{L}$. For each $q\in \mathcal{L}$, we fix a choice of prime $\frk{q}$ of $K$ above $q$ and we define
\[
P_\frk{q}(X)=\det(1-\Fr_\frk{q}X\vert V^c).
\]

\begin{theorem}\label{thm:tamenormrelations}
    Let $q\in \mathcal{L}$. Suppose that $m, mq\in \mathcal{S}$. Then, for $\bullet\in \lbrace \sharp,\flat\rbrace$,
    \[
    \cor_{K[mq]/K[m]} (z_m^\bullet)=\chi_f(q)(\varphi_0(\overline{\frk{q}}) \Fr_{\overline{\frk{q}}})^2 P_{\frk{q}}(\Fr_{\frk{q}}) z_{m}^\bullet \quad \mathrm{mod}\, (q-1)H^1(K[mp^\infty],M). 
    \]
\end{theorem}
\begin{proof}
    This follows from Theorem~\ref{thm:decompo}. 
\end{proof}

\begin{theorem}\label{thm:mainconjectures}
Let $\bullet\in\lbrace \sharp,\flat\rbrace$. Assume that \eqref{item_LOC}, \eqref{item_H0minus} and \eqref{item_BI} hold. Further, assume that $z_1^\bullet\in H^1_{\bullet,\bullet}(K[p^\infty],M)$ is not $\Lambda$-torsion. Then both $H^1_{\bullet,\bullet}(K[p^\infty],M)$ and $X_{\bullet,\bullet}(K[p^\infty],A)$ are rank-one $\Lambda$-modules and
\[
\Char_{\Lambda}\left(X_{\bullet,\bullet}(K[p^\infty],A)_{\mathrm{tors}}\right)\supseteq\Char_{\Lambda} \left(\frac{H^1_{\bullet,\bullet}(K[p^\infty],M)}{\Lambda \cdot z_1^\bullet} \right)^2.
\]
\end{theorem}
\begin{proof}
    By Theorem~\ref{thm:tamenormrelations}, the collection
    \[
    \lbrace z_m^{\bullet}\in H^1_{\bullet,\bullet}(K[mp^\infty],M)\,\colon\, m\in \mathcal{S} \rbrace
    \]
    is a split anticyclotomic Euler system for $M$ in the sense of \cite{JNS}.
    This implies the desired inclusion by the Euler system machinery.
\end{proof}

\begin{remark}
    Although the work \cite{JNS} is not publicly available at the time of writing, the reader is invited to consult \cite[\S8]{ACR} for an extensive review of the results therein.
\end{remark}


\appendix

\section{A comparison of towers of Shimura curves}
\label{sec:comparison_of_towers}
The tower of Shimura curves of level $U_n$ naturally carries the Galois representations attached to self-dual twists of (ordinary families) modular forms. In this appendix, we compare its cohomology with the more classical (but for our purposes, rather ad hoc) definition using the tower of classical modular curves (that are associated to level subgroups $U_{1,n}$).

    \subsection{A comparison at finite layers}\label{subsection_2025_07_03_1620} Our comparison will go through the modular curve $X^{1,1}_n$ (that covers both $X_{1,n}$ and $X(U_n)$) of level subgroup 
$$U^{1,1}_{n}:=\{ g \in U_0(N) \,:\, \Psi_p(g_p) \equiv \left( \begin{smallmatrix} 1& \star  \\  & 1 \end{smallmatrix}\right)\, (\text{mod } p^n)\}\,.$$
We recall also the level subgroup
$$U_{0,n}:=\{ g \in U_0(N) \,:\, \Psi_p(g_p) \equiv \left( \begin{smallmatrix} \star& \star  \\  & \star \end{smallmatrix}\right)\, (\text{mod } p^n)\}$$
and the corresponding Shimura curve $X_{0,n}$, as well the subgroup
$$U_{n}^{1}:=\{ g \in U_0(N) \,:\, \Psi_p(g_p) \equiv \left( \begin{smallmatrix} 1& \star  \\  & \star \end{smallmatrix}\right)\, (\text{mod } p^n)\}\,.$$

We remark that both $X^{1,1}_n\xrightarrow{\,q_1\,} X_{1,n}$ and $X^{1,1}_n\xrightarrow{\,q_2\,} X(U_n)$ are Galois coverings with cyclic Galois groups (recall that $p>2$) isomorphic to 
\begin{align}
    \label{eqn_2025_05_15_0823}
    \begin{aligned}
    T_n^{(2)}:=U_{1,n}/U^{1,1}_n\xrightarrow[\sim]{\,\det\,} (\ZZ/p^n\ZZ)^\times\xleftarrow[\sim]{\,\det^{\frac{1}{2}}\,} U_{n}/U^{1,1}_n=Z_n\,.
    \end{aligned}
\end{align}
where the square-root of the determinant can be defined since $\det$ maps $Z_n$ isomorphically onto squares in the multiplicative group $(\Z/p^n\Z)^\times$. However, it depends on the choice of an element of order $2$ in $\FF_p^\times$, equivalently a square-root of $g^2$ where $g$ is any topological generator of $\ZZ_p^\times$. We declare\footnote{This choice will not affect the main constructions of this paper (that concern triple products).} henceforth that $(g^2)^{\frac{1}{2}}=g$ (rather than $-g$). Let us also fix a generator $\pmb{\zeta}:=\{\zeta_{p^n}\}\in \Zp(1)$.

\subsubsection{} 
\label{subsubsec_B111_2025_07_01}
The cover $X^{1,1}_n\xrightarrow{\,q_0\,} X_{0,n}$ is also Galois with Galois group
$$T_n:=U_{0,n}/U^{1,1}_n \xrightarrow{\sim} T_n^{(1)}\times T_n^{(2)} \xrightarrow[\sim]{(\det,\det)} (\ZZ/p^n\ZZ)^\times \times (\ZZ/p^n\ZZ)^\times\,,$$
where $T_n^{(1)}:=U_{n}^1/U_{n}^{1,1}\xrightarrow[\sim]{\det}(\ZZ/p^n\ZZ)^\times$.

\subsubsection{} Let us put $G_n:=\Gal(\QQ(\mu_{p^n})/\QQ)$ and denote by $\langle \,\cdot\,\rangle$ the tautological Galois character 
$$\langle \,\cdot\,\rangle\,:\,G_\QQ \twoheadrightarrow G_n \lra \Zp[G_n]\,.$$
We let $\Zp[G_n]^\iota$ denote the free $\Zp[G_n]$-module of rank one on which $G_\QQ$ acts via $\langle \,\cdot\,\rangle^{-1}$. More explicitly, $\Zp[G_n]^\iota=\Zp[G_n]$ as a set, and $g\in G_\QQ$ acts multiplication by $g^{-1}$ on the group like elements of $\Zp[G_n]$. Let us put
$$\chi_n\,:\, G_n\lra (\Zp/p^n\ZZ)^\times\,,\qquad g \mapsto \chi_\cyc(g) \mod p^n\,.$$
We denote by 
$$\chi_n^{-1}:(\Zp/p^n\ZZ)^\times \xrightarrow{\,\sim\,} G_n\,,\qquad g\mapsto \delta_g $$ its inverse (not the reciprocal of the character $\chi_n$). Via \eqref{eqn_2025_05_15_0823}, \S\ref{subsubsec_B111_2025_07_01} and $\chi_n^{-1}$, we will identify the tautological character $T_n^{(i)} \lra \Zp[T_n^{(i)}]^\times$ with $\langle \,\cdot\,\rangle$ as well (and denote it by $\langle \,\cdot\,\rangle_{(i)}$ when we wish to distinguish the choices $i=1,2$), so that we have the following commutative diagram:
$$\xymatrix{
T_n^{(i)} \ar[rr]^-{\langle \,\cdot\,\rangle_{(i)}}\ar[d]^{\sim}_{\chi_n^{-1}\circ\, \det}&& \Zp[T_n^{(i)}]^\times\ar[d]_{\sim}^{\chi_n^{-1}\circ \,\det}\\
G_n  \ar[rr]_-{\langle \,\cdot\,\rangle}&& \Zp[G_n]^\times
}$$

\subsubsection{} Similarly, we denote by $Z_n \xrightarrow{\,\langle \,\cdot\,\rangle_Z\,} \Zp[Z_n]^\times$ the tautological character, and let  $\Zp[Z_n]^\iota$ denote the free $\Zp[Z_n]$-module on which $Z_n$ acts via $\langle \,\cdot\,\rangle_Z^{-1}$. Via the morphisms \eqref{eqn_2025_05_15_0823}, we may endow $\Zp[T_n^{(1)}]$ and $\Zp[G_n]$ with the structure of a $\Zp[Z_n]$-module:
$$ (\Z/p^n\Z)^{\times} \xrightarrow{\,\sim\,} Z_n \xrightarrow{\det} (\Z/p^n\Z)^{\times,2} \xrightarrow{\,\chi_n^{-1}\,} G_n^2\hookrightarrow G_n \lra \ZZ_p[G_n]^\times\,,\qquad  [g]=\begin{bmatrix}
    g&0\\0&g
\end{bmatrix}\mapsto \delta_{g^2}=\delta_g^2\,,$$
where for a multiplicative abelian group $G$, we write $G^2$ to denote the squares in $G$. This tells us that we have 
$$\langle\,\cdot\,\rangle_{(i)}\simeq \ZZ_p[G_n]\simeq \langle\,\cdot\,\rangle^2_Z$$ as $\Zp[(\Z/p^n\Z)^{\times}]$-modules.
\subsubsection{} We then have the Hecke equivariant natural isomorphisms
\begin{align}
    \begin{aligned}
        \label{eqn_2025_05_27_1258}
        H^1_\etale(X^{1,1}_n\times\overline{\Q},\QQ_p)\simeq H^1_\etale(X_{1,n} \times_{\QQ} \overline{\Q},\Qp) \otimes_{\Zp} \Zp[G_n]^\iota &\simeq H^1_\etale(X_{1,n} \times_{\QQ} \overline{\Q},\Qp) \otimes_{\Zp}  \langle\,\cdot\,\rangle_{(1)}^{-1}\\
        &\simeq H^1_\etale(X_{1,n} \times_{\QQ} \overline{\Q},\Qp) \otimes_{\Zp} \langle\,\cdot\,\rangle_Z^{-2}\,,
    \end{aligned}
\end{align}  
of Galois modules, where the torus $T_n$ acts diagonally on the tensor products on the first row (through $T_n\twoheadrightarrow T_n^{(2)}$ on the first factor, and $T_n\twoheadrightarrow T_n^{(1)}$ on the second). In \eqref{eqn_2025_05_27_1258}, the first is induced from the isomorphism 
$$X_{n}^{1,1}\simeq X_{1,n}\times_{\QQ}\QQ(\mu_{p^n}) \simeq X_{1,n}\times_{\QQ}\QQ[G_n]$$ 
 which is determined by $\pmb{\zeta}$ and Shapiro's lemma (the second can be described similarly). 
 
\subsubsection{}
Let $\m$ be a maximal ideal of the Hecke algebra $\TT_n^B$ acting on $H^1_\etale(X(U_{n}) \times_{\QQ} \overline{\Q},\Qp)$, and let us denote its pullback to the Hecke algebra acting on $H^1_\etale(X_{n}^{1,1} \times_{\QQ} \overline{\Q},\Qp)$ also by $\m$. Let us assume that $\m$ is non-Eisenstein in the sense that $\mathrm{T}_q'-(q+1)\not\in \m$ for some prime $q\nmid Np^n$. The following lemma is borrowed from \cite{Fouquet2013}.  
\begin{lemma}
     \label{lemma_2025_06_16_1216}
    For any non-Eisenstein maximal ideal $\m$ as above, we have the following Hecke equivariant isomorphism of Galois modules:
    \begin{align}
    \begin{aligned}
        \label{eqn_2025_06_01_0817}
        H^1_\etale(X(U_n)\times_\QQ\overline{\QQ},\QQ_p)_{\m} 
        \simeq H^1_\etale(X_{n}^{1,1} \times_{\QQ} \overline{\Q},\Qp)_{\m} \otimes_{\ZZ_p[T_n]} \ZZ_p[T_n/Z_n]\,.
         \end{aligned}
\end{align} 
\end{lemma}

\begin{proof}
         This is proved as part of the proof of \cite[Corollary 3.2]{Fouquet2013}, where the key point is that, for any intermediate curve $X_n^{1,1}\to Y\to X(U_n)$, the Hecke module $H^2_\etale(Y\times_{\QQ} \overline{\Q},\Qp)_\m$ is annihilated by the degree-0 Hecke operator $\mathrm{T}_q'-(q+1)$, which also acts invertibly on the said module since $\mathrm{T}_q'-(q+1)\not\in\m$. 
\end{proof}

\begin{remark}
\item[i)] In \cite{Fouquet2013}, the author does not require the non-Eisenstein condition, but instead applies the ordinary projector. Our lemma slightly extends Fouquet's corollary as we do not pass to the $p$-ordinary direct summand of the cohomology, at the expense of excluding the Eisenstein locus.
\item[ii)] Let us assume for all the Shimura curves in this paragraph that the level away from $p$ is $U_0(N)$ $($rather than $U_1(N)$$)$. We have a decomposition\footnote{In general, the constituents in the analogous decomposition of $H^1_\etale(X_{n}^{1,1} \times_{\QQ} \overline{\Q},\Qp) \otimes_{\ZZ_p[T_n]} \ZZ_p[T_n/Z_n]$ will involve twists of Deligne's representations that are self-dual up to a Dirichlet character of conductor dividing $N$.} 
$$H^1_\etale(X_{n}^{1,1} \times_{\QQ} \overline{\Q},\Qp) \otimes_{\ZZ_p[T_n]} \ZZ_p[T_n/Z_n]=:H^1_\etale(X_{n}^{1,1} \times_{\QQ} \overline{\Q},\Qp)^\dagger\simeq \bigoplus_g V_g^\dagger$$ 
into Hecke eigenspaces, where the direct sum is over all weight-two eigenforms $g$ on $U_{1,n}$ and $V_g^\dagger$, as in \cite{howard2007}, is the self-dual twist of Deligne's representation.  Therefore, the Galois module $H^1_\etale(X_{n}^{1,1} \times_{\QQ} \overline{\Q},\Qp)^\dagger$ coincides with the Hecke module $H^1_\etale(X_{1,n} \times_{\QQ} \overline{\Q},\Qp)\left(-\doublefrac{1}{2}\right)$ in the notation of \cite{DarmonRotger}. 
The promised relationship between the cohomology of towers $X(U_n)$ and $X_{1,n}$ can be extracted from \eqref{eqn_2025_05_27_1258} and \eqref{eqn_2025_06_01_0817}.
\end{remark}

\subsection{Hida theory: A variant of Ohta's control theorem and a result of Fouquet}

Recall that $p$ is a prime which does not divide $N^-$ and that we have identified $\Psi_p: \mathcal{O}_B \otimes \ZZ_p \simeq M_2(\ZZ_p)$. 

\subsubsection{} Let $K/\QQ_p$ be an unramified finite extension, with ring of integers $\mathcal{O}$. Let $T_{\GL_2}$, resp.  $Z_{\GL_2}$, denote the maximal diagonal torus, resp. center, of $\GL_2$. Then we identify $T_{\GL_2}(\ZZ_p) / Z_{\GL_2}(\ZZ_p)$ with $\ZZ_p^\times$ via the map
$$ T_{\GL_2}(\ZZ_p) / Z_{\GL_2}(\ZZ_p) \lra \ZZ_p^\times,\, \qquad \left( \begin{smallmatrix}a & \\ & d \end{smallmatrix}\right) \mapsto a d^{-1},$$ and let $\Lambda_Z : = \mathcal{O}\llbracket T_{\GL_2}(\ZZ_p) / Z_{\GL_2}(\ZZ_p)\rrbracket \simeq \mathcal{O} \llbracket\ZZ_p^\times \rrbracket$ be the corresponding Iwasawa algebra; its $n$-th layer will be denoted by $\Lambda_{Z,n}$. Note that $$\Lambda_{Z,n} \simeq \mathcal{O}[T_n/Z_n],$$
where $T_n$ and $Z_n$ have been introduced in \S \ref{subsection_2025_07_03_1620}. Note in addition that, if $\Lambda_{\rm full} : = \mathcal{O}\llbracket T_{\GL_2}(\ZZ_p)\rrbracket$, the natural projection $T_{\GL_2}(\ZZ_p) \twoheadrightarrow T_{\GL_2}(\ZZ_p) / Z_{\GL_2}(\ZZ_p)$ induces a morphism $\Lambda_{\rm full} \to \Lambda_Z$.

\subsubsection{} We define 
$$ H^1_{\rm Iw}( X(U_\infty)\times_{\QQ} \overline{\Q}, \mathcal{O}) : = \varprojlim_n H^1_{\etale}( X(U_n)\times_{\QQ} \overline{\Q}, \mathcal{O}),$$ 
where the inverse limit is taken with respect to the pushforward of the natural degeneracy maps $X(U_{n+1})  \xrightarrow{p_n} X(U_n)$.  We note that $H^1_{\rm Iw}( X(U_\infty)\times_{\QQ} \overline{\Q}, \mathcal{O})$ is naturally a $\Lambda_Z$-module. Moreover, as the Hecke operator ${\rm U}_p'$ commutes with the pushforward maps $p_{n,*}$, it acts on the inverse limit. Hence, the anti-ordinary projector $e_{{\rm U}_{p}'}=\lim ({\rm U}_p')^{n!}$ acts on $H^1_{\rm Iw}( X(U_\infty)\times_{\QQ} \overline{\Q}, \mathcal{O})$. 

\subsubsection{} As a special case of the work \cite{LRZ}, we shall prove (as Proposition~\ref{prop_2025_07_04_1618}(iii) below) a control theorem for $e_{{\rm U}_{p}'} H^1_{\rm Iw}( X(U_\infty)\times_{\QQ}  \overline{\Q}, \mathcal{O})$. Moreover, we shall record (cf. Proposition~\ref{prop_2025_07_04_1618}(i) below) the fact (borrowed from \cite{Fouquet2013}, Corollary 3.2) that the module $e_{{\rm U}_{p}'} H^1_{\rm Iw}( X(U_\infty)\times_{\QQ} \overline{\Q}, \mathcal{O})$ is the essentially self-dual twist (in the sense of \cite[Definition 2.1.3]{howard2007}; where ``essentially self-dual'' means self-dual up to Dirichlet characters of conductor coprime to $p$) of the big Galois representations associated to Hida families.  

\subsubsection{} Before we state these, we introduce some notation. For a weight $\lambda = (\lambda_1 \geq \lambda_2)$  of $\GL_2$, we let $V_\lambda$ denote the irreducible representation of $\GL_2$ (over $K$) of highest weight $\lambda$. If $V_{\rm std} \simeq K^2$ denotes the standard representation of $\GL_2$ over $K$, then $$V_\lambda  = {\rm Sym}^{\lambda_1 - \lambda_2} (V) \otimes {\rm det}^{\lambda_2}.$$
Let $V_{\lambda, \mathcal{O}}$ be the minimal admissible lattice of $V_\lambda$, which is explicitly given by 
$$ V_{\lambda, \mathcal{O}} =  {\rm TSym}^{\lambda_1 - \lambda_2} (\mathcal{O}^2) \otimes  {\rm det}^{\lambda_2},$$
and denote by $\mathscr{V}_{\lambda, \mathcal{O}}$ the corresponding local system on $X(U_n)$.

\begin{remark}
Note that if $\lambda = (\lambda_1,\lambda_2)$ is trivial on $Z_{\GL_2}$, then $\lambda_2 = - \lambda_1$ and $ V_{\lambda, \mathcal{O}} =  {\rm TSym}^{2 \lambda_1} (\mathcal{O}^2) \otimes  {\rm det}^{-\lambda_1}.$
\end{remark}

Let us denote by
$$ \mathrm{T}^{\rm a.o.}_{B,\mathcal{O}}(Np^\infty):= \varprojlim_n e_{{\rm U}_{p}'} \TT_n^B \otimes_{\ZZ_p} \mathcal{O}$$
the anti-ordinary Hecke algebra.

\begin{proposition} 
\label{prop_2025_07_04_1618} Let $\mathfrak{m}$ be a maximal ideal of $\mathrm{T}^{\rm a.o.}_{B,\mathcal{O}}(Np^\infty)$  such that the associated mod-$p$ Galois representation $\overline{\rho}_\mathfrak{m}$ is 
irreducible and has residual multiplicity $1$ (in the sense of \cite{Fouquet2013}, Proposition 3.7(iii)).
\item[i)] We have an isomorphism of $\Lambda_Z$-modules $$e_{{\rm U}_{p}'} H^1_{\rm Iw}( X(U_\infty)\times_{\QQ} \overline{\Q}, \mathcal{O})_\mathfrak{m} \simeq  e_{{\rm U}_{p}'} H^1_{\rm Iw}( X_\infty^{1,1}\times_{\QQ} \overline{\Q}, \mathcal{O})_{\widetilde{\mathfrak{m}}} \otimes_{\Lambda_{\rm full}} \Lambda_Z\,,$$
where
$$H^1_{\rm Iw}( X_\infty^{1,1}\times_{\QQ} \overline{\Q}, \mathcal{O}) : = \varprojlim_n H^1_{\etale}( X_n^{1,1}\times_{\QQ} \overline{\Q}, \mathcal{O})$$
 and $\widetilde{\mathfrak{m}}$ denotes the pullback of $\mathfrak{m}$ to the Hecke algebra $\mathrm{T}^{\rm a.o.}_{B,\mathcal{O}}(U_\infty^{1,1})$ acting on $e_{{\rm U}_{p}'} H^1_{\rm Iw}( X_\infty^{1,1}\times_{\QQ} \overline{\Q}, \mathcal{O})$.
 \item[ii)] The $\Lambda_{Z}$-module $e_{{\rm U}_{p}'} H^1_{\rm Iw}( X(U_\infty)_{\overline{\QQ}}, \mathcal{O})_\mathfrak{m}$ is finitely generated and free. Moreover,  $$ e_{{\rm U}_{p}'} H^1_{\rm Iw}( X(U_\infty)_{\overline{\QQ}}, \mathcal{O})_\mathfrak{m} \otimes_{\mathrm{T}^{\rm a.o.}_{B,\mathcal{O}}(Np^\infty)_\mathfrak{m}} R(\mathfrak{a})$$ is of rank 2 over $R(\mathfrak{a}):= \mathrm{T}^{\rm a.o.}_{B,\mathcal{O}}(Np^\infty)_\mathfrak{m} / \mathfrak{a}$, where $ \mathfrak{a}$ is the unique minimal prime ideal contained in $\m$. 
\item[iii)] Let $\lambda$ be a weight which is trivial on $Z_{\GL_2}$ and let $\mathcal{O}[\lambda^{-1}]$ be $\mathcal{O}$ regarded with $\Lambda_{Z}$-action given by $\lambda^{-1}$. For every $n \geq 1$, we have an isomorphism of $\Lambda_{Z,n}$-modules 
    $$e_{{\rm U}_{p}'} H^1_{\rm Iw}( X(U_\infty) \times_{\QQ} \overline{\Q}, \mathcal{O})_\mathfrak{m} \otimes_{\Lambda_{Z,n}} \mathcal{O}[\lambda^{-1}] \simeq e_{{\rm U}_{p}'} H^1_{\etale}( X(U_n) \times_{\QQ} \overline{\Q}, \mathscr{V}_{\lambda,\mathcal{O}})_\mathfrak{m},$$
    which is compatible with Galois and Hecke actions.

\end{proposition}
\begin{proof}
The first part is a special case of \cite[Corollary 3.2]{Fouquet2013}, whereas the second follows from (i) combined with \cite[Proposition 3.7]{Fouquet2013}. We now prove (iii).
By \cite[Theorem 2.6.2]{LRZ}, together with the fact that localization is exact and commutes with the Tor-functor, we have a convergent spectral sequence 
$$E^{i,j}_{2,\mathfrak{m}} = {\rm Tor}^{\Lambda_{Z,n}}_{-i} (e_{{\rm U}_{p}'} H^j_{\rm Iw}( X(U_\infty)\times_{\QQ} \overline{\Q}, \mathcal{O})_\mathfrak{m}, \mathcal{O}[\lambda^{-1}]) \Rightarrow  e_{{\rm U}_{p}'} H^{i+j}_{\etale}( X(U_n) \times_{\QQ} \overline{\Q}, \mathscr{V}_{\lambda,\mathcal{O}})_\mathfrak{m}\,,$$
which is supported on the second quadrant $i \leq 0$, $j \geq 0$. Since $E^{i,j}_{2,\mathfrak{m}} = 0$ unless $j=1$, we have the desired isomorphism $E^{0,1}_{2,\mathfrak{m}} \simeq e_{{\rm U}_{p}'} H^1_{\etale}( X(U_n) \times_{\QQ} \overline{\Q}, \mathscr{V}_{\lambda,\mathcal{O}})_\mathfrak{m}$. The proof of our proposition is now complete.
\end{proof}

\section{The definite case: signed balanced \texorpdfstring{$p$-adic\,}\ \texorpdfstring{$L$-functions\,}\ }
\label{sec:Regularized_definite_diagonal_cycles}
The main purpose of this appendix is to present an alternative construction of Hsieh's diagonal theta elements for Hida families using a representation-theoretic approach, similar to the construction of diagonal cycles presented in \S\ref{sec:families_of_indefinite_cycles} (which is essentially due to Loeffler). Instead of working with indefinite quaternion algebras, we work with definite ones. 

Even though this portion of the appendix has no utility in the main body of the article, it is included here to highlight the parallelism between Hsieh's construction of theta elements in terms of the special locus on (zero-dimensional) Shimura sets in the definite scenario and the construction of diagonal cycles in the indefinite setting. This resemblance suggests the possibility of the existence of a bipartite Euler system (in the sense of Bertolini--Darmon and Howard~\cite{BertoliniDarmon2005, howard06}), but we do not explore this point further in the present article. Furthermore, some of the ideas presented here seem pertinent to the upcoming work of Yi-Li Wu on signed triple-product $p$-adic $L$-functions, which, combined with our construction of signed diagonal classes, may yield a signed version of the explicit reciprocity law.

\subsection{A family of definite theta elements}\label{subsec:familyofthetaelements}
Let $B$ be the definite quaternion algebra over $\QQ$ of discriminant $N^-$. Let $N^+$ be a positive integer prime to $N^-$ and let $N =N^+N^-$. We suppose that $p$ is an odd prime such that $p \nmid N$. In what follows, we retain the notation in \S \ref{subsec:ShimuraCurves}. In particular, recall that $E = \QQ \oplus \QQ \oplus \QQ$ denotes the totally split \'etale cubic algebra over $\QQ$, set $B_E  = B \otimes_\QQ E = B \oplus B \oplus B$, and denote by $\iota\,:\,B \hookrightarrow B_E$ the diagonal embedding.  We also recall that 
\begin{align*}
\mathcal{U}_{Z,(n)} &=\left \{ g \in \mathcal{U}_1(N)\,:\, g \equiv \left(\left( \begin{smallmatrix}  a_1 & b \\ 0 & a_0 \end{smallmatrix}\right),\left( \begin{smallmatrix}  a_2 & c \\ 0 & a_0 \end{smallmatrix}\right),\left( \begin{smallmatrix} a_3 & d \\ 0 & a_0 \end{smallmatrix}\right)\right)\, (\text{mod }p^n),\,a_i \in\Zp^\times, \, b,c,d \in\Zp\right\} 
\end{align*}
for any natural number $n$ (as well as our convention that $(n)$ is a shorthand for the triple $(n,n,n)$).
\begin{defn}\label{def:shimura-set}
    For every $n \geq 0$, let   
        \item[i)] $S_0(Np^n)$ be the Shimura set 
        \[S_0(Np^n) := B^\times \backslash \widehat{B}^\times / U_{0,n}, \]
        with  $U_{0,n}$ as in \eqref{eqn_2024_07_31_1649};
        \item[ii)] $\mathbf{S}_{Z,(n)}$ be the ``non-split triple product'' of Shimura sets given by 
        \[ \mathbf{S}_{Z,(n)}= B_E^\times \backslash \widehat{B}_E^\times / \mathcal{U}_{Z,(n)}. \]
\end{defn}

\subsubsection{} For each element $(g_1,g_2,g_3) \in \widehat{B}_E^\times$, we write $[(g_1,g_2,g_3)]$ for the corresponding double coset in $\mathbf{S}_{Z,(n)}$. As in \cite[Definition 4.6]{hsiehpadicbalanced}, we define $\Delta_n  \in \Zp[\mathbf{S}_{Z,(n)}]$ by setting
\begin{align}\label{eqn_Hsieh_element}
    \Delta_n  := \sum_{[g] \in S_0(N p^n)} \sum_{\substack{b \in \ZZ/p^n \ZZ \\ z\in  (\ZZ/p^n \ZZ)^\times}} \left[\left( g \left( \begin{smallmatrix} p^n & b \\ 0 & 1 \end{smallmatrix} \right),  g \left( \begin{smallmatrix}
    p^n & b + z \\ 0 & 1 
\end{smallmatrix} \right), g \tau_{p^n} \left( \begin{smallmatrix}
 1 & 0 \\ 0 & z   
\end{smallmatrix} \right) \right) \right ]\,, 
\end{align} 
where $\tau_{p^n} = \left( \begin{smallmatrix}
 0 & 1 \\ - p^n & 0   
\end{smallmatrix} \right)$. 

\begin{remark}
Note that the Shimura set $ \mathbf{S}_{Z,(n)}$ differs from the one used in \cite[\S 4.6]{hsiehpadicbalanced}, which is given as
$$B_E^\times \backslash \widehat{B}_E^\times / \mathcal{U}_{1,(n)} \iota (\AA_f^\times)\,.$$
Although this definition of $\Delta_n$ allows more freedom in principle, this distinction does not affect applications, as one eventually evaluates the theta elements at triples of modular forms whose central characters multiply to the trivial character.
\end{remark}

\subsubsection{}\label{subsec_2025_10_14_1517} Let $\mathbb{U}_p'$ be the Hecke operator defined as in \S \ref{subsubsec_2027_07_02_1557}. Its action on $\Zp[\mathbf{S}_{Z,(n)}]$ is associated with the pushforward of $\iota(\eta_p) \in \mathrm{GL}_2(\Qp)^3$, where we recall that $\eta_p =\left( \begin{smallmatrix} p &  \\  & 1 \end{smallmatrix} \right) \in \GL_2(\Qp)$ (this operator is denoted as $\mathbf{U}_p$ in \cite[\S 4.4]{hsiehpadicbalanced}).

\begin{proposition}[\cite{hsiehpadicbalanced}, Lemma 4.7]\label{Prop_of_Hsieh}
Let us denote by $\Zp[\mathbf{S}_{Z,(n+1)}] \xrightarrow{{\rm Nm}_{n}^{n+1} } \Zp[\mathbf{S}_{Z,(n)}]$ the natural projection. Then for every $n \geq 1$, we have  
\[ {\rm Nm}_{n}^{n+1}(\Delta_{n+1}) = \mathbb{U}_p' \,\Delta_{n} .\]
\end{proposition}

As before, let us denote by $e_{\mathbb{U}_p'}$ the anti-ordinary projector associated with $\mathbb{U}_p'$, and define the regularized definite theta element on setting
\[ \Delta^\dagger_n := (\mathbb{U}_p')^{-n} (e_{\mathbb{U}_p'} \, \Delta_{n}),\]
as well as the element 
\[ \Delta^\dagger_\infty := \varprojlim_{n \to \infty} \Delta^\dagger_n \in \varprojlim_{n \to \infty} e_{\mathbb{U}_p'}\Zp[\mathbf{S}_{Z,(n)}], \]
where the inverse limit is taken with respect to the maps ${\rm Nm}_{n}$ given as above.

\subsection{A variant of Hsieh's construction \`a la Loeffler}\label{subsec:Loefflerfamilyofthetaelements}
In \S\ref{sec:families_of_indefinite_cycles}, we constructed families of diagonal cycles arising from the embedding of a Shimura curve into its triple product by essentially mimicking the definition of $\Delta^\dagger_\infty$. However, our construction allows more freedom in terms of $p$-adic variation. This was achieved by exploiting the general machinery of Loeffler developed in \cite{LoefflerSpherical}. 

With this in mind, our goal in the present subsection is to give an alternative definition of the element $\Delta_n$ defined as in \eqref{eqn_Hsieh_element}, which in turn will allow a direct comparison with the cycles of Definition~\ref{def_2025_07_03_1206}. The proof of Proposition \ref{Prop_of_Hsieh} will then be an application of the main result of \cite{LoefflerSpherical}.

\subsubsection{} We introduce basic notation that we shall dwell on in this subsection, and prove a simple lemma.

Let $S_n$ be the Shimura set of level $U_{Z,n}$, defined as in \eqref{eqn_2024_07_31_1640}. As $U_{Z,n} \subseteq U_{0,n}$, we have a natural map $S_n \xrightarrow{\pi_{Z,n}} S_0(Np^n)$ of Shimura sets. Let us denote by $\Zp[S_n] \xrightarrow{\pi_{Z,n}} \Zp[S_0(Np^n)]$ also the induced morphism and by $\Zp[S_0(Np^n)] \xrightarrow{\nu_{Z,n}} \Zp[S_n]$ the trace map induced by 
$$ S_0(Np^n) \lra \bb{Z}_p[S_n],\, \qquad [x] \mapsto \sum_{\pi_{Z,n}([y])=[x]} [y]\,.$$ For every $(z,b) \in (\ZZ/p^n \ZZ)^\times \times \ZZ/p^n \ZZ$, let us choose a lift to $ \Zp^\times \times \Zp$, which we still denote by $(z,b)$, and define $\gamma_{z,b} := \left( \begin{smallmatrix}
    z & b \\ & 1
\end{smallmatrix}\right) \in \GL_2(\Zp)$. 
\begin{lemma}\label{lemma_cosetdecomposition}
 A set of left coset representatives for the quotient $U_{0,n} / U_{Z,n}$ is given by \[\{ \gamma_{z,b}:\, (z,b) \in (\ZZ/p^n \ZZ)^\times \times \ZZ/p^n \ZZ\}.\]   
\end{lemma}
\begin{proof}
This follows from a direct computation. Note first that two different elements $\gamma_{z,b}$ and $\gamma_{z',b'}$ define different left cosets. We now check that any given $g \in U_{0,n}$ can be written as $\gamma_{z,b} \cdot h$, for some $(z,b) \in (\ZZ /p^n \ZZ)^\times \times \ZZ /p^n \ZZ$ and $h \in U_{Z,n}$, or equivalently that there exists $(z,b) \in (\ZZ /p^n \ZZ)^\times \times \ZZ /p^n \ZZ$ such that $\gamma_{z,b}^{-1} \cdot g \in U_{Z,n}$. Indeed, if $g \equiv \left( \begin{smallmatrix} \alpha & \beta \\ & \delta  \end{smallmatrix} \right)\pmod {p^n}$, we may put $(z,b) = (\alpha \delta^{-1},\beta \delta^{-1})$.
\end{proof}

\noindent By Lemma \ref{lemma_cosetdecomposition}, the trace map $\Zp[S_0(Np^n)] \xrightarrow{\nu_{Z,n}} \Zp[S_n]$ is explicitly given by $[x] \mapsto \sum_{(z,b)} [x \cdot \gamma_{z,b}]$.

\subsubsection{}\label{subsec_2025_10_14_1536} For any triple $\mathbf{n}=(n_1,n_2,n_3) \in \mathbb{N}^3$, we also let $\mathbf{S}_{\mathcal{U}_\mathbf{n}}$ and $\mathbf{S}_{\mathcal{U}^{\mathbf{n}}}$ be the ``non-split triple product'' Shimura sets associated with the level subgroups $\mathcal{U}_{\mathbf{n}}$ and $\mathcal{U}^{\mathbf{n}}$ introduced in \S \ref{subsec_2025_10_14_1351}. We recall that 
\[ u = \left(  \left( \begin{smallmatrix} 1 &  \\  & 1 \end{smallmatrix} \right), \left( \begin{smallmatrix} 1 &  1 \\  & 1 \end{smallmatrix} \right),\left( \begin{smallmatrix}  &  1 \\ -1 &  \end{smallmatrix} \right)\right) \in \GL_2^3(\Zp).\]
This element, together with the diagonal embedding $\iota$, gives rise to a map of Shimura sets 
$$ \iota^u_{(n),\star} : \Zp[S_n] \xrightarrow{u_\star \circ \iota_\star} \Zp[\mathbf{S}_{\mathcal{U}^{(n)}}],$$ 
where we recall that $(n)=(n,n,n)$. Composing $\iota^u_{(n),\star}$ with the map $\Zp[\mathbf{S}_{\mathcal{U}^{(n)}}] \xrightarrow{\iota(\eta^n_p)_\star} \Zp[\mathbf{S}_{\mathcal{U}_{(n)}}]$ induced by right-multiplication by $\iota(\eta^n_p)$, we obtain a morphism
\[ \iota(\eta^n_p)_\star \circ  \iota^u_{(n),\star}: \Zp[S_n] \longrightarrow \Zp[\mathbf{S}_{\mathcal{U}_{(n)}}]. \] 
Paralleling our Definition~\ref{def_2025_07_03_1206}, we introduce the following element.

\begin{defn}\label{def_2025_10_14_1506}
    \item[i)] For $n \in \mathbb N$, we define $\Theta_{\mathcal{U}_{(n)}} := \iota(\eta^n_p)_\star \circ  \iota^u_{(n),\star}(\mathbb{1}_{S_n}) \in  \Zp[\mathbf{S}_{\mathcal{U}_{(n)}}]$, where $\mathbb{1}_{S_n} = \sum_{g \in S_n} [g] \in \Zp[S_n]$.        
    \item[ii)]    For every ${\mathbf{n} } = (n_1,n_2,n_3) \in \mathbb N^3$, we define $\Theta_{\mathcal{U}_{\mathbf{n}}} : = {\rm Nm}^{(n)}_{\mathbf{n}} \Theta_{\mathcal{U}_{(n)}} $, where $n = \max \{ n_1,n_2,n_3 \}$ and ${\rm Nm}^{(n)}_{\mathbf{n}} : \Zp[\mathbf{S}_{\mathcal{U}_{(n)}}] \to \Zp[\mathbf{S}_{\mathcal{U}_{\mathbf{n}}}]$ denotes the natural projection.
    \item[iii)] For every ${\mathbf{n} } = (n_1,n_2,n_3) \in \mathbb N^3$, let us consider the natural morphism (cf. \S\ref{subsubsec_2025_10_15_1335}) 
$$\mathbf{S}_{\mathcal{U}_{\mathbf{n}}}  \xrightarrow{\,{\rm b}_{\mathbf{n}}\,} \mathbf{S}_{\mathbf{n}}=\mathbf{S}_{n_1,n_2,n_3} := S_{U_{n_1}} \times S_{U_{n_2}} \times S_{U_{n_3}}\,,$$ 
and define 
\begin{equation}
\label{eqn_2025_10_29_0931}
    \Theta_{\mathbf{n}}:={\rm b}_{\mathbf{n},*}\,\Theta_{\mathcal{U}_{\mathbf{n}}}\in \Zp[S_{U_{n_1}}]\otimes\Zp[S_{U_{n_2}}]\otimes\Zp[S_{U_{n_3}}] \,.
\end{equation}
\end{defn}

Thanks to \cite[Prop. 4.5.2]{LoefflerSpherical} (see also Proposition \ref{prop:norm_comp_indefinite}), the following norm-compatibility property holds.

\begin{proposition}
 For every triple ${\mathbf{n} } = (n_1, n_2, n_3)$ of positive integers 
 we have 
 \[ {\rm Nm}^{\mathbf{n}+\mathbf{1}}_{\mathbf{n}}(\Theta_{\mathcal{U}_{\mathbf{n+1}}}) = \mathbb{U}_p' \, \Theta_{\mathcal{U}_{\mathbf{n}}}\,.\] 
\end{proposition}

\subsubsection{} We explain how to compare explicitly the element $\Delta_n$ given as in \eqref{eqn_Hsieh_element} and the element $\Theta_{\mathcal{U}_{(n)}}$ of Definition \ref{def_2025_10_14_1506}.

Mirroring our discussion in \S \ref{subsec_2025_10_14_1536}, the diagonal embedding $\iota$, together with right multiplication by $u$, we define a map 
$$\iota_{Z,(n),\star}^u : \Zp[S_n] \xrightarrow{u_\star \circ \iota_\star} \Zp[\mathbf{S}_{\mathcal{U}^Z_{(n)}}]$$
of Shimura sets, where $\mathbf{S}_{\mathcal{U}^Z_{(n)}} = B_E^\times \backslash \widehat{B}_E^\times / {\mathcal{U}^Z_{(n)}}$ and 
\[\mathcal{U}^Z_{(n)}:=\left \{ g \in \mathcal{U}_1(N)\,:\, g \equiv \left(\left( \begin{smallmatrix}  a_1 & 0 \\ b & a_0 \end{smallmatrix}\right),\left( \begin{smallmatrix}  a_2 & 0 \\ c & a_0 \end{smallmatrix}\right),\left( \begin{smallmatrix} a_3 & 0 \\ d & a_0 \end{smallmatrix}\right)\right)\, (\text{mod }p^n),\,a_i \in\Zp^\times, \, b,c,d \in\Zp\right\}\,,\] 
i.e., ${\mathcal{U}^Z_{(n)}}\subset \mathcal{U}_1(N)$ is the subgroup with conditions modulo $p^n$ transposed to those defining ${\mathcal{U}_{Z,(n)}}$. We therefore have a map
\[ \iota(\eta^n_p)_\star \circ  \iota_{Z,(n),\star}^u: \Zp[S_n] \longrightarrow \Zp[\mathbf{S}_{Z,(n)}]. \]

\begin{proposition}\label{prop_rewriting_Hsieh}
For every natural number $n$, we have 
$$\Delta_n  = \iota(\eta_p^n)_\star \circ \iota_{Z,(n),\star}^u\left(\mathbb{1}_{S_n}\right).$$     
\end{proposition}

\begin{proof}

 Recall that $\Zp[S_0(Np^n)] \xrightarrow{\nu_{Z,n}} \Zp[S_n]$ denotes the trace map induced by $$ S_0(Np^n) \lra \bb{Z}_p[S_n],\, \qquad [x] \mapsto \sum_{\pi_{Z,n}([y])=[x]} [y]\,.$$ Crucially, thanks to  Lemma \ref{lemma_cosetdecomposition}, $\nu_{Z,n}$ is explicitly given by $[x] \mapsto \sum_{(z,b)} [x \cdot \gamma_{z,b}]$. 
 
 Observe that 
 $$\iota(\gamma_{z,b}) \cdot u \cdot \iota(\eta_p^n)=\left(  \left( \begin{smallmatrix} z p^n & b \\  & 1 \end{smallmatrix} \right), \left( \begin{smallmatrix} z p^n &  z+b \\  & 1 \end{smallmatrix} \right),\left( \begin{smallmatrix} -b p^n  &  z \\ -p^n &  \end{smallmatrix} \right)\right)\,,$$ 
 and that the third component can be written as $\tau_{p^n}\left( \begin{smallmatrix} 1 &  \\ -b p^n & z \end{smallmatrix} \right)$. Then, we have the right-$\mathcal{U}_{Z,(n)}$ congruence 
 $$\iota(\gamma_{z,b}) \cdot u \cdot \iota(\eta_p^n)\,\mathcal{U}_{Z,(n)}=\left(  \left( \begin{smallmatrix} p^n & b \\  & 1 \end{smallmatrix} \right), \left( \begin{smallmatrix} p^n &  z+b \\  & 1 \end{smallmatrix} \right),\tau_{p^n}\left( \begin{smallmatrix} 1 &  \\  & z \end{smallmatrix} \right)\right)\,\mathcal{U}_{Z,(n)}\,.$$ 
 Therefore, by Lemma~\ref{lemma_cosetdecomposition}, 
 \begin{align*}
    \Delta_{n}  &= \sum_{[g] \in S_0(Np^n)} \iota(\eta_p^n)_\star \circ \iota_{Z,(n),\star}^u \circ \nu_{Z,n}[g] \\ &= \iota(\eta_p^n)_\star \circ \iota_{Z,(n),\star}^u\left(\sum_{[g] \in S_0(Np^n)}  \nu_{Z,n}[g]\right) \\    
    &=\iota(\eta_p^n)_\star \circ \iota_{Z,(n),\star}^u\left(\sum_{[g] \in S_n}[g]\right) \\ 
    &=\iota(\eta_p^n)_\star \circ \iota_{Z,(n),\star}^u\left(\mathbb{1}_{S_n}\right), 
\end{align*}
as claimed.
\end{proof}

\begin{corollary}\label{cor_rewriting_Hsieh}
  For every natural number $n$, we have
$$ {\rm pr}^{(n)}_{Z} \Theta_{\mathcal{U}_{(n)}} = \Delta_n\,,$$ 
where  ${\rm pr}^{(n)}_{Z} : \Zp[\mathbf{S}_{\mathcal{U}_{(n)}}] \to \Zp[\mathbf{S}_{{Z,(n)}}]$ denotes the natural projection.
\end{corollary}
\begin{proof}
Since $$u\, \mathcal{U}^{(n)} u^{-1} \cap \iota(\cO_{B,f}^\times) = u\, \mathcal{U}^Z_{(n)} u^{-1} \cap \iota(\cO_{B,f}^\times)  =U_{Z,{n}},$$
we have a commutative diagram 
\[  \xymatrix{  
  S_{n} \ar@{=}[d] \ar[rr]^-{\iota_{(n)}^u} & &  \mathbf{S}_{\mathcal{U}^{(n)} } \ar[d] \ar[rr]^-{\iota(\eta_p^n)}  & & \mathbf{S}_{\mathcal{U}_{(n)} } \ar[d]^-{{\rm pr}^{(n)}_{Z}} \\ 
S_{n} \ar[rr]_-{\iota_{Z,(n)}^u} & & \mathbf{S}_{\mathcal{U}^Z_{(n)} } \ar[rr]_-{\iota(\eta_p^n)} & & \mathbf{S}_{Z,(n)} .}\,  \] 
The result then follows from this, together with Proposition \ref{prop_rewriting_Hsieh}.
\end{proof}

\subsection{Balanced triple product \texorpdfstring{$p$-adic\,}\ \texorpdfstring{$L$-functions\,}\ }
\label{subsec_2025_10_20_1818}
We recall the definition of balanced $p$-adic triple-product $L$-functions from \cite[\S4.7]{hsiehpadicbalanced}, which are defined in terms of the theta element $\Delta^\dagger_\infty$. 

We let 
$$(\Bf,\Bg,\Bh)\in e \bS(N_1,\psi_1,\BI_1)\times e \bS(N_2,\psi_2,\BI_2)\times e \bS(N_3,\psi_3,\BI_3)$$ 
denote a triple of primitive Hida families of tame conductors $(N_1,N_2,N_3)$ and branch characters $(\psi_1,\psi_2,\psi_3)$, where $e=e_{\mathrm{U}_p}$ and each $\BI_i$ is a normal domain finite flat over the Iwasawa algebra $\cO\llbracket1+p\Zp\rrbracket$ (here $\cO$ is the ring of integers of a finite extension of $\Qp$). Let us put $\cR=\BI_1\hat{\otimes}_\cO\BI_2\hat{\otimes}_\cO\BI_3$ and assume that
\begin{itemize}
    \item  $\psi_1\psi_2\psi_3=\omega^{2a}$,\quad where $a\in\ZZ$,
    \item  ${\rm gcd}(N_1,N_2,N_3)$ is square-free.
    \end{itemize}

Let us put $N ={\rm lcm}(N_1,N_2,N_3)$. We denote by $\Sigma^-$ the set of primes $\ell \mid N$ where the local root number of the triple-product $L$-function attached to the specialisation of $(\Bf,\Bg,\Bh)$ at some (therefore,  by automorphic rigidity, any) arithmetic point equals $-1$. We further assume that 
\begin{itemize}
\item[\mylabel{item_JL}{$\mathbf{JL}$})] 
\begin{itemize}
    \item[$\bullet$]  $\#\Sigma^- $ is odd,
    \item[$\bullet$] $N^- := \prod_{\ell \in \Sigma^-} \ell$ and $N/N^-$ are relatively prime,
    \item[$\bullet$]  the residual Galois representation attached to each of the Hida families $\Bf,\Bg,$ and $\Bh$ is absolutely irreducible, $p$-distinguished, and it is ramified at every $\ell \in \Sigma^-$ with $\ell \equiv 1 $ ($\text{mod }p$).
      \end{itemize}
    \end{itemize}
Under these hypotheses, we consider the definite quaternion algebra $B$ over $\QQ$ with discriminant $N^-$ and consider the primitive Jacquet-Langlands lift to $B_E$ of $(\Bf,\Bg,\Bh)$  (cf. \cite{hsiehpadicbalanced}, Theorem 4.5):
$$(\Bf^B,\Bg^B,\Bh^B)\in e \bS^B(N_1,\psi_1,\BI_1)\times e \bS^B(N_2,\psi_2,\BI_2)\times e \bS^B(N_3,\psi_3,\BI_3).$$

\begin{defn}
\label{def_2025_11_02_0949}
    \item[i)] We define $\BF^B=\Bf^B\boxtimes\Bg^B\boxtimes\Bh^B:B_E^\times\backslash\widehat{B}_E^\times\rightarrow\cR$ as the triple product given by 
    \[
    \BF^B(x_1,x_2,x_3)=\Bf^B(x_1)\otimes\Bg^B(x_2)\otimes\Bh^B(x_3).
    \]
    \item[ii)] Let $\chi_\cR^*:\QQ^\times\backslash\widehat\QQ^\times\rightarrow\cR^\times$ be the character given by
    \[
    z\mapsto \omega^a(z)\langle\epsilon_\cyc(z)\rangle^{-3}\langle\epsilon_\cyc(z)\rangle_{\BI_1}^{1/2}\langle\epsilon_\cyc(z)\rangle_{\BI_2}^{1/2}\langle\epsilon_\cyc(z)\rangle_{\BI_3}^{1/2}.
    \]
    \item[iii)]We define $\BF^{B, \dagger}: B_E^\times\backslash\widehat{B}_E^\times\rightarrow\cR$ by the formula
    \[\BF^{B, \dagger}(x_1,x_2,x_3)=\BF^B(x_1,x_2,x_3)\chi_\cR^*(\nu(x_3)),\]
    where $\nu$ is the reduced norm on $B$. 
    \item[iv)] The theta element attached to $\BF^{B, \dagger}$ is defined by
    \[
    \Theta_{\BF^{B, \dagger}}:=\BF^{B, \dagger}(\Delta_\infty^\dagger)\in\cR.
    \]
\end{defn}
The theta element $\Theta_{\BF^{B, \dagger}}$ is Hsieh's balanced $p$-adic $L$-function. Its interpolative properties, which justify this nomenclature, are recorded in \cite[Corollary 4.13]{hsiehpadicbalanced}. It is worth recalling that the interpolation formula proven by Hsieh relates the square of the specialization of $\Theta_{\BF^{B, \dagger}}$ at each arithmetic point $(P_1,P_2,P_3)$ in the balanced range with the algebraic part of the critical central $L$-value attached to $(\Bf_{P_1},\Bg_{P_2},\Bh_{P_3})$.

\subsection{Signed balanced \texorpdfstring{$p$-adic\,}\ \texorpdfstring{$L$-function\,}\ }\label{subsec_signedpadic}
We retain the notation of \S \ref{subsec:familyofthetaelements}--\ref{subsec_2025_10_20_1818}. Let $f \in S_2(\Gamma_1(N_1)\cap U_1, \psi_1)$ be a $p$-old ordinary eigenform. We also let $g \in S_2(\Gamma_1(N_2),\psi_2
)$ be a new form with $\ord_p\,a_p(g)>0$. Finally, we let $\h\in e S(N_3,\psi_3,\mathbf{I})$ be a Hida family as in the previous section. We assume that $(\psi_1,\psi_2,\psi_3)$ and $(N_1,N_2,N_3)$ satisfy the conditions of \S\ref{subsec_2025_10_20_1818}. For the definite quaternion algebra $B$ as in loc. cit., we then have a primitive Jacquet--Langlands lift 
$$(f^B,g^B,\h^B)\in e S^B(N_1p,\psi_1)\times S^B(N_2,\psi_2) \times \mathbf{S}^B(N_3,\psi_3,\mathbf{I})$$
of the triple $(f,g,\h)$.

\subsubsection{} For a triple $\mathbf{n}=(n_1,n_2,n_3)$ of natural numbers, we define the element
$$\Theta_{{\mathbf{n}}}^{\ddagger}:=((\mathrm{U}_{p}')^{-n_1}e_{\mathrm{U}_{p}'},{\rm id},(\mathrm{U}_{p}')^{-n_3}e_{\mathrm{U}_{p}'})\,\Theta_{{\mathbf{n}}}\in \Zp[S_{U_{n_1}}]^{\rm ord}\otimes\Zp[S_{U_{n_2}}]\otimes\Zp[S_{U_{n_3}}]^{\rm ord}\,,$$ 
where $\Theta_{\bf n}$ is as in \eqref{eqn_2025_10_29_0931}, and for any $\Zp$-algebra $A$ and positive integer $m$ we have set $A[S_{U_{m}}]^{\rm ord}:=e_{\mathrm{U}_{p}'}\cdot A[S_{U_{m}}]$. We let 
$$\Theta_{n_1,0,n_3}(g)\in  \cO[S_{U_{n_1}}]^{\rm ord}\otimes_\cO\cO[S_{U_{n_3}}]^{\rm ord}$$ 
denote the image of $\Theta_{n_1,0,n_3}^{\ddagger}$
under the map induced from $S_{U_0}:=S_{U_{1}(N)}\xrightarrow{\,g^B\,} \cO$\,. We then define $\Theta_{n}^\ddagger(g)\in  \cO[S_{U_{1}}]^{\rm ord}\otimes_\cO\cO[S_{U_{n}}]^{\rm ord}$ as the image of $\Theta_{n,0,n}(g)$ under the natural morphism $({\rm pr}^{n,0,n}_{1,0,n})$. 

Finally, we let $\Theta_{n}(f,g)\in \cO[S_{U_n}]^{\rm ord}$ denote the image of $\Theta_{n}^\ddagger(g)$ under the map induced from $S_{U_1}\xrightarrow{\,f^B\,} \cO$\,. 

\subsubsection{} Let us consider the subgroup $U_{n-1}\cap U_{0,n} \supseteq U_n$ and let  $S_{U_n} \xrightarrow{\varpi_{n}} S_{U_{n-1}\cap U_{0,n}}$ be the corresponding map of Shimura sets. We denote by 
$$ \ZZ_p[S_{U_{n-1}\cap U_{0,n} }] \xrightarrow{\nu_{n}} \ZZ_p[S_{U_n}], \,\,\,\qquad [x] \mapsto \sum_{[y]\, :\, \varpi_{n}([y]) =[x] }[y]$$ 
the corresponding trace map. Note that 
$$U_{n-1}\cap U_{0,n}  / U_n \simeq  (1 + p^{n-1}\ZZ_p ) /  (1 + p^{n}\ZZ_p )\simeq \ZZ/p\ZZ.$$  
For each  $a \in (1 + p^{n-1}\ZZ_p ) /  (1 + p^{n}\ZZ_p )$, let us choose a lift (which we again denote by $a$) to $\ZZ_p^\times$. Then, a system of left coset representatives for $\left(U_{n-1}\cap U_{0,n} \right) / U_n $ is given by $$\left \{ \gamma_a =\left( \begin{smallmatrix}
    1 & \\ & a
\end{smallmatrix} \right)\,:\, a \in (1+p^{n-1}\ZZ_p)/(1+p^n \ZZ_p) \right\}.$$
\begin{remark}\label{rmk_trace_Hida}
    The trace map 
$$\ZZ_p[S_{U_{n-1}\cap U_{0,n} }] \xrightarrow{\nu_{n}} \ZZ_p[S_{U_n}]$$ can be written as 
$$[x] \mapsto \sum_{a \in (1+p^{n-1}\ZZ_p)/(1+p^n \ZZ_p)} [x \cdot \gamma_a] = \sum_{a \in (1+p^{n-1}\ZZ_p)/(1+p^n \ZZ_p)} \langle a \rangle \cdot [x],$$
where $\langle a \rangle$ is the diamond operator given as in \S
\ref{subsubsec_A21_2025_06_30_1533}. 
\end{remark}

\begin{lemma}\label{lemma:Hida}
    Let $n\geq 2$ be an integer. Then the natural map $\pi_{n,n-1}:\mathbb{Z}_p[S_{U_{n-1}\cap U_{0,n}}]\rightarrow \mathbb{Z}_p[S_{U_{n-1}}]$ induces an isomorphism
    \[
    \mathbb{Z}_p[S_{U_{n-1}\cap U_{0,n}}]^{\ord}\xrightarrow{\,\,\sim\,\,} \mathbb{Z}_p[S_{U_{n-1}}]^\ord.
    \]
\end{lemma}
\begin{proof}
    The proof is a standard argument in Hida theory, see for instance \cite{EmertonHida}. We give the details for the convenience of the reader.
    The map $\pi_{n,n-1}$ is surjective and its kernel is generated as a $\mathbb{Z}_p$-module by the elements of the form
    \[
    [xg]-[x]\quad \text{with } [x]\in S_{U_{n-1}\cap U_{0,n}} \text{ and } g=\begin{pmatrix} a & b \\ c & d \end{pmatrix}\in \GL_2(\bb{Z}_p) \text{ such that } p^{n-1}\mid \gcd(c, a-d)\,.
    \]
    It therefore suffices to show that the operator $\mathrm{U}_p'$ annihilates these elements.

    To that end, let $[x]$ and $g$ be as above. Then
    \begin{equation}\label{eq:202510221329}
    \mathrm{U}_p'([xg]-[x])=\sum_{j=0}^{p-1} \left(\left[xg\begin{pmatrix} p & j \\ 0 & 1 \end{pmatrix}\right]-\left[x\begin{pmatrix} p & j \\ 0 & 1 \end{pmatrix}\right]\right),
    \end{equation}
    where the matrices $\left(\begin{smallmatrix} p & j \\ 0 & 1 \end{smallmatrix}\right)$ are regarded as elements in $\GL_2(\bb{Q}_p)$. Let $k\in\lbrace 0,1,\ldots, p-1\rbrace$ be the unique element with $k\equiv a^{-1}b\pmod{p}$. Then an easy computation shows that
    \[
    \begin{pmatrix} p & j \\ 0 & 1 \end{pmatrix}^{-1} \begin{pmatrix} a & b \\ c & d \end{pmatrix}\begin{pmatrix} p & j+k \\ 0 & 1 \end{pmatrix}\in U_{n-1}\cap U_{0,n} \qquad \forall\, j\in \lbrace 0,1,\ldots,p-1\rbrace.
    \]
    It follows that
    \[
    \left[xg\begin{pmatrix} p & j+k \\ 0 & 1 \end{pmatrix}\right]-\left[x\begin{pmatrix} p & j \\ 0 & 1 \end{pmatrix}\right]=0\qquad \forall\,  j\in \lbrace 0,1,\ldots,p-1\rbrace
    \]
    as elements in $\mathbb{Z}_p[S_{U_{n-1}\cap U_{0,n}}]$. Combined with equation~(\ref{eq:202510221329}), this yields
    \[
    \mathrm{U}_p'([xg]-[x])=0,
    \]
    as desired.
\end{proof}

For each integer $n\geq 2$, we denote by
\[
\res_{n-1}^{n}:\bb{Z}_p[S_{U_{n-1}}]^\ord\longrightarrow \bb{Z}_p[S_{U_n}]^{\ord}
\]
the composition of the inverse of the isomorphism $\bb{Z}_p[S_{U_{n-1}\cap U_{0,n}}]^\ord\xrightarrow{\,\sim\,} \bb{Z}_p[S_{U_{n-1}}]^\ord$ in Lemma~\ref{lemma:Hida} with the trace map $\nu_n  : \bb{Z}_p[S_{U_{n-1}\cap U_{0,n}}]^\ord\rightarrow \bb{Z}_p[S_{U_{n}}]^\ord$.

The following $Q$-system relation parallels that of the diagonal cycles (cf. Propositions~\ref{prop_2025_03_08_1713} and \ref{prop:2025_08_11_1553}).

\begin{proposition}
    \label{prop_2025_10_15_1640}
    For any integer $n\geq 2$, we have
    $${\rm pr}_{n}^{n+1}\,\Theta_{n+1}(f,g)=a_p(g)\cdot\Theta_{n}(f,g)-\chi_g(p)\cdot \res_{n-1}^{n}\Theta_{n-1}(f,g)\,.$$
\end{proposition}

\begin{proof}
The proof of this proposition is similar to that of Proposition~\ref{prop_thm_2025_07_03_1627}, and we indicate the key steps here for the convenience of the reader. As in Proposition~\ref{prop:2025_08_11_1553}, we have
\begin{equation}
    \label{eqn_2025_10_20_1436}
 {\rm pr}_{n}^{n+1}\,\Theta_{n+1}(f,g)=a_p(g)\cdot\Theta_{n}(f,g)-\chi_g(p)\cdot\widetilde\Theta_{n}(f,g)\,,   
\end{equation}
\begin{equation}
    \label{eqn_2025_10_20_1453}
 {\rm pr}_{n}^{n+1}\,\widetilde\Theta_{n+1}(f,g)=p\cdot\Theta_{n}(f,g)\,.
\end{equation}
Here, $\widetilde\Theta_{n}(f,g) \in  \cO[S_{U_n}]^{\rm ord}$ is defined (paralleling our discussion in \S\ref{subsubsec_2025_10_20_1513}) as the image of $\Theta_{(n,1,n)}^{\ddagger}$ under the compositum of the following maps: 
\begin{align*}
    \Zp[S_{U_n}]^{\rm ord}\otimes \Zp[S_{U_1}] \otimes \Zp[S_{U_n}]^{\rm ord}\xrightarrow{({\rm pr}_{1}^n,\pi_{2,*},{\rm id})} \Zp[S_{U_1}]^{\rm ord}\otimes \Zp[S_{U_0}] \otimes \Zp[S_{U_n}]^{\rm ord}\xrightarrow{(f,g,{\rm id})} \cO[S_{U_n}]^{\rm ord}\,.
\end{align*}
The argument of Proposition~\ref{prop_2025_03_08_1713} applies formally to show that 
\begin{equation}
    \label{eqn_2025_10_20_1437}
    \langle 1,1,d\rangle\,\widetilde\Theta_{n}(f,g)= \widetilde\Theta_{n}(f,g)
\end{equation}
for any integer $d\equiv 1 \pmod{p^{n-1}}$. Note that, thanks to Remark \ref{rmk_trace_Hida}, \eqref{eqn_2025_10_20_1437} implies that $\widetilde\Theta_{n}(f,g)$ belongs to the image of the trace map $\nu_n$. Then, arguing as in the proof of Proposition~\ref{prop:2025_08_11_1554} (and relying on Lemma~\ref{lemma:Hida}), we deduce using \eqref{eqn_2025_10_20_1453} that 
\begin{equation}
    \label{eqn_2025_10_20_1450}
    \widetilde\Theta_{n}(f,g) = \res_{n-1}^{n}\Theta_{n-1}(f,g)\,.
\end{equation}
The proof of our proposition now follows by combining \eqref{eqn_2025_10_20_1436} and \eqref{eqn_2025_10_20_1450}.
\end{proof}

\subsubsection{} As in the main body of our article, let us denote by $\alpha$ and $\beta$ the roots of the Hecke polynomial $X^2-a_p(g)X+\chi_g(p)p$ of $g$ at $p$. As in Definition~\ref{def_2025_10_15_1649}, we define the $p$-stabilized class 
$$\Theta_{n,\xi}(f,g)=\frac{1}{\xi^{n}}\left(\Theta_{n}(f,g)-\frac{\chi_g(p)}{\xi}\res_{n-1}^{n} \Theta_{n-1}(f,g)\right) \,\in\, L[S_{U_n}]^{\ord}\,,\qquad \xi=\alpha,\beta\,,\quad n\in \ZZ_{\geq 2}\,.$$ 

Denote \[ U_\infty :=\{ g \in U_1(N) \,:\, \Psi_p(g_p) = \left( \begin{smallmatrix} a & b  \\  & a \end{smallmatrix}\right),\,\,a \in\Zp^\times, b \in \ZZ_p\}. \] 

\noindent We then have the following ``definite'' version of Theorem~\ref{thm:non-integral-classes}, which is also proved identically:
\begin{theorem}
        \label{thm_2025_10_15_2117}
Let $\xi\in\{\alpha,\beta\}$ be such that $\ord_p(\xi)<1$. 
\item[i)] There exists a unique element 
$$\Theta_{\infty,\xi}(f,g)\in \mathscr{L}\otimes_\LL\varprojlim_n \cO[S_{U_n}]^{\ord}=\mathscr{L}\otimes_\LL \cO[[B^\times\backslash \widehat{B}^\times / U_\infty]]^{\ord}$$ 
such that its image under the natural morphism 
$ \mathscr{L}\otimes_{\LL}\varprojlim_n \cO[S_{U_n}]^{\ord} \lra L[S_{U_n}]^{\ord}$ coincides with $\Theta_{n,\xi}(f,g)$. 
\item[ii)] Let us put $\Theta_{\infty,\xi}(f,g,\h):=\h^{B, \dagger}\left(\Theta_{\infty,\xi}(f,g)\right)\in \mathscr{L}\otimes_\LL \mathbf{I}$. Then, $\Theta_{\infty,\xi}(f,g,\h) \in \mathcal{D}_{\ord_p(\xi)}$\,.
\end{theorem}
The element $\Theta_{\infty,\xi}(f,g,\h)$ is the trito-non-ordinary version of Hsieh's balanced $p$-adic $L$-function. Thanks to Hsieh's calculations of trilinear periods (see \cite[Propositions 4.9 \& 4.10, Corollary 4.13]{hsiehpadicbalanced}), it enjoys identical interpolative properties. Note that we allow variation of only the third factor; as a result, the said interpolation property concerns only specialisations of weight $2$.

\subsubsection{} We are now ready to introduce signed balanced $p$-adic $L$-functions, via the following ``definite'' analogue of Theorem~\ref{thm:decompo} (which is proved in an identical manner):

\begin{theorem}
    \label{thm_2025_10_16_0900}
     There exist unique $\Theta_{\infty}^\sharp(f,g,\h)$, $\Theta_{\infty}^\flat(f,g,\h) \in \mathbf{I}$ such that
    \[
    \begin{bmatrix}
      \Theta_{\infty,\alpha}(f,g,\h)\\  
      \Theta_{\infty,\beta}(f,g,\h)\end{bmatrix}=Q_g^{-1}M_{\log,g}
\begin{bmatrix}
      \Theta_{\infty}^\sharp(f,g,\h)\\  
      \Theta_{\infty}^\flat(f,g,\h)
\end{bmatrix}    .
    \]
\end{theorem}
\begin{proof}
    This follows from Proposition~\ref{prop_2025_10_15_1640} and \cite[Theorem~3.4]{BBL1}.
\end{proof}

\bibliographystyle{amsalpha}
\bibliography{non-ord_diagonal}

@article {ACR,
    AUTHOR = {Alonso, Ra\'ul and Castella, Francesc and Rivero, \'Oscar},
     TITLE = {The diagonal cycle {E}uler system for {$\rm
              GL_2\times GL_2$}},
   JOURNAL = {J. Inst. Math. Jussieu},
  FJOURNAL = {Journal of the Institute of Mathematics of Jussieu. JIMJ.
              Journal de l'Institut de Math\'ematiques de Jussieu},
    VOLUME = {24},
      YEAR = {2025},
    NUMBER = {5},
     PAGES = {1591--1653},
   
}

@article {ACR25,
    AUTHOR = {Alonso, Ra\'ul and Castella, Francesc and Rivero, \'Oscar},
     TITLE = {An anticyclotomic {E}uler system for adjoint modular {G}alois
              representations},
   JOURNAL = {Ann. Inst. Fourier (Grenoble)},
  FJOURNAL = {Universit\'e{} de Grenoble. Annales de l'Institut Fourier},
    VOLUME = {75},
      YEAR = {2025},
    NUMBER = {1},
     PAGES = {291--329},     
}

@article {GreenbergSeveso,
    AUTHOR = {Greenberg, Matthew and Seveso, Marco Adamo},
     TITLE = {Triple product {$p$}-adic {$L$}-functions for balanced
              weights},
   JOURNAL = {Math. Ann.},
  FJOURNAL = {Mathematische Annalen},
    VOLUME = {376},
      YEAR = {2020},
    NUMBER = {1-2},
     PAGES = {103--176},
      ISSN = {0025-5831,1432-1807},
   MRCLASS = {11F67},
  MRNUMBER = {4055157},
MRREVIEWER = {Chris\ Williams},
       DOI = {10.1007/s00208-019-01865-w},
       URL = {https://doi.org/10.1007/s00208-019-01865-w},
}

@article {EmertonHida,
    AUTHOR = {Emerton, Matthew},
     TITLE = {A new proof of a theorem of {H}ida},
   JOURNAL = {Internat. Math. Res. Notices},
  FJOURNAL = {International Mathematics Research Notices},
      YEAR = {1999},
    NUMBER = {9},
     PAGES = {453--472},
 
}

@article {BSV,
    AUTHOR = {Bertolini, Massimo and
              Seveso, Marco Adamo and Venerucci, Rodolfo},
     TITLE = {Reciprocity laws for balanced diagonal classes},
   JOURNAL = {Ast\'{e}risque},
  FJOURNAL = {Ast\'{e}risque},
    NUMBER = {434},
      YEAR = {2022},
     PAGES = {77--174},
     
}

@article {DarmonRotger,
    AUTHOR = {Darmon, Henri and Rotger, Victor},
     TITLE = {{$p$}-{A}dic families of diagonal cycles},
   JOURNAL = {Ast\'{e}risque},
  FJOURNAL = {Ast\'{e}risque},
    NUMBER = {434},
      YEAR = {2022},
     PAGES = {29--75},
     
}

@article {EPW,
    AUTHOR = {Emerton, Matthew and Pollack, Robert and Weston, Tom},
     TITLE = {Variation of {I}wasawa invariants in {H}ida families},
   JOURNAL = {Invent. Math.},
  FJOURNAL = {Inventiones Mathematicae},
    VOLUME = {163},
      YEAR = {2006},
    NUMBER = {3},
     PAGES = {523--580},
     
}

@article {GrossSchoen-cycle,
    AUTHOR = {Gross, B. H. and Schoen, C.},
     TITLE = {The modified diagonal cycle on the triple product of a pointed
              curve},
   JOURNAL = {Ann. Inst. Fourier (Grenoble)},
  FJOURNAL = {Universit\'{e} de Grenoble. Annales de l'Institut Fourier},
    VOLUME = {45},
      YEAR = {1995},
    NUMBER = {3},
     PAGES = {649--679},
    
}

@book {Hida_Elementary,
    AUTHOR = {Hida, Haruzo},
     TITLE = {Elementary theory of {$L$}-functions and {E}isenstein series},
    SERIES = {London Mathematical Society Student Texts},
    VOLUME = {26},
 PUBLISHER = {Cambridge University Press, Cambridge},
      YEAR = {1993},
     PAGES = {xii+386},
    
}

@article {howard06,
    AUTHOR = {Howard, Benjamin},
     TITLE = {Bipartite {E}uler systems},
   JOURNAL = {J. Reine Angew. Math.},
  FJOURNAL = {Journal f\"{u}r die Reine und Angewandte Mathematik. [Crelle's
              Journal]},
    VOLUME = {597},
      YEAR = {2006},
     PAGES = {1--25},
}

@article{hsiehpadicbalanced,
  title={Hida families and {$p$}-adic triple product {$L$}-functions},
  author={Hsieh, Ming-Lun},
  journal={American Journal of Mathematics},
  volume={143},
  number={2},
  pages={411--532},
  year={2021},
  publisher={Johns Hopkins University Press}
}

@article {JNS,
    AUTHOR = {Jetchev, Dimitar and Nekov{\'a}{\v{r}}, Jan and Skinner, Christopher},
    note = {preprint},
    title = {Split {E}uler systems for conjugate-dual {G}alois representations},
    year = {2024},
}

@article{LoefflerSpherical,
	Author = {Loeffler, David},
	Journal = {J. Th. Nombres Bordeaux (Iwasawa 2019 special issue)},
	Title = {Spherical varieties and norm relations in {I}wasawa theory},
	Year = {2022},
}

@misc{stacks-project,
  author       = {The {Stacks project authors}},
  title        = {The Stacks project},
  howpublished = {\url{https://stacks.math.columbia.edu}},
  year         = {2023},
}

@article {NN16,
    AUTHOR = {Nekov\'a{\v r}, Jan and Nizio\l, Wies{\l}awa},
     TITLE = {Syntomic cohomology and {$p$}-adic regulators for varieties
              over {$p$}-adic fields},
      NOTE = {With appendices by Laurent Berger and Fr\'ed\'eric D\'eglise},
   JOURNAL = {Algebra Number Theory},
  FJOURNAL = {Algebra \& Number Theory},
    VOLUME = {10},
      YEAR = {2016},
    NUMBER = {8},
     PAGES = {1695--1790},
     
}

@MISC{CCSS,
Author={Castella, Francesc and Ciperiani, Mirela and Skinner, Christopher and Sprung, Florian},
Title={The {I}wasawa main conjectures for modular forms at non-ordinary primes},
note={preprint},
year={$\geq$2018},
}

@article {BBL1,
    AUTHOR = {Burungale, Ashay and B\"uy\"ukboduk, K\^az{\i}m and Lei,
              Antonio},
     TITLE = {Anticyclotomic {I}wasawa theory of abelian varieties of {$\rm
              GL_2$}-type at non-ordinary primes},
   JOURNAL = {Adv. Math.},
  FJOURNAL = {Advances in Mathematics},
    VOLUME = {439},
      YEAR = {2024},
     PAGES = {Paper No. 109465, 63},
   
}

@misc{BBL2,
      title={Anticyclotomic {I}wasawa theory of abelian varieties of {$\rm GL_2$}-type at non-ordinary primes \textup{II}}, 
AUTHOR = {Burungale, Ashay and B\"uy\"ukboduk, K\^az{\i}m and Lei,
              Antonio},
year={$\geq$2025},
      eprint={2310.06813},
      archivePrefix={arXiv},
      primaryClass={math.NT},
      url={https://arxiv.org/abs/2310.06813}, 
note={preprint, \href{https://arxiv.org/abs/2310.06813}{arXiv: 2310.06813}}
}

@article {BFSuper,
    AUTHOR = {B\"uy\"ukboduk, K\^az{\i}m and Lei, Antonio},
     TITLE = {Iwasawa theory of elliptic modular forms over imaginary quadratic fields at non-ordinary primes},
     journal =  {Int. Math. Res. Not. IMRN},
  year = {2021},
  number = {14},
  pages={10654--10730}
}

@article {BL-MSMF,
    AUTHOR = {B\"uy\"ukboduk, K\^az{\i}m and Lei, Antonio},
     TITLE = {Iwasawa theory of twists of elliptic modular forms over
              imaginary quadratic fields at inert primes},
   JOURNAL = {M\'em. Soc. Math. Fr. (N.S.)},
  FJOURNAL = {M\'emoires de la Soci\'et\'e{} Math\'ematique de France.
              Nouvelle S\'erie},
    NUMBER = {185},
      YEAR = {2025},
     PAGES = {vi+119},
}

@article {PR-Heegner-IMC,
    AUTHOR = {Perrin-Riou, Bernadette},
     TITLE = {Fonctions {$L$} {$p$}-adiques, th\'eorie d'{I}wasawa et points
              de {H}eegner},
   JOURNAL = {Bull. Soc. Math. France},
  FJOURNAL = {Bulletin de la Soci\'et\'e{} Math\'ematique de France},
    VOLUME = {115},
      YEAR = {1987},
    NUMBER = {4},
     PAGES = {399--456},
      
}

@article {castellawan,
    AUTHOR = {Castella, Francesc and Wan, Xin},
     TITLE = {Perrin-{R}iou's main conjecture for elliptic curves at
              supersingular primes},
   JOURNAL = {Math. Ann.},
  FJOURNAL = {Mathematische Annalen},
    VOLUME = {389},
      YEAR = {2024},
    NUMBER = {3},
     PAGES = {2595--2636},
      
}

@article {Colmez2010Ast,
    AUTHOR = {Colmez, Pierre},
     TITLE = {Fonctions d'une variable {$p$}-adique},
   JOURNAL = {Ast\'{e}risque},
  FJOURNAL = {Ast\'{e}risque},
    NUMBER = {330},
      YEAR = {2010},
     PAGES = {13--59},
      
}

@article {grosskudla,
    AUTHOR = {Gross, Benedict H. and Kudla, Stephen S.},
     TITLE = {Heights and the central critical values of triple product
              {$L$}-functions},
   JOURNAL = {Compositio Math.},
  FJOURNAL = {Compositio Mathematica},
    VOLUME = {81},
      YEAR = {1992},
    NUMBER = {2},
     PAGES = {143--209},
     
}

@article {BertoliniDarmon2005,
    AUTHOR = {Bertolini, Massimo and Darmon, Henri},
     TITLE = {Iwasawa's main conjecture for elliptic curves over
              anticyclotomic {$\Bbb Z_p$}-extensions},
   JOURNAL = {Ann. of Math. (2)},
  FJOURNAL = {Annals of Mathematics. Second Series},
    VOLUME = {162},
      YEAR = {2005},
    NUMBER = {1},
     PAGES = {1--64},
      
}

@book {asterisque,
    AUTHOR = {Bertolini, Massimo and Darmon, Henri and Rotger, Victor and
              Seveso, Marco Adamo and Venerucci, Rodolfo},
     TITLE = {Heegner points, {S}tark-{H}eegner points, and diagonal
              classes},
      NOTE = {Ast\'erisque No. 434 (2022)},
 PUBLISHER = {Soci\'et\'e{} Math\'ematique de France, Paris},
      YEAR = {2022},
     PAGES = {xviii+201},
 }

@misc {YZZ,
    AUTHOR = {Yuan, Xinyi and Zhang, Shou-Wu and Zhang, Wei},
Title={{Triple product $L$-series and Gross–Kudla–Schoen cycles}},
Note={preprint},
Year={2023},
}

@article {LZ0,
    AUTHOR = {Loeffler, David and Zerbes, Sarah Livia},
     TITLE = {Iwasawa theory and {$p$}-adic {$L$}-functions over
              {$\Bbb{Z}_p^2$}-extensions},
   JOURNAL = {Int. J. Number Theory},
  FJOURNAL = {International Journal of Number Theory},
    VOLUME = {10},
      YEAR = {2014},
    NUMBER = {8},
     PAGES = {2045--2095},
      
}

@article {LLZ0.5,
    AUTHOR = {Lei, Antonio and Loeffler, David and Zerbes, Sarah Livia},
     TITLE = {Coleman maps and the {$p$}-adic regulator},
   JOURNAL = {Algebra Number Theory},
  FJOURNAL = {Algebra \& Number Theory},
    VOLUME = {5},
      YEAR = {2011},
    NUMBER = {8},
     PAGES = {1095--1131},
      
}

@article {LZ1,
    AUTHOR = {Loeffler, David and Zerbes, Sarah Livia},
     TITLE = {Rankin-{E}isenstein classes in {C}oleman families},
   JOURNAL = {Res. Math. Sci.},
  FJOURNAL = {Research in the Mathematical Sciences},
    VOLUME = {3},
      YEAR = {2016},
     PAGES = {Paper No. 29, 1--53},
     
}

@article {PR,
    AUTHOR = {Perrin-Riou, Bernadette},
     TITLE = {Fonctions {$L$} {$p$}-adiques associ\'ees \`a une forme
              modulaire et \`a un corps quadratique imaginaire},
   JOURNAL = {J. London Math. Soc. (2)},
  FJOURNAL = {Journal of the London Mathematical Society. Second Series},
    VOLUME = {38},
      YEAR = {1988},
    NUMBER = {1},
     PAGES = {1--32},
     
}

@InCollection{LRZ,
  title={Spherical varieties and {$p$}-adic families of cohomology classes},
  author={Loeffler, David and Rockwood, Robert and Zerbes, Sarah},
    BookTitle = {Elliptic curves and modular forms in arithmetic geometry{--}celebrating {M}assimo {B}ertolini's {60th} birthday},
    Year = {2024},
   
}

@article {KLZ2,
    AUTHOR = {Kings, Guido and Loeffler, David and Zerbes, Sarah Livia},
     TITLE = {Rankin-{E}isenstein classes and explicit reciprocity laws},
   JOURNAL = {Camb. J. Math.},
  FJOURNAL = {Cambridge Journal of Mathematics},
    VOLUME = {5},
      YEAR = {2017},
    NUMBER = {1},
     PAGES = {1--122},
      
}

@article{howard2007,
author="Howard, Benjamin",
title="{Variation of Heegner points in Hida families.}",
language="English",
journal="Invent. Math.",
volume="167",
number="1",
pages="91-128",
year="2007",
}

@article {LLZ2,
    AUTHOR = {Lei, Antonio and Loeffler, David and Zerbes, Sarah Livia},
     TITLE = {Euler systems for modular forms over imaginary quadratic
              fields},
   JOURNAL = {Compos. Math.},
  FJOURNAL = {Compositio Mathematica},
    VOLUME = {151},
      YEAR = {2015},
    NUMBER = {9},
     PAGES = {1585--1625},
     
}

@article {Fouquet2013,
    AUTHOR = {Fouquet, Olivier},
     TITLE = {Dihedral {I}wasawa theory of nearly ordinary quaternionic
              automorphic forms},
   JOURNAL = {Compos. Math.},
  FJOURNAL = {Compositio Mathematica},
    VOLUME = {149},
      YEAR = {2013},
    NUMBER = {3},
     PAGES = {356--416},
     
}

@article {Saito97,
    AUTHOR = {Saito, Takeshi},
     TITLE = {Modular forms and {$p$}-adic {H}odge theory},
   JOURNAL = {Invent. Math.},
  FJOURNAL = {Inventiones Mathematicae},
    VOLUME = {129},
      YEAR = {1997},
    NUMBER = {3},
     PAGES = {607--620},
     
}

@article {buyukboduklei0,
    AUTHOR = {B\"uy\"ukboduk, K\^az{\i}m and Lei, Antonio},
     TITLE = {Integral {I}wasawa theory of {G}alois representations for
              non-ordinary primes},
   JOURNAL = {Math. Z.},
  FJOURNAL = {Mathematische Zeitschrift},
    VOLUME = {286},
      YEAR = {2017},
    NUMBER = {1-2},
     PAGES = {361--398},
    
}

@INCOLLECTION{blochkato,
author={Bloch, Spencer and Kato, Kazuya},
title={{$L$}-functions and {T}amagawa numbers of motives},
booktitle={The {G}rothendieck {F}estschrift, {V}ol.\ {I}},
pages={333--400}, 
series={Progr. Math.}, 
volume={86}, 
publisher={Birkh\"auser Boston}, 
address={Boston, MA},
year={1990},}

@article {colmez98,
    AUTHOR = {Colmez, Pierre},
     TITLE = {Th\'eorie d'{I}wasawa des repr\'esentations de de {R}ham d'un
              corps local},
   JOURNAL = {Ann. of Math. (2)},
  FJOURNAL = {Annals of Mathematics. Second Series},
    VOLUME = {148},
      YEAR = {1998},
    NUMBER = {2},
     PAGES = {485--571},
     
}

@article {darmonrotger14,
    AUTHOR = {Darmon, Henri and Rotger, Victor},
     TITLE = {Diagonal cycles and {E}uler systems {I}: {A} {$p$}-adic
              {G}ross-{Z}agier formula},
   JOURNAL = {Ann. Sci. \'Ec. Norm. Sup\'er. (4)},
  FJOURNAL = {Annales Scientifiques de l'\'Ecole Normale Sup\'erieure.
              Quatri\`eme S\'erie},
    VOLUME = {47},
      YEAR = {2014},
    NUMBER = {4},
     PAGES = {779--832},
     
}

@article {darmonrotger17,
    AUTHOR = {Darmon, Henri and Rotger, Victor},
     TITLE = {Diagonal cycles and {E}uler systems {II}: {T}he {B}irch and
              {S}winnerton-{D}yer conjecture for {H}asse-{W}eil-{A}rtin
              {$L$}-functions},
   JOURNAL = {J. Amer. Math. Soc.},
  FJOURNAL = {Journal of the American Mathematical Society},
    VOLUME = {30},
      YEAR = {2017},
    NUMBER = {3},
     PAGES = {601--672},
 }

@ARTICLE{pollack03,
  author = {Pollack, Robert},
  title = {On the {$p$}-adic {$L$}-function of a modular form at a supersingular
	prime},
  journal = {Duke Math. J.},
  year = {2003},
  volume = {118},
  pages = {523--558},
  number = {3},
  fjournal = {Duke Mathematical Journal}
}

@article {sprung09,
    AUTHOR = {Sprung, Florian E. Ito},
     TITLE = {Iwasawa theory for elliptic curves at supersingular primes: a
              pair of main conjectures},
   JOURNAL = {J. Number Theory},
  FJOURNAL = {Journal of Number Theory},
    VOLUME = {132},
      YEAR = {2012},
    NUMBER = {7},
     PAGES = {1483--1506},
    
}

@article {sprung16,
    AUTHOR = {Sprung, Florian Ito},
     TITLE = {On {I}wasawa main conjectures for elliptic curves at
              supersingular primes: beyond the case {$a_ p = 0$}},
   JOURNAL = {Adv. Math.},
  FJOURNAL = {Advances in Mathematics},
    VOLUME = {449},
      YEAR = {2024},
     PAGES = {Paper No. 109741, 47},
     
}

@InCollection{nekovarbanff,
 Author = {Nekov{\'a}{\v{r}}, Jan},
 Title = {{{\(p\)}}-adic {A}bel-{J}acobi maps and {{\(p\)}}-adic heights},
 BookTitle = {The arithmetic and geometry of algebraic cycles. Proceedings of the CRM summer school, Banff, Alberta, Canada, June 7--19, 1998. Vol. 2},
 ISBN = {0-8218-1954-2},
 Pages = {367--379},
 Year = {2000},
 Publisher = {Providence, RI: American Mathematical Society (AMS)},
 Language = {English},
 zbMATH = {1444660},
 Zbl = {0983.14009}
}
\end{document}